\documentclass[reqno]{amsart}

\usepackage{graphicx,pdfsync,amssymb,mathrsfs,amsfonts,amsbsy,color,esint,mathtools,tikz,mdframed}
\usetikzlibrary{positioning}
\usepackage[shortlabels]{enumitem}
\usepackage{makecell}

\usepackage{booktabs,array,tabularx}

\mathtoolsset{showonlyrefs}

\usetikzlibrary{arrows.meta,positioning,calc}

\usepackage{booktabs}
\usepackage{array}
\usepackage[table]{xcolor}
\usepackage{hyperref}
\usepackage{amsfonts} 

\hypersetup{
colorlinks=true,
linkcolor=blue,
filecolor=blue,
urlcolor=blue,
citecolor=blue,
}

\usepackage[utf8]{inputenc}
\usepackage[T1]{fontenc}

\DeclareMathOperator{\Supp }{supp}

\newtheorem{theorem}{Theorem}[section]
\newtheorem{lemma}[theorem]{Lemma}
\newtheorem{proposition}[theorem]{Proposition}

\newtheorem{remark}[theorem]{Remark}

\setlist[enumerate]{itemsep=3pt}
\setlist[itemize]{itemsep=3pt}

\def \R {\mathbb{R}}  
\def \N {\mathbb{N}}  

\def \p {\partial}

\def \ep {\varepsilon}
\def \om {\omega}
\def \Om {\Omega}

\def \H {\mathcal{H}}

\numberwithin{equation}{section}

\begin{document}

\title[Inviscid damping without derivatives]{Inviscid damping without derivatives near Couette flow}

\author{Dengjun Guo}

\address{Academy of Mathematics and Systems Science, Chinese Academy of Sciences, Beijing, China}

\email{djguo@amss.ac.cn}

\author{Xiaoyutao Luo}

\address{Academy of Mathematics and Systems Science, Chinese Academy of Sciences, Beijing, China}

\email{xiaoyutao.luo@amss.ac.cn}

\subjclass[2020]{35Q31, 35B40}
 
\keywords{Inviscid Damping, Couette Flow, 2D Euler equations}
\date{\today}

\begin{abstract}
    We study the long-time dynamics near Couette flow for the 2D Euler equations in unbounded geometries. We identify a mechanism of nonlinear inviscid damping at Yudovich regularity, distinct from classical phase mixing: velocity decay driven by spatial evacuation of vorticity.

The mechanism leads to a geometry–sign classification of the dynamics. In the infinite channel, small nonnegative bounded-vorticity perturbations undergo global damping without derivative assumptions. 
In the whole plane, small nonnegative perturbations exhibit enhanced dispersion and velocity damping, while non-positive perturbations remain confined and do not damp, even when arbitrarily small in Gevrey classes. Moreover, damping can coexist with infinite-time growth of vorticity derivatives, even under arbitrary shear modulation. The classification is obtained through an interplay between new Lyapunov functionals and Hamiltonian conservation.
\end{abstract}
\maketitle

\section{Introduction}\label{sec:intro}

 The inviscid damping problem asks whether velocity perturbations near a shear flow of an ideal fluid can decay in time despite the absence of viscosity or any other dissipative mechanism. Phase mixing was proposed as an underlying mechanism of inviscid damping, whose origins go back to the classical works of Kelvin, Rayleigh, and Orr~\cite{Kelvin1887,Rayleigh1879,Orr1907}. A background shear transports vorticity to progressively finer spatial scales. Since the velocity is one derivative smoother than the vorticity through the Biot--Savart law, the transfer toward high frequencies produces decay of the velocity field. In this sense, inviscid damping is the fluid analogue of Landau damping.

The nonlinear phase-mixing theory, initiated in the breakthrough work of Bedrossian and Masmoudi~\cite{MR3415068} and further developed in~\cite{MR4076093,MR4740211,MR4628607}, is formulated in streamwise-periodic geometries and requires regularity above the Gevrey-$2$ threshold. This regularity requirement is essentially sharp within that framework. Below the Gevrey-$2$ scale, nonlinear instability occurs~\cite{MR4630602}; in low Sobolev regularity, nontrivial steady states~\cite{MR2796139} and traveling waves~\cite{CastroLear23Travellingwaves} exist arbitrarily close to Couette flow, and only finite-time damping can hold in general~\cite{2511.21583}. Thus, in periodic geometry, nonlinear inviscid damping from the Gevrey theory cannot   be extended  to the Yudovich or low-Sobolev regime.

These results raise a more fundamental question. Is the need for derivatives intrinsic to inviscid damping itself, or only to phase mixing? In a periodic domain, vorticity cannot escape spatially, so relaxation must be generated through fine-scale mixing. In an unbounded domain, this restriction disappears: the shear can transport vorticity over arbitrarily large distances, allowing the vorticity to disperse. Such spatial evacuation can weaken the induced velocity without relying on a transfer to fine scales and therefore need not obey the same regularity threshold as phase mixing.

The purpose of this paper is to establish this new mechanism of nonlinear inviscid damping near the simplest prototypical shear flow: the planar Couette flow $(y, 0)$. We show that spatial evacuation of vorticity produces velocity damping at Yudovich regularity, with no derivative assumptions on the initial perturbation. This mechanism is distinct from the classical phase mixing: the decay is driven by the dispersion of vorticity rather than by the formation of progressively finer structures.

Writing $\omega$ for the perturbation vorticity and $u$ for the corresponding perturbation velocity, the governing 2D Euler equations take the form:
\begin{equation}\label{eq:euler}
\begin{cases}
\partial_t\omega+y\partial_x\omega+u\cdot\nabla\omega = 0 
& \\
\om |_{t=0} = \omega_{in} .\\
\end{cases}
\end{equation}
In this paper, the spatial domain $\Om $ is either the infinite channel $\R \times [-1,1] $ or the whole plane $\R^2$. The perturbation velocity is recovered from the stream function by
\begin{equation}\label{eq:def_stream}
    u=\nabla^\perp\psi=(-\partial_y\psi,\partial_x\psi),
    \qquad
    \Delta\psi=\omega,
\end{equation}
with the usual no-penetration boundary conditions in the channel.

Throughout this work, we study weak solutions in the classical Yudovich class. For initial vorticity $\omega_{in} \in L^{1}(\Omega) \cap L^{\infty}(\Omega)$, the 2D Euler equations \eqref{eq:euler} admit a unique global Yudovich weak solution. In this regime, the vorticity $\omega(t)$ is transported by the area-preserving flow generated by the velocity field $(y,0)+u$ and all $L^p$ norms of the vorticity are conserved for all times, $\| \omega(t)\|_{L^p} = \|\omega_{in}\|_{L^p}$ for $1 \le p \le \infty$. 
This low-regularity setting provides the natural functional framework for our results, which do not impose  any derivative assumptions on the initial data $\om_{in}$.

\subsection{Inviscid damping without derivatives}
Our first result establishes nonlinear inviscid damping in the infinite channel at the level of bounded vorticity. Throughout the paper, we write $z=(x,y)$ for a point in the channel or $\R^2$.

\begin{theorem}[{Inviscid damping in $\R\times[-1,1]$}]\label{thm:intro-channel-positive-damping}
Consider \eqref{eq:euler} in the infinite channel $\Om=\R\times[-1,1]$. There exists $\ep_0>0$ such that for any $\om_{in}\in L^1 \cap L^\infty(\Om)$ satisfying
\begin{equation}\label{eq:thm:intro-channel-positive-damping_1}
\omega_{in} \ge 0 \quad \text{and} \quad \| \om_{in} \|_{L^\infty (\Om)} \le \ep_0, 
\end{equation}
the corresponding Yudovich solution $\om(t)$ satisfies
\begin{equation}\label{eq:thm:intro-channel-positive-damping_2}
    \sup_{x_0 \in\mathbb R}
    \int_{|x-x_0| \le 1}
    \omega(t,z)\,dz
    \rightarrow 0 \quad \text{as $t \to \infty $}.
\end{equation}
Consequently,
\begin{equation}\label{eq:thm:intro-channel-positive-damping_3}
    \|u(t)\|_{L^p(\Om)}\to 0 \quad \text{for all $1 <  p \le \infty$}.  
\end{equation}

\end{theorem}
Notably, the smallness assumption is imposed only on $\|\omega_{in}\|_{L^\infty}$; no smallness is required of $\|\omega_{in}\|_{L^1}$, which can be arbitrarily large.

The uniformity with respect to $x_{0}$ in \eqref{eq:thm:intro-channel-positive-damping_2} excludes any vorticity packet of positive mass traveling to spatial infinity, which is the key to obtaining a global damping estimate. 

Several aspects of this damping mechanism are sharp in the Yudovich class:

\begin{itemize}
    \item \textbf{Exclusion of $L^{1}$ damping:} The endpoint \(p=1\) is excluded in \eqref{eq:thm:intro-channel-positive-damping_3} due to the velocity lower bounds demonstrated by  Theorem \ref{thm:couette-lower-bound} below.

    \item \textbf{Absence of universal rates:} In the Yudovich class, there can be no universal damping rate even within the positivity class (as demonstrated in Remark \ref{rmk:no uniform rate channel}).
\end{itemize}

Furthermore, the positivity assumption in Theorem \ref{thm:intro-channel-positive-damping} is structural: outside the nonnegative class, Hamiltonian conservation can prevent velocity damping.

Recall that the 2D Euler equations in the infinite channel preserve the Hamiltonian:
\[
    \mathcal H[\omega]
    =
    \int_\Omega |u|^2\,dz
    +
    \int_\Omega y^2\omega\,dz.
\]

\begin{theorem}[Non-damping in the channel] 
\label{thm:intro-channel-hamiltonian}
Consider \eqref{eq:euler} in the infinite channel $\Om=\R\times[-1,1]$. If the initial data $\om_{in} \in L^1\cap L^\infty (\Om)$ satisfy
\begin{equation}\label{eq:thm:channel-hamiltonian}
    \mathcal H[\om_{in}]>
    \int_\Omega(\om_{in})_+\,dz,
\end{equation}
then for the corresponding Yudovich solution $\om(t)$, its velocity field $u(t )$  satisfies for any $1\le p \le \infty $ the global lower bound:
\[
   \inf_{t \ge 0} \|u(t )\|_{L^p(\Omega)} \ge c_p \,>0.
\]

In particular, the criterion \eqref{eq:thm:channel-hamiltonian} applies to the following classes of initial data:
\begin{itemize}
    \item For any $s<1+\frac1p$ and $1\le p\le \infty$, there exist smooth compactly supported  (sign-changing or negative)  
\(\om_{in}^{(\ep)}\)  satisfying \eqref{eq:thm:channel-hamiltonian} and 
\[
\|\om_{in}^{(\ep)}\|_{W^{s,p}(\Omega)}\to0
\]
as \(\ep\to0\).

\item For every nonzero compactly supported vorticity profile \(\omega_\ast\), the
data
\[ \om_{in}=\lambda\omega_\ast
\]
satisfy \eqref{eq:thm:channel-hamiltonian} whenever \(|\lambda|\) is sufficiently large.
\end{itemize}
\end{theorem}

Earlier non-damping results near Couette flow were provided by nearby steady states or traveling waves at low regularity~\cite{MR2796139,CastroLear23Travellingwaves,2605.19971}. By contrast, the criterion~\eqref{eq:thm:channel-hamiltonian} is an \emph{open} condition in the Yudovich class,  thus giving, to our knowledge, the first open sets of  non-damping examples near Couette flow.

\subsection{A sign dichotomy in the plane}

Passing from the infinite channel to the whole plane changes the dynamics substantially because the vorticity is no longer confined in the vertical direction.

For all subsequent results on $\R^2$, we assume that
\begin{equation}\label{eq:intro_fix_R2_barycenter}
\omega_{in} \not\equiv 0, \quad  \omega_{in} \in L^1(\mathbb{R}^2) \cap L^\infty(\mathbb{R}^2), \quad \int_{\mathbb{R}^{2}}\vert{}z\vert{}^{2}\vert{}\omega_{in}(z)\vert{}dz < \infty  .
\end{equation}

Under these assumptions, the two signs of the perturbation lead to opposite long-time behavior.

First, for positive perturbations, we obtain uniform vorticity evacuation and velocity decay as $t\to\infty$, together with quantitative dispersion estimates.

\begin{theorem}[{Dispersion and damping in $\R^2$}]\label{thm:intro-R2-positive-damping} 
Consider \eqref{eq:euler} in the whole plane $\R^2$. 
There exists $\ep_0>0$   such that  if  the initial data $\om_{in} \in L^1\cap L^\infty(\R^2)$ are nonnegative and satisfy \eqref{eq:intro_fix_R2_barycenter} and the bound
$$
  \| \om_{in} \|_{L^\infty (\R^2  )} \le \ep_0,
$$
then the corresponding Yudovich solution $\om(t)$ exhibits:
\begin{enumerate}
    
\item  \textbf{Enhanced dispersion:} There exists   $C_{in}>0$ depending on $\om_{in}$ such that
\begin{equation}\label{eq:thm:intro-R2-positive-damping_4}
    \int_{\mathbb R^2}y^2\omega(t,z)\,dz \ge C_{in}\log \langle t \rangle \quad \text{and} \quad \int_{\mathbb R^2}x^2\omega(t,z)\,dz \ge C_{in}t^2\log \langle t \rangle  
\end{equation}
for all $t\ge 0$.
 
\item \textbf{Evacuation and damping:}
\begin{equation}
  \lim_{ t\to \infty}  \sup_{x_0\in\R}
    \int_{|x-x_0|\le1}
    \om(t,z)\,dz
    =0
\end{equation}
and for every $2<p\le\infty$,
$$
   \lim_{t\to\infty}\|u(t)\|_{L^p(\R^2)}=0.
$$

\end{enumerate}

\end{theorem}

As in Theorem \ref{thm:intro-channel-positive-damping}, the range $p>2$ is in fact sharp in the class of sign-definite vorticity: the corresponding velocity only decays like $|z|^{-1}$ at infinity. This result also reveals a nonlinear dispersion mechanism in the whole plane, with both the vertical and horizontal second moments gaining an extra logarithmic factor compared with the linearized Couette evolution.

Using the filamentation mechanism from \cite{MR4350517}, the enhanced dispersion above also leads to the growth of vorticity derivatives, even under arbitrary shear modulation. See Section \ref{subsec:notation} for the (standard) definitions of the Gevrey spaces $\mathcal G^\lambda_\sigma$ and the anisotropic H\"older spaces $C_x^s(\R^2) $ and $C_y^s(\R^2) $.

\begin{theorem}[Norm growth under arbitrarily small perturbations]
\label{thm:intro-R2-Gevrey-norm-growth}
Let $\ep_0>0$ be the constant in
Theorem~\ref{thm:intro-R2-positive-damping}. Fix any Gevrey space $\mathcal G^\lambda_\sigma $ for some
$0<\sigma<1$ and $\lambda>0$. For any nonzero initial
vorticity $\om_{b}$ satisfying
\[
    0\le \om_{b}\in C_c^\infty(\R^2),
    \qquad 
    \|\om_{b}\|_{L^\infty(\R^2)}
    \le
    \frac12 \ep_0,
\]
and any $\delta>0$, there exists a nonnegative perturbation
$\om_{p}\in C_c^\infty(\R^2) $
with
\begin{equation}\label{eq:intro-R2-Gevrey-small-perturbation}
    \|\om_{p}\|_{\mathcal G^\lambda_\sigma}
    \le
    \delta,
\end{equation}
such that the solution  $ \om(t)$ of \eqref{eq:euler} with
initial vorticity $ \om_{in}: =\om_{b}+ \om_{p} $ satisfies for any $s>0$,
\begin{equation}\label{eq:intro-R2-Cxs-growth}
    \| \om(t)\|_{C_x^s}
    \gtrsim_{s, \om_{in}}
    (\log \langle t\rangle)^{\frac{s}{2}}
\end{equation}
and
\begin{equation}\label{eq:intro-R2-Cys-growth}
    \| \om(t)\|_{C_y^s}
    \gtrsim_{s, \om_{in}}
    \left(
        t\sqrt{\log \langle t\rangle}
    \right)^s.
\end{equation}

\end{theorem}

Theorem~\ref{thm:intro-R2-Gevrey-norm-growth} shows that high-regularity asymptotic stability cannot in general hold on $\R^2$. Indeed, the growth of the horizontal norms cannot be removed by modulation with a shear: any change of coordinates $x \mapsto x+ \Phi(t,y)$ leaves the $C_x$-norm unchanged. This is reminiscent of the instability result~\cite{MR4630602} in the periodic setting, where finite-time norm inflation below the critical Gevrey threshold also persists under arbitrary shear modulation. However, our result shows that  the vorticity can exhibit unbounded growth (modulo any shear), while the velocity still undergoes inviscid damping.

For non-positive perturbations, the Hamiltonian acts in the opposite direction and prevents such spatial spreading.

\begin{theorem}[{Confinement and non-damping  in $\R^2$}] 
\label{thm:intro-R2-trapping}
Consider \eqref{eq:euler} in the whole plane $\R^2$. For any non-positive   and nonzero  initial data $\om_{in} \in L^1\cap L^\infty(\R^2)$ satisfying \eqref{eq:intro_fix_R2_barycenter}, the corresponding Yudovich solution $\om(t)$ exhibits:
\begin{itemize}
    \item  \textbf{Vorticity confinement:} The vertical second moment remains uniformly bounded
    \begin{equation}\label{eq:thm:intro-R2-nondamping_2}
    \limsup_{t\to\infty} 
    \int_{\mathbb R^2}y^2 |\omega(t,z)|\,dz
    <
    +\infty.
\end{equation} 
There exists a center trajectory $t\mapsto z_{t}=(x_{t},y_{t})$  such that for every $\ep>0$, there is a radius $R_{\ep}>0$ for which
\begin{equation}\label{eq:thm:intro-R2-nondamping_3}
    \int_{|z-z_t| \le R_{\ep}}
    |\om(t,z)|\,dz
    \ge
    \|\om_{in}\|_{L^1(\R^2)}-\ep \quad \text{for all } t\ge0.
\end{equation}
Moreover, the center trajectory $z_t=(x_t,y_t) \in \R^2$ satisfies 
\begin{equation}\label{eq:thm:intro-R2-nondamping_4}
    \lim_{t \to \infty} \left|\frac{x_t}{t} - v_{in}\right|=0
\quad \text{and} \quad
    \limsup_{t \to \infty} |y_t| <\infty,
\end{equation}  
where $v_{in} = \frac{ \int_{\R^2} y\om_{in}(z)\,dz}{\int_{\R^2} \om_{in}(z)\,dz}$ denotes the vertical barycenter of $\om_{in}$.

\item \textbf{Velocity non-damping:} For any $2 < p \le \infty$ and   $c >|v_{in}|$, the velocity field $u(t)$ satisfies
\begin{equation}\label{eq:thm:intro-R2-nondamping_1}
    \liminf_{t\to \infty} \|u(t)\|_{L^p(|z|\le ct)} >0.
\end{equation}

\end{itemize}

\end{theorem}

 A notable feature of Theorem~\ref{thm:intro-R2-trapping} is its generality. In contrast to the channel settings (both periodic and infinite), here, any non-positive perturbation, even arbitrarily small in smooth or Gevrey classes, is confined by Hamiltonian conservation.

The velocity non-damping in Theorem~\ref{thm:intro-R2-trapping} is not
merely caused by a non-vanishing horizontal zero mode. Analogous non-damping estimates are established for every component of $u$ and its first-order derivatives in Section \ref{subsec:R2 nondamping main proof}.

Theorem \ref{thm:intro-R2-positive-damping} and Theorem \ref{thm:intro-R2-trapping} therefore yield the following sign dichotomy.

\begin{table}[htbp]
\centering
\begin{tabularx}{\textwidth}{l X X}
\toprule
& $\omega_{in} \ge 0$ & $\omega_{in} \le 0$ \\
\midrule
\textbf{Vorticity} & Enhanced dispersion   & Trapping  \\
\addlinespace
\textbf{Velocity} & Inviscid damping & Non-damping\\
\bottomrule
\end{tabularx}
\small
\caption{Sign dichotomy in $\mathbb{R}^2$.}
\label{tab:sign_dichotomy_academic}
\end{table}

\subsection{Background}
We discuss our work in the context of hydrodynamic stability. Due to the vast literature in this field, we only discuss works that are most related to our setting: Couette flow, unbounded shear geometries, inviscid damping. Further references can be found in the surveys
\cite{MR3974608,MR4680382}.

 \subsubsection*{Phase mixing paradigm}

Inviscid damping belongs to a broader class of relaxation phenomena in conservative evolution equations. For the 2D Euler equations, the established nonlinear theory has been developed primarily through phase mixing near shear flows~\cite{MR3415068,MR4076093,MR4740211,MR4628607} and point vortices~\cite{MR4400903}. We refer to~\cite{Case1960Couette,Zillinger2016FiniteChannel,Zillinger2017MonotoneShear,WeiZhangZhao2018Sobolev,Jia2020Gevrey,MR4302767} for representative developments in the linear theory. Within this paradigm, nonlinear control requires regularity above the Gevrey--2 threshold, whose sharpness is demonstrated by the instability constructions in~\cite{MR4630602}. We also refer to~\cite{
ChenWeiZhangZhangInhomogeneous,
ZhaoInhomogeneousDamping,
bedrossian2024uniforminvisciddampinginviscid}
for further developments on nonlinear inviscid damping beyond the homogeneous Euler setting.

Although the mechanism developed in the present work is different, it bears a qualitative resemblance to vorticity depletion, a phenomenon first predicted in~\cite{BouchetMorita2010Euler} and later rigorously established in~\cite{WeiZhangZhao2019VorticityDepletion} at the linear level.

 \subsubsection*{Vorticity escape and spatial evacuation}

In~\cite{2603.20065}, we established Yudovich-level asymptotic stability for non-stagnant shear flows in the infinite channel by proving evacuation from fixed compact sets. That result gives decay relative to a fixed observer, but does not exclude a coherent vorticity packet drifting through the channel and therefore does not imply global velocity damping. The present work treats Couette flow, whose stagnation line $y=0$ lies outside the non-stagnant framework of \cite{2603.20065}, and upgrades fixed-set evacuation to uniform evacuation over all translated observation windows. The significance of stagnation is also visible in stationary rigidity results~\cite{HamelNadirashvili2017ShearFlows,HamelNadirashvili2019LiouvillePlane,DrivasNualart2026LaminarFlow,2605.19971}. The resulting large-scale transport mechanism is reminiscent of Taylor dispersion~\cite{Taylor54}, although here the relaxation is nonlinear and entirely inviscid.

 \subsubsection*{Viscous results with non-compact geometries}
 
Non-compact shear geometries have recently also gained traction in the viscous theory for the Navier-Stokes equations near Couette flow. Following the recent breakthrough~\cite{AB25}, the viscous stability of Couette flow has been established on the plane~\cite{AB25,LLZ26}, half-plane~\cite{AB25,LiuMasmoudiZhao2026HalfPlane}, and  infinite channel~\cite{AB25,ChenLiMiao2026InfiniteChannel}.  Our inviscid results intersect directly with these studies: the initial data considered here satisfy the initial conditions required in some of these viscous stability results, demonstrating  that the viscous threshold exponent cannot be extended to $0$.

\subsubsection*{Vorticity confinement}

The program of studying confinement and large-time spreading of sign-definite vorticity in open domains   was initiated in the work of Marchioro and Pulvirenti~\cite{MarchioroPulvirenti1983SingularVorticity,
Marchioro1988GlobalVortices,
MarchioroPulvirenti1993Localization}. Since then confinement estimates for planar vorticity were obtained using conservation of circulation, moments, and energy~\cite{Marchioro1994VortexSupport,IftimieSiderisGamblin1999PlanarVorticity,IftimieLopesFilhoNussenzveigLopes2007Confinement,
ChoiDenisov2019InfiniteCylinder}. Even though these works do not concern  inviscid damping near a background shear, their use of the vorticity sign, weighted moments, and logarithmic or renormalized energies \cite{Turkington1987} has inspired the techniques used here.

\subsubsection*{Filamentation and norm growth}

Filamentation of vorticity interfaces is a classical mechanism for the creation of small scales in 2D Euler flows, with early studies including the work of Dritschel~\cite{MR988302}. Recently, stability-based versions of this mechanism were used in~\cite{MR4168275,MR4555167,MR4350517}  to produce norm growth, which inspired our construction in Theorem \ref{thm:intro-R2-Gevrey-norm-growth}.  
Notably, the implicitly defined center introduced in \cite{jyz2025superlineargradientgrowth2d} inspired the $H$--center used here.

\subsection{Monotonicity and Lyapunov functionals}

We now describe the main strategy of the proof in the damping theorems. The basic idea is that positive vorticity cannot remain localized without producing an   increase in a suitable weighted moment. This creates a dissipation-like mechanism in our setting. We discuss below the specific strategies adapted to the geometry of each domain.

\subsubsection*{The infinite channel: moving functionals for uniform evacuation}

To detect vorticity localized away from the stagnation line $y=0$ in the channel, let $H$ be a bounded increasing function of the horizontal variable and define
$$
\mathcal M(t)=\int yH(x)\omega(t,x,y) \quad \text{and} \quad
\mathcal D(t)=\int y^2H'(x)\omega(t,x,y) .
$$
Differentiating along the flow produces key monotonicity in the perturbative regime:
\begin{equation}\label{eq:intro_mono}
\frac{d}{dt}\mathcal M(t)\geq \frac12 \mathcal D(t). 
\end{equation}
The sign condition $\omega\geq0$ is crucial in proving \eqref{eq:intro_mono}: after symmetrization, the vertical Biot--Savart interaction contributes with a favorable sign, while the remaining nonlinear term is absorbed by the smallness of the vorticity.

In the channel, $\mathcal M(t)$ is uniformly bounded, so the total dissipation  $\int_0^\infty \mathcal D(t)\,dt$  must remain finite. This  prevents a positive amount of vorticity from trapping near any fixed horizontal location. This already proves a form of local  convergence/stability in the spirit of \cite{2603.20065}. 

However, proving uniform  local evacuation requires tracking any drifting packet. We therefore shift to a moving functional along a linear trajectory $X(t)=X_0+ct$:
$$
\mathcal M_X(t)=\int (y-c)H(x-X(t))\omega(t ) 
\,\, \text{and} \,\,
\mathcal D_X(t)=\int (y-c)^2H'(x-X(t))\omega(t ) .
$$
Here $y-c$ measures the difference between the background transport speed and the speed of the moving observer. The same sign structure yields $
\frac{d}{dt}\mathcal M_X(t)\geq \frac12 \mathcal D_X(t).$

If uniform local evacuation fails, then a fixed amount of vorticity remains concentrated near a sparsely chosen sequence of space-time points $(t_n,z_n)$.  Joining these points by a piecewise linear trajectory,  the moving monotonicity estimate on each segment  gives total finite dissipation $\sum_n\int_{t_n}^{t_{n+1}} \mathcal D_n(t)\,dt<\infty.$

On the other hand, concentration at these points forces a uniformly positive dissipation on each segment, causing the sum to diverge and yielding a contradiction. This establishes uniform local evacuation and global velocity damping in the channel.

\subsubsection*{The whole plane: time-weighted Lyapunov functionals}

In the whole plane, the vorticity may spread in the vertical direction and $\mathcal M(t)$ need not remain bounded. We use a time-weighted Lyapunov functional adapted from the conserved Hamiltonian to prove uniform evacuation. 

The key new idea is to adapt the Hamiltonian to the linear Couette evolution. Along a linear trajectory, $y(t)-x(t)/t=-x_0/t$, while particles with different heights separate linearly in time. This suggests the functional
\[
\begin{aligned}
\mathcal F(t)
={}&\int_{\R^2}\left(y-\frac xt\right)^2\om(t,z)\,dz\\
&-\frac1{2\pi}\iint_{\R^2\times\R^2}
\log\frac{|z-z'|}{t}\,\om(t,z)\om(t,z')\,dzdz',
\qquad t>0.
\end{aligned}
\]
Differentiating along the flow gives, under the sign and smallness assumptions,
\[
\frac{d}{dt}\mathcal F(t)
\le
-\frac1t\int_{\R^2}\left(y-\frac xt\right)^2\om(t,z)\,dz
+\frac{C}{t^2}
\quad \text{for a.e. }t\ge1.
\]
The smallness of the vorticity absorbs the nonlinear interaction into the negative quadratic term, while the logarithmic normalization cancels the mass-dependent terms of order $1/t$. 

If a fixed amount of vorticity remains in a  ball at arbitrarily large times, it contributes a positive term of order $O(\log t)$ to $\mathcal F(t)$. On the other hand, the moment bounds from Hamiltonian conservation limit the negative contribution of distant pairs to order $O(\log\log t)$. This contradicts the upper bound for $\mathcal F(t)$ and establishes uniform evacuation from unit balls and global velocity damping.  The strip evacuation statement then follows from vertical truncation and moment bounds.

To prove enhanced dispersion, we adapt the moving functional in the channel case. For a suitable power  weight $H$, define the $H$--center $X_0(t) \in \R$ by
\begin{equation}\label{eq:intro_centering}
\int_{\R^2}H(x-X_0(t))\om(t,z)\,dz=0.
\end{equation}

For this $H$-center $X_0(t)$, we set
\[
\begin{aligned}
\mathcal M_0(t)&=\int_{\R^2}yH(x-X_0(t))\om(t,z)\,dz,\\
\mathcal D_0(t)&=\int_{\R^2}\bigl(y-X_0'(t)\bigr)^2H'(x-X_0(t))\om(t,z)\,dz.
\end{aligned}
\]

For a suitable choice of $H$, the centering condition \eqref{eq:intro_centering}  yields a coercive
estimate for the positive terms in the monotonicity identity.
Combining this with the moment bounds, we obtain
\[
\frac{d}{dt}\mathcal M_0(t)
\gtrsim_{\om_{in}}
\frac{1}{1+t\sqrt{1+\log\langle t\rangle}}.
\]
Integrating in time gives
$\mathcal M_0(t)\gtrsim_{\om_{in}}\sqrt{\log t}$ for large $t$,
and hence a logarithmic lower bound for the vertical second moment.
The corresponding enhancement of horizontal dispersion follows from
the background shear.

\subsection{Hamiltonian conservation and non-damping}
The non-damping results arise from the opposing mechanism: conserved Hamiltonian quantities prevent the velocity from becoming small. Just as with the Lyapunov functionals, the exact nature of this mechanism is dictated by the geometry of the domain.

\subsubsection*{The infinite channel: velocity lower bounds and Sobolev examples}

In the channel, the Hamiltonian relates the kinetic energy to a weighted vorticity moment. When criterion \eqref{eq:thm:channel-hamiltonian} holds, this relationship forces a uniform-in-time lower bound for the velocity. This criterion then allows for the construction of small non-damping perturbations: by concentrating negative vorticity near the stagnation line $y=0$, the kinetic part of the Hamiltonian dominates the weighted vorticity moment, yielding data arbitrarily small in $W^{s,p}$ for every $s < 1+\frac{1}{p}$.

\subsubsection*{The whole plane: logarithmic divergence and vorticity trapping}

In the whole plane, the Hamiltonian balances a logarithmic interaction energy against the vertical second moment. For non-positive perturbations, if the velocity were to damp, the equivalence in Proposition~\ref{prop:damping-circulation-equivalence} would force the vorticity mass to vanish uniformly from every fixed-radius ball. This loss of spatial tightness drives the logarithmic interaction energy to infinity, contradicting Hamiltonian conservation.

\subsection{Organization}

The rest of the paper is organized as follows.

\begin{itemize}

\item In Section \ref{sec:preliminaries}, we collect the notation and some Biot--Savart estimates for the infinite channel and whole plane geometries.

\item Sections \ref{sec:channel_damping} and \ref{sec:nodamping_channel} concern the infinite channel. In Section \ref{sec:channel_damping}, we introduce the moving Lyapunov functionals and prove global inviscid damping for small nonnegative perturbations. In Section \ref{sec:nodamping_channel}, we use Hamiltonian conservation to derive quantitative velocity lower bounds and construct low-Sobolev non-damping examples.

\item Sections \ref{sec:damping_wholespace} and \ref{sec:plane-negative} concern the whole plane. 
In Section \ref{sec:damping_wholespace}, we prove uniform vorticity evacuation, global velocity damping, and quantitative dispersion for small nonnegative vorticity. In Section \ref{sec:plane-negative}, we treat negative vorticity and establish velocity non-damping together with tightness modulo translations.

\item Appendix \ref{append:proof} proves the general equivalence between damping and local evacuation used throughout the paper, while Appendix \ref{append:linear rate} establishes the
 linear inviscid damping rates in both the infinite channel and whole plane.

\end{itemize}

\subsection*{Statement on AI usage}
In the first version of this paper, the authors used ChatGPT 5.6 for mathematical discussion and language editing. That version contained all the results of the present paper, except that damping for nonnegative vorticity in $\R^2$ was established only along a set of times of asymptotic density one. In revising the paper, the authors used ChatGPT 6  to assist in strengthening this result to global-in-time damping, with smallness required only in $L^\infty$. All results and proofs were written and verified by the authors, who take full responsibility for the paper.

\subsection*{Acknowledgments}

This work was supported by the National Natural Science Foundation of China   under Grant No. 12421001 and No. 12288201. DG is supported by NSFC No. 12671283.

\section{Preliminaries}
\label{sec:preliminaries}

\subsection{Notation}
\label{subsec:notation}

Throughout the paper, the spatial domain is either the infinite channel
$
    \Omega=\mathbb R\times[-1,1]
$
or the full plane
$
    \Omega=\mathbb R^2.
$
The relevant domain will always be clear from the section under consideration.

For a point $z\in\Om$, we denote its components as $z=(x,y)$. We also write
$ dz=dxdy. $ When no ambiguity can arise, we omit the domain in integrals, i.e. $
    \int \om(z)\,dz$.
 
We also use $\N$ to denote the set of positive integers.

The positive and negative parts of a function $f$ are denoted by
$$
    f_+:=\max\{f,0\},
    \qquad
    f_-:=\max\{-f,0\},
$$
so that
$$
    f=f_+-f_-,
    \qquad
    |f|=f_++f_-.
$$
For $1\le p\le\infty$, the Lebesgue norm $\|f\|_{L^p}$ is taken over the underlying
spatial domain unless otherwise specified.

We use the Japanese bracket notation
$
    \langle a\rangle:=(1+|a|^2)^{1/2}.
$
The notation
$
    A\lesssim B
$
means that there exists a constant $C>0$ such that
$
    A\le CB.
$
Similarly,
$
    A\gtrsim B
$
means $B\lesssim A$. The notation
$
    A\approx B
$
means both $A\lesssim B$ and $B\lesssim A$. If the implicit constant depends on
a parameter, we indicate this by a subscript, for instance
$
    A\lesssim_p B.
$

For any $R>0$ and $z_0\in \Om$, $B_R(z_0)$ denotes the closed ball of radius $R$
$$
    B_R(z_0):= \{z\in \Om:  |z-z_0| \le R \}.
$$

We recall the Gevrey norm $\mathcal G^\lambda_\sigma$ on $\R^2$: for any $ \sigma \ge 0$ and $\lambda >0$,
\begin{equation}\label{eq:R2-Gevrey def}
\begin{aligned}
\|f\|_{\mathcal G^\lambda_\sigma}^2
:=
\int_{\R^2}
e^{2\lambda\langle\xi\rangle^\sigma}
 |
\widehat{f}(\xi) 
 |^2\,d\xi ,
\end{aligned}
\end{equation}
where $\widehat{f}$ denotes the Fourier transform.
 
For any $s=m+\alpha$ with integer $m \ge 0$ and $0\le \alpha < 1$, we recall the H\"older norm
\begin{equation}\label{eq:def-classical-holder}
\begin{aligned}
    \|f\|_{C^{m+\alpha}(\R)}
    :=
    \sum_{j=0}^{m}
    \|\partial_x^j f\|_{L^\infty(\R)}
    +
    \sup_{x\neq x'}
    \frac{|\p_x^mf(x)-\p_x^mf(x')|}{|x-x'|^\alpha}.
\end{aligned}
\end{equation}
Next, we define the anisotropic H\"older norms on
$\R^2$. For any $s>0$ and $f:\R^2\to\R$, we define
\begin{equation}\label{eq:def-anisotropic-holder}
\begin{aligned}
    \|f\|_{C_x^s}
    :=
    \sup_{y\in\mathbb R}
    \|f(\cdot,y)\|_{C^s(\mathbb R)},\quad \text{and}\quad
    \|f\|_{C_y^s}
    :=
    \sup_{x\in\mathbb R}
    \|f(x,\cdot)\|_{C^s(\mathbb R)}.
\end{aligned}
\end{equation}

\subsection{Biot--Savart kernel on \texorpdfstring{$\R \times [-1,1]$}{R x [-1,1]}} 
We first fix the sign conventions used throughout the paper. For a scalar function
$f$, we write
$
\nabla^\perp f:=(-\p_y f,\p_x f).
$
Thus, if $
u=\nabla^\perp\psi$ and $\Delta\psi=\om$,
then $
\nabla\cdot u=0$ and $\p_x u^y-\p_y u^x=\om.$

Let $\Omega:=\mathbb R\times[-1,1]$, with elliptic equations understood in the interior
$\mathbb R\times(-1,1)$. 

Let $G(z,z')$ denote the Dirichlet Green function for the Laplacian $\Delta$:
$$
\Delta_zG(z,z')=\delta_{z'}
\quad\text{in }\Om,
\qquad
G(z,z')=0
\quad\text{for }y=\pm1.
$$

Using Dirichlet eigenfunctions on $[-1,1]$  and separation of variables, we have
\[
G(z,z')
=-\frac1\pi\sum_{n=1}^{\infty}\frac1n
e^{-n \frac\pi2 |x-x'|}
\sin\left(\frac{n\pi}{2}(y+1)\right)
\sin\left(\frac{n\pi}{2}(y'+1)\right).
\]
Note here the exponential screening effect in the infinite channel. 

Summing the series yields the Green function on $\Omega$:
\[
G(z,z')
=\frac{1}{4\pi}\log\frac{\cosh(a\xi)-\cos\bigl(a(y-y')\bigr)}{\cosh(a\xi)+\cos\bigl(a(y+y')\bigr)
 },
\]
where $\xi =x - x'$ and $a =\frac\pi2$ for short.

The formula also shows directly that
\[
G(z,z')=G(z',z),
\qquad
G((x,\pm1),z')=0.
\]

For any vorticity $\om\in L^1(\Om)\cap L^\infty(\Om)$, we define
\[
\psi(z):=\int_\Omega G (z,z')\omega(z')\,dz',
\qquad
u(z)=\nabla^\perp\psi(z).
\]

The corresponding Biot--Savart representation is
 \begin{equation}\label{eq Biot Savart}
u(z)=\int_\Om K (z,z')\om(z')\,dz',
\qquad
K(z,z'):=\nabla_z^\perp G(z,z'),
 \end{equation}
 where the Biot--Savart kernel $K(z,z') = (K^x ,K^y )$ is given by
\begin{equation}\label{eq:channel Kx}
    K^x =
    \frac18
    \left[
        \frac{
            -\sin\left(a(y+y')\right)
        }{
            \cosh\left(a\xi\right)
            +
            \cos\left(a(y+y')\right)
        }
        -
        \frac{
            \sin\left(a(y-y')\right)
        }{
            \cosh\left(a\xi\right)
            -
            \cos\left(a(y-y')\right)
        }
    \right],
\end{equation}
and
\begin{equation}\label{eq:channel Ky}
    K^y =
    \frac14
    \frac{
         \sinh\left(a\xi\right)
        \cos\left(ay\right)
        \cos\left(ay'\right)
    }{
        \left[
            \cosh\left(a\xi\right)
            -
            \cos\left(a(y-y')\right)
        \right]
        \left[
            \cosh\left(a\xi\right)
            +
            \cos\left(a(y+y')\right)
        \right]
    }.
\end{equation}
The following estimates collect the properties of the channel kernel needed
throughout the paper.
\begin{lemma}\label{lem:kernel_change x x'}
    For any $0\le \sigma <\frac{\pi}{2}$, the kernel $K=(K^x,K^y)$ in \eqref{eq:channel Kx} and \eqref{eq:channel Ky} satisfies
    \begin{equation}\label{eq K1}
     \sup_{x,y,y'}
    \int_{\mathbb R}e^{\sigma|x-x'|}|K^x(x,y,x',y')|\,dx'
    \lesssim_{\sigma} 1
\end{equation}
and
    \begin{equation}\label{eq:anti Kx}
(x-x')K^y(z,z') \ge 0.
\end{equation}
Moreover,
\begin{equation}\label{eq estimate on K}
|K(z,z')|
\lesssim
\begin{cases}
\dfrac{1}{|z-z'|}, & |z-z'|\le 1, \\[1ex]
e^{-\frac{\pi}{2}|x-x'|}, & |z-z'|\ge 1.
\end{cases}
\end{equation}
\end{lemma}
\begin{proof}
    The estimates \eqref{eq:anti Kx} and \eqref{eq estimate on K} are standard so we omit them. For \eqref{eq K1}, from the explicit formula above, we see that
$$
\begin{aligned}
    \int_{\R} e^{\sigma|x-x'|}|K^x(x,y;x',y')|dx'
    &\lesssim
    \int_{\R} \frac{
        e^{\sigma|\xi|}\left|\sin\left(\frac{\pi}{2}(y+y')\right)\right|
    }{
        \cosh\left(\frac{\pi}{2}\xi\right)
        +
        \cos\left(\frac{\pi}{2}(y+y')\right)
    }\,d\xi \\
    &\quad+
    \int_{\R}\frac{
        e^{\sigma|\xi|}\left|\sin\left(\frac{\pi}{2}(y-y')\right)\right|
    }{
        \cosh\left(\frac{\pi}{2}\xi\right)
        -
        \cos\left(\frac{\pi}{2}(y-y')\right)
    }\,d\xi.
\end{aligned}
$$
Since $|\sin \theta|=|\sin(\theta-\pi)|$ and $\cos \theta=-\cos(\theta-\pi)$, it suffices to prove that
\begin{equation}\label{eq bound for Kx}
 \sup_{\theta\in [0, 2\pi]}
    \int_{\mathbb R}
    \frac{e^{\sigma|\xi|} |\sin\theta|}
    {\cosh\left(\frac{\pi}{2}\xi\right)-\cos\theta}
    \,d\xi
    \lesssim_{\sigma} 1.
\end{equation}
 
To show \eqref{eq bound for Kx}, for any $\theta \in [0,2\pi]$, we set $d(\theta)=\min\{ \theta, 2\pi-\theta \}$ and consider two regions $|\xi|\le 1 $ and $|\xi|\ge 1 $.

When $|\xi|\le 1 $, we have $\cosh (\frac{\pi}{2}\xi ) \geq 1 + |\xi|^2/4   $. Using also $|\sin \theta| \lesssim d(\theta)$ and $ 1-\cos\theta \gtrsim d^2(\theta)$, 
$$
    \int_{|\xi|\le1}
    e^{\sigma|\xi|}
    \frac{|\sin\theta|}
    {\cosh\left(\frac{\pi}{2}\xi\right)-\cos\theta}
    \,d\xi \lesssim \int_{|\xi| \le 1} \frac{|d(\theta)|}{\xi^2+d^2(\theta)}\,d\xi \lesssim 1.
$$

When $|\xi|\ge 1$, we have
$
    \cosh\left(\frac{\pi}{2}\xi\right)-\cos\theta
    \gtrsim
    e^{\frac{\pi}{2}|\xi|}.
$
Since $\sigma<\frac{\pi}{2}$, the integral in this region is also integrable:
$$
\int_{|\xi|\ge 1} e^{\sigma|\xi|}
    \frac{|\sin\theta|}
    {\cosh\left(\frac{\pi}{2}\xi\right)-\cos\theta}
    \lesssim \int_{|\xi|\ge 1}
    e^{-\left(\frac{\pi}{2}-\sigma\right)|\xi|} \lesssim_{\sigma} 1 .
$$

\end{proof}
Applying \eqref{eq estimate on K}, it follows directly that
\begin{equation}\label{eq:velocity-bound}
   \|u\|_{L^1} \lesssim\|\om\|_{L^1} \quad \text{and} \quad \|u\|_{L^\infty}
    \lesssim \|\om\|_{L^{\infty}}.
\end{equation}
Moreover, the usual log-Lipschitz coefficient does not depend on $L^1$ of $\omega$ due to the exponential decay of the Biot--Savart kernel:
$$
    |u(z_1 )-u( z_2 )|
    \le
    C \|\om\|_{L^\infty(\Omega)}
    |z_1- z_2|
    \left(
        1+\left|\log |z_1-z_2|\right|
    \right) \quad \text{for any $|z_1 -z_2| \le 1$}.
$$

\subsection{Biot--Savart kernel on \texorpdfstring{$\R^2$}{R2}}

Next, we recall the fundamental solution of the Laplacian on $\R^2$, 
$$
G(z,z')
:=\frac{1}{2\pi}\log|z-z'|,
\qquad
\Delta_z G(z,z')=\delta_{z'}.
$$

The corresponding velocity field is given by the planar Biot--Savart law
\begin{equation}\label{eq:R2-Biot-Savart}
    u(z)
    =
    \int_{\mathbb R^2}K(z,z')\om(z')\,dz',
\end{equation}
where
$$
    K^x(z,z')
    =
    -\frac1{2\pi}
    \frac{y-y'}{|z-z'|^2},
    \qquad
    K^y(z,z')
    =
    \frac1{2\pi}
    \frac{x-x'}{|z-z'|^2}.
$$

For later use, we record some basic velocity estimates.  By splitting the
Biot--Savart integral, we obtain that for every $2<p\leq\infty$,
\begin{equation}\label{eq:R2-velocity-bound}
 \|u\|_{L^p(\mathbb R^2)}
\lesssim_p
\|\omega\|_{L^1(\R^2)}^{\frac12+\frac1p}
\|\omega\|_{L^\infty(\R^2)}^{\frac12-\frac1p},
\end{equation}
and the usual  log-Lipschitz estimates: for any $|z_1 -z_2| \le 1$, 
$$ 
    |u(z_1 )-u( z_2 )|
    \le
    C  \|\om\|_{L^1 \cap L^\infty}
    |z_1- z_2|
    \left( 1+\left|\log |z_1-z_2|\right|
    \right). 
 $$

We also need the following lemma, concerning the divergence of the logarithmic interaction energy on $\R^2$.

\begin{lemma}
\label{lem:R2-log-energy-divergence}
Let $\om\in L^1(\R^2)\cap L^\infty(\R^2)$  with a fixed sign satisfying
$$
\int_{\R^2}|z|^2 |\om(z)| \,dz<\infty.
$$
 
Then for every $R\ge 2$,  
\begin{equation}\label{eq:R2-positive-logarithmic-lower-bound}
\begin{aligned}
&\quad \iint_{\R^2\times\R^2}
\log|z-z'|
\om( z)\om( z')\,dzdz'\\
& \ge
\log R \, \|\om\|_{L^1}
\left(
\|\om\|_{L^1}
-
\sup_{ z_0\in\R^2}
\int_{B_R( z_0)}|\om(z)|\,dz
\right) - \frac{\pi}{2}\|\om \|_{L^1 }
\|\om \|_{L^\infty } .
\end{aligned}
\end{equation}
\end{lemma}

\begin{proof}
By replacing $\om$ with $-\om$ if necessary, we may assume that $\om \ge0$.

Since $\log|z-z'| \ge 0$ when $|z-z'| \ge 1 $, for every $R \ge 2$, we have
$$
\begin{aligned}
\quad &\iint_{|z-z'|\ge1}
\log|z-z'|\om(z)\om(z')\,dzdz' \ge \log R \iint_{|z-z'|\ge R}\om(z)\om(z')\,dzdz'
\\
&\ge  \log R \left( \|\om\|_{L^1(\R^2)}^2
-
\iint_{|z-z'|<R}\om(z)\om(z')\,dzdz' \right)\\
&\ge
\log R \|\om\|_{L^1(\R^2)}
\left(
\|\om\|_{L^1(\R^2)}
-
\sup_{ z_0\in\R^2}
\int_{B_R( z_0)}|\om(z)|\,dz
\right).
\end{aligned}
$$
Moreover, in the regime $0< |z-z'|\le 1 $, $\log |z-z'| \le 0$, and by $\om \ge 0$ we have
\begin{equation}
\begin{aligned}
& \iint_{|z-z'|<1}
\log|z-z'|\om(z)\om(z')\,dzdz'\\
&\ge
- \|\om\|_{L^1(\R^2)}
  \|\om\|_{L^\infty(\R^2)}
\int_{|z|<1}|\log|z||\,dz = -\frac{\pi}{2} \|\om\|_{L^1(\R^2)}
  \|\om\|_{L^\infty(\R^2)}.
\end{aligned}
\end{equation}
Combining the preceding two estimates, we obtain
\eqref{eq:R2-positive-logarithmic-lower-bound}.
\end{proof}

\subsection{Inviscid damping and uniform evacuation}

For vorticity with a definite sign, we have the following criterion for damping.

\begin{proposition}[Equivalence of damping and uniform evacuation]
\label{prop:damping-circulation-equivalence}
Let $\Omega$ be either $\mathbb R\times[-1,1]$ or $\mathbb R^2$. Let
$\{\om_n\}_{n\ge1}$ be a sequence of fixed sign vorticities satisfying
$$
    \sup_{n\ge1}
    \left(
        \|\om_n\|_{L^1(\Omega)}
        +
        \|\om_n\|_{L^\infty(\Omega)}
    \right)
    <\infty.
$$
Let $u_n:\Omega\to\mathbb R^2$ be the velocity fields of $\om_n$. Then the
following statements are equivalent:
\begin{enumerate}
    \item There exists an admissible $p$ such that
    $
        \lim_{n\to\infty}\|u_n\|_{L^p(\Omega)}=0.
    $
    \item For every admissible $p$,
    $
        \lim_{n\to\infty}\|u_n\|_{L^p(\Omega)}=0.
    $
    \item 
    $
        \lim_{n\to\infty}
        \sup_{ z_0\in\Omega}
        \int_{|z-z_0|\le 1}|\om_n(z)|\,dz
        =0.
    $
\end{enumerate}
Here the admissible range is $1< p\le\infty$ when
$\Omega=\mathbb R\times[-1,1]$ and $2<p\le\infty$ when
$\Omega=\mathbb R^2$.
\end{proposition}

The proof of Proposition \ref{prop:damping-circulation-equivalence} is based on elementary facts about the Biot--Savart law and is given in Appendix \ref{append:proof}.

\section{Inviscid damping for positive vorticity in the infinite channel}\label{sec:channel_damping}

We now turn to the positive-vorticity regime in the infinite channel and prove Theorem \ref{thm:intro-channel-positive-damping}. In view of Proposition \ref{prop:damping-circulation-equivalence}, it is enough to establish uniform local evacuation of the vorticity. 

The mechanism that rules out such concentration is a moving Lyapunov functional adapted to linear trajectories which will be our observation windows. We first establish its monotonicity, then use it in a contradiction argument to deduce uniform evacuation and, consequently, global decay of the perturbation velocity.

Throughout this section, we fix the spatial domain $\Omega =\R \times [-1,1]$.

\subsection{Lyapunov functionals}
We begin by constructing the moving Lyapunov functional that detects vorticity concentration along linear trajectories. 

Throughout this section, we define a horizontal weight $H(\xi)$ on $\R $:
\begin{equation}\label{eq:def_channel_H}
 H(\xi)=\int_0^{\xi}e^{-|s|}\,ds.
\end{equation}

\begin{lemma}\label{lem:property of H}
The weight $H$ satisfies
\begin{equation}\label{eq anti Kx}
 \|H\|_{L^\infty} \lesssim 1 \quad \text{and} \quad
\quad H'(\xi)
    >0. 
\end{equation}
Moreover, for any $X,x'\in\mathbb R$ and any $y,y'\in[-1,1]$,  
\begin{equation}\label{eq change x of Kx}
    \int_{\mathbb R}
    H'(x-X)|K^x(x,y;x',y')|\,dx
    \le  CH'(x'-X),
\end{equation}
where the constant
$C$ is independent of $X,x',y,y'$.
\end{lemma}

\begin{proof}
Estimate \eqref{eq anti Kx} is a direct consequence of \eqref{eq:def_channel_H}. The estimate \eqref{eq change x of Kx} follows directly from Lemma \ref{lem:kernel_change x x'}  (since $1<\frac{\pi}{2}$) and the fact that
$$
\frac{H'(x-X)}{H'(x'-X)}=e^{|x'-X|-|x-X|}\le e^{|x-x'|}.
$$
\end{proof}
\begin{remark} 
The particular choice of $H$ in \eqref{eq:def_channel_H} is inspired by \cite{2603.20065}  but not essential for the
moving Lyapunov argument. The argument below holds for any
weight $H \in C^1(\R)$ satisfying \eqref{eq anti Kx} and \eqref{eq change x of Kx}. 
\end{remark}
Let $I=[t_1,t_2] \subset [0,+\infty)$ be a non-empty bounded interval. Consider a linear trajectory $X: I \to \R  $ 
$$
    X(t)=x_*+c(t-t_1),
    \qquad
    t\in I,
$$
where $x_*, c \in \R$ are constants. Given a linear trajectory $X: I \to \R  $ and a Yudovich solution $\om(t)$ of \eqref{eq:euler}, we define the functionals
\begin{equation}\label{eq def of M_I}
\mathcal M_{X}(t)
    :=
    \int_\Omega
    (y-c)H(x-X(t))\omega(t,z)\,dz
\end{equation}
and
\begin{equation}\label{eq def of D_I}
 \mathcal   D_{X}(t)
    :=
    \int_\Omega
    (y-c)^2H'(x-X(t))\omega(t,z)\,dz .
\end{equation}

The following moving Lyapunov monotonicity estimate is the key ingredient in the proof. We state it here and postpone its proof to Section \ref{sec proof of Morawetz in channel}. 

We emphasize that the following estimate is uniform in $x_*$, $c$, $t_1$ and $t_2$. This is important---in the applications the trajectory will be chosen piecewise, with different slopes on different time intervals.
\begin{proposition}[Moving Lyapunov monotonicity]\label{prop:one-segment}

There exists a universal $\ep_0>0$ such that the following holds. 

 Let  $\om_{in}\in L^1 \cap L^{\infty}(\Omega)$ be the initial vorticity such that
 $$
 \om_{in}\ge 0 \quad \text{and} \quad \|\om_{in}\|_{L^{\infty}} \le \ep_0 
 $$
and let $\om(t)$  be the corresponding Yudovich solution to \eqref{eq:euler}.

Consider a  linear trajectory  $X: I \to \R$  on a  bounded interval $I \subset [ 0,\infty)$ and the quantities \eqref{eq def of M_I} and  \eqref{eq def of D_I}. Then $t\mapsto  \mathcal M_{X}(t)$ is Lipschitz on $I$ and 
\begin{equation}\label{eq:one-segment}
    \frac{d}{dt}\mathcal M_{X}(t)\ge\frac12 \mathcal D_{X}(t),
    \qquad a.e. \,\,
    t\in I.
\end{equation}

\end{proposition}

\subsection{Proof of inviscid damping  in the channel}\label{subsection proof of thm1.1}

Now we are ready to prove our main theorem on $\R \times [-1,1]$.
\begin{proof}[Proof of Theorem \ref{thm:intro-channel-positive-damping}]
Let $\ep_0>0$ be given by Proposition \ref{prop:one-segment}. By Proposition \ref{prop:damping-circulation-equivalence}, it suffices to show that  
$$
    \sup_{ x_0\in\R} \int_{|x- x_0|\le 1} \int_{-1}^1 \om(t,x,y)\,dy\,dx \to 0
    \quad
    \text{as }\quad t\to\infty.
$$

We argue by contradiction. Suppose  that there exist  $\delta>0$,
a sequence of times $t_k\to\infty$ and centers $x_k\in\mathbb R$ such that
\begin{equation}\label{eq:bad-original}
    \int_{|x-x_k|\le 1}\int_{-1}^{1}
    \omega(t_k,x,y)\,dydx
    \ge\delta  \quad \text{for every} \quad k\in \N.
\end{equation}

\vspace{0.5em}
\noindent
\textbf{Step 1: Bounded slopes.}

Since $\om_{in} \in L^1$, we can choose $R>0$ such that 
\begin{equation}\label{eq:tail-choice}
    \int_{|x|>R}\omega_{in}(x,y)\,dxdy
    \le\frac{\delta}{4}.
\end{equation}
Let $\Phi_t$ be the flow map generated by the velocity $(y+u^x,u^y)$, namely
$$
\begin{cases}
\displaystyle
\frac{d}{dt}\Phi_t(z)
=
\left(
\Phi_t^y(z)+u^x(t,\Phi_t(z)),
u^y(t,\Phi_t(z))
\right),\\[1ex]
\Phi_{0}(z)=z.
\end{cases}
$$  

Next, we decompose the solution as the ``main'' plus ``tail'':
$
    \om  (t)=\om_{m} (t)+\om_{e} (t),
$
where
$$
    \om_{m}(t,\Phi_t(z))
    =
    \omega_{in}(z)\mathbf 1_{\{|x|\le R\}},
$$
and
$$
    \om_{e}(t,\Phi_t(z))
    =
    \omega_{in}(z)\mathbf 1_{\{|x|>R\}}.
$$
Since the velocity is incompressible, we see that
$$
    \int_{\Om}\om_{e}(t,z)\,dz
    =
    \int_{|x|>R}\omega_{in}(z)\,dz
    \le\frac{\delta}{4}.
$$
Therefore, it follows from \eqref{eq:bad-original} that
$$
    \int_{|x-x_k|\le 1 }\int_{-1}^{1}
    \om_{m}(t_k,x,y)\,dydx
    \ge\frac{3\delta}{4}.
$$
Moreover, one checks directly that there exists $V>0$ such that $|y+u^x(t,z)| \le V$ for any $z \in \R \times [-1,1]$ and $t \ge0$. So $\Supp \om_{m} (t) \subset \{|x|\le R+Vt\}$, which implies 
\begin{equation}\label{eq def x_n t_n}
    \left|\frac{x_k}{t_k}\right|
    \le
    V+\frac{R+1}{t_k}
\end{equation}
uniformly  for all sufficiently large $k$. 

\vspace{0.5em}
\noindent
\textbf{Step 2: Slopes of bounded variation.}

By \eqref{eq def x_n t_n}  and passing to a subsequence, there exists $c_{\infty} \in \R$ such that
$$
    \frac{x_k}{t_k}\to c_\infty.
$$
We choose a further subsequence, denoted as $(t_n,x_n)$, such that
\begin{equation}\label{eq def t_n 1}
    t_n\to\infty \quad \text{and} \quad t_n \ge 1 \quad \text{for all} \quad n \in \N,
\end{equation}
\begin{equation}\label{eq:bad-subsequence}
    \int_{|x-x_n|\le 1 }\int_{-1}^{1}
    \omega(t_n,x,y)\,dydx
    \ge\delta,
\end{equation}
and
\begin{equation}\label{eq:sparse-choice}
    t_{n+1}\ge4t_n,
    \qquad
    \left|\frac{x_n}{t_n}-c_\infty\right|\le2^{-n}.
\end{equation}
For each $n$, we define the slope
$$
    c_n=
    \frac{x_{n+1}-x_n}{t_{n+1}-t_n}.
$$
We will show that these slopes are of finite total variation due to the exponential   rates \eqref{eq:sparse-choice}. Indeed, we have
$$
    c_n-c_\infty
    =
    \frac{t_{n+1}\left(\frac{x_{n+1}}{t_{n+1}}-c_{\infty}\right)-t_n\left(\frac{x_n}{t_n}-c_{\infty}\right)}{t_{n+1}-t_n}.
$$
Since $t_{n+1}\ge4t_n$, we have $\frac{t_{n+1}+t_n}{t_{n+1}-t_n} \lesssim 1$, which yields
$$
    |c_n-c_\infty|
    \le \frac{(t_{n+1}+t_n)\left( \left| \frac{x_{n+1}}{t_{n+1}}-c_{\infty} \right|+\left| \frac{x_{n}}{t_{n}}-c_{\infty} \right|\right)}{t_{n+1}-t_n}
    \le C2^{-n}.
$$
Therefore, 
\begin{equation}\label{eq:total-turn}
   \sup_{n\ge1}|c_n|<\infty \qquad \text{and} \qquad \sum_{n=1}^{\infty}|c_{n+1}-c_n|<\infty.
\end{equation}

\vspace{0.5em}
\noindent
\textbf{Step 3: Moving Lyapunov functional along the path.}

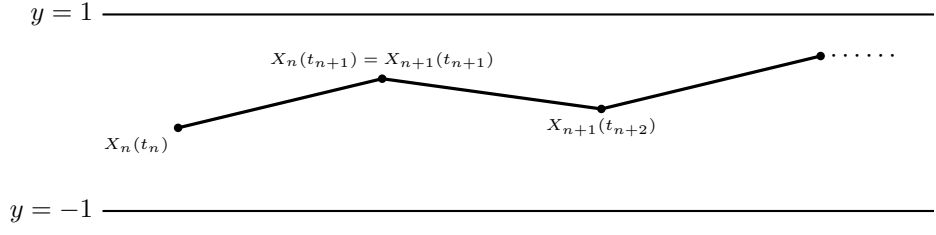
\begin{figure}
\begin{tikzpicture}[scale=1.0]
    \draw[thick] (0,1.3) -- (11,1.3);
    \draw[thick] (0,-1.3) -- (11,-1.3);
    \node[left] at (0,1.3) {$y=1$};
    \node[left] at (0,-1.3) {$y=-1$};

    \coordinate (P0) at (1.0,-0.2);
    \coordinate (P1) at (3.7,0.45);
    \coordinate (P2) at (6.6,0.05);
    \coordinate (P3) at (9.5,0.75);

    \draw[very thick] (P0) -- (P1) -- (P2) -- (P3);

    \fill (P0) circle (1.6pt);
    \fill (P1) circle (1.6pt);
    \fill (P2) circle (1.6pt);
    \fill (P3) circle (1.6pt);

    \node[below left]  at (P0) {\tiny $X_n(t_n)$};
    \node[above]       at (P1) {\tiny $X_n(t_{n+1})=X_{n+1}(t_{n+1})$};
    \node[below]       at (P2) {\tiny $X_{n+1}(t_{n+2})$};
    \node[ right] at (P3) {$\cdots \cdots$};
\end{tikzpicture}
\caption{Schematic illustration of zig-zag lines in the moving functional argument.}
\label{fig:zig-zag}
\end{figure}

Define on each $[t_n,t_{n+1}]$ the trajectories (the moving observers, see Figure \ref{fig:zig-zag})
$$
    X_n(t):=x_n+c_n(t-t_n).
$$
Then by their definitions
\begin{equation}\label{eq X_n(t_n)}
    X_n(t_n)=x_n,
    \qquad
    X_n(t_{n+1})=x_{n+1}.
\end{equation}
With these trajectories $X_n(t)$, we define the  functionals
\begin{equation}\label{eq def of M_n}
  \mathcal   M_n(t)
    :=
    \int_\Omega
    (y-c_n)H(x-X_n(t))\omega(t,z)\,dz,
\end{equation} 
\begin{equation}\label{eq def of D_n}
     \mathcal  D_n(t)
    :=
    \int_\Omega
    (y-c_n)^2H'(x-X_n(t))\omega(t,z)\,dz.
\end{equation} 
It follows from Proposition~\ref{prop:one-segment} that
\begin{equation}\label{eq:segment-integral}
    \frac12\int_{t_n}^{t_{n+1}}  \mathcal D_n(t)\,dt
    \le
     \mathcal M_n(t_{n+1})-  \mathcal M_n(t_n).
\end{equation}
We now sum \eqref{eq:segment-integral} over $n=1,\dots,N$ and telescope to obtain
\begin{equation}\label{eq differential eq for D}\begin{aligned}
    \frac12
    \sum_{n=1}^{N}
    \int_{t_n}^{t_{n+1}}\mathcal D_n(t)\,dt
    &\le
    \sum_{n=1}^{N}
    \left[
       \mathcal  M_n(t_{n+1})- \mathcal M_n(t_n)
    \right] \\
    &= \mathcal M_N(t_{N+1})-\mathcal M_1(t_1) \\
    &+
    \sum_{n=1}^{N-1}
    \left[
        \mathcal M_n(t_{n+1})-\mathcal M_{n+1}(t_{n+1})
    \right].
    \end{aligned}
    \end{equation}
Using \eqref{eq X_n(t_n)}, we see that
$$
\begin{aligned}
    \mathcal M_n(t_{n+1})-  \mathcal M_{n+1}(t_{n+1})
    &=
    (c_{n+1}-c_n)
    \int H(x-x_{n+1})\omega(t_{n+1},z)\,dz.
\end{aligned}
$$
By Lemma \ref{lem:property of H}, $\om \ge0$ and $\| \om(t) \|_{L^1}=\| \om_{in} \|_{L^1}  $, we obtain
\begin{equation}\label{eq:jump}
    | \mathcal M_n(t_{n+1})-  \mathcal M_{n+1}(t_{n+1})|
    \lesssim
   |c_{n+1}-c_n| \|H\|_{L^\infty} \int \om_{in}(z)\,dz \lesssim |c_{n+1}-c_n|.
\end{equation}
Moreover, \eqref{eq:total-turn} also implies the uniform bound
\begin{equation}\label{eq uniform bounde for M_n}
    | \mathcal M_n(t)|
    \lesssim \|H\|_{L^\infty} \int \om_{in}(z)\,dz \lesssim 1.
\end{equation}
Using \eqref{eq differential eq for D}, \eqref{eq:jump} and \eqref{eq uniform bounde for M_n}, we obtain
$$
\begin{aligned}
    \frac12
    \sum_{n=1}^{N}
    \int_{t_n}^{t_{n+1}} \mathcal D_n(t)\,dt
    &\lesssim 1
    +
    \sum_{n=1}^{N-1}|c_{n+1}-c_n|,
\end{aligned}
$$
where the implicit constant is independent of $t$ and $N$.
By \eqref{eq:total-turn} again, we can let  $N\to\infty$ and obtain
\begin{equation}\label{eq:finite-budget}
    \sum_{n=1}^{\infty}
    \int_{t_n}^{t_{n+1}} \mathcal D_n(t)\,dt
    <\infty.
\end{equation}
 
\vspace{0.5em}
\noindent
\textbf{Step 4: The final contradiction.}

Next we prove that \eqref{eq:finite-budget} leads to a contradiction. 

Setting $\rho = \frac{\delta}{16} \min\{1,\|\om_{in} \|_{L^\infty}^{-1} \}$, by volume constraints and boundedness of the vorticity, it follows from
\eqref{eq:bad-subsequence} that  
\begin{equation}\label{eq:rho-choice}
    \int_{\substack{|x-X_n(t_n )|\le 1 \\ |y-c_n|\ge   2\rho}}
    \omega(t_n,x,y)\,dxdy
    \ge 
    \delta-8\rho\|\om(t_n)\|_{L^\infty}  \ge
    \frac{\delta}{2}.
\end{equation}
For this $\rho>0$, we have
\begin{equation}\label{eq:Dt rho-choice}
\begin{aligned}
    \mathcal D_n(t)
    &=
    \int_{\R \times [-1,1]}
    (y-c_n)^2H'(x-X_n(t))\omega(t,x,y)\,dxdy \\
    &\ge C_{ \rho,\delta} \int_{\substack{|x-X_n(t)|\le 2 \\ |y-c_n|\ge    \rho}}  \omega(t,x,y)\,dxdy .
\end{aligned}
\end{equation}

To derive a contradiction between \eqref{eq:rho-choice} and \eqref{eq:Dt rho-choice}, we use a continuity argument.

For any $z=(x,y)$ and $s,t\ge0$, let
$
    \Phi_{t,s}(z)
$
be the Lagrangian flow map associated with the velocity field
$(y+u^x,u^y)$, namely
\begin{equation}\label{eq step 4 flow}
    \begin{cases}
\frac{d}{dt}\Phi_{t,s}(z)
    =
    \left(
        \Phi_{t,s}^y(z)+u^x(t,\Phi_{t,s}(z)),
        u^y(t,\Phi_{t,s}(z))
    \right), & \\ 
    \Phi_{s,s}(z)=z. &
    \end{cases}
\end{equation}
 
By \eqref{eq:velocity-bound}, \eqref{eq step 4 flow}, and \eqref{eq:total-turn},  we have bounded propagation speeds:
$$
   \sup_{ t,s,z}  \left|
    \frac{d}{dt} \Phi_{t,s}(z)\right|+ \sup_{t,n } \left| \frac{d}{dt}  X_n(t) 
    \right|<\infty.
$$
Therefore, together with \eqref{eq:sparse-choice}, there exists a small $\tau>0$ independent of $n$ such that $t_n+ \tau <t_{n+1}$ and for any $t\in[t_n,t_n+\tau]$, 
\begin{equation}\label{eq step 4 flow 1}
\begin{aligned}
 \Phi_{t,t_n}( E_n)   \subset \left\{z:    |x-X_n(t)|\le 2 ,\quad  |y-c_n|\ge   \rho  \right\},
\end{aligned}
\end{equation}
where the set $E_n : = \{z : |x-X_n(t_n)|\le 1  ,\quad  |y-c_n|\ge 2  \rho   \} $ is the integration domain in \eqref{eq:rho-choice}.

Since $\om \ge 0$, the flow is area preserving, and $ \om(t, \Phi_{t,t_n}(z))  = \om(t_n,z)$, we obtain from \eqref{eq step 4 flow 1}  and \eqref{eq:rho-choice} that  for every $t\in[t_n,t_n+\tau]$, 
\begin{equation}\label{eq step 4 flow 2}
\begin{aligned}
 &  \int_{\substack{|x-X_n(t)|\le 2 \\ |y-c_n|\ge   \rho}}
    \omega(t,x,y)\,dxdy \ge \int_{\substack{|x-X_n(t_n)|\le 1 \\ |y-c_n|\ge  2 \rho}} 
    \omega(t_n,x,y)\,dxdy
    \ge \frac{\delta}{2}.
\end{aligned}
\end{equation}

Immediately, by \eqref{eq step 4 flow 2} and \eqref{eq:Dt rho-choice}, we have for every
$t\in[t_n,t_n+\tau]$,
\begin{equation}\label{eq:Dt final}
\begin{aligned}
    \mathcal D_n(t)
    &\ge C_{ \rho,\delta},
\end{aligned}
\end{equation}
where $C_{ \rho,\delta}$ is a positive constant independent of $n$ and $t$. Therefore,
$$
    \sum_{n=1}^{\infty}
    \int_{t_n}^{t_{n+1}}  \mathcal D_n(t)\,dt
    \ge
    \sum_{n=1}^{\infty}
    \int_{t_n}^{t_n+\tau} \mathcal D_n(t)\,dt
    \ge
    \sum_{n=1}^{\infty}C_{ \rho,\delta}\tau
    =
    \infty.
$$
This contradicts \eqref{eq:finite-budget} and hence completes the proof.
\end{proof}

We remark that  no uniform quantitative damping rate can be obtained from the Yudovich bounds alone.

\begin{remark}[No uniform rate from  Yudovich bounds]\label{rmk:no uniform rate channel}
The convergence in Theorem~\ref{thm:intro-channel-positive-damping} cannot be uniform over classes controlled only by
$$
\|\omega_{ in}\|_{L^1(\Omega)},\qquad
\|\omega_{ in}\|_{L^\infty(\Omega)},\qquad
|\Supp\omega_{ in}|.
$$
Indeed, these quantities are preserved by the time-reversible equation \eqref{eq:euler}. Hence any nonzero admissible state may be prescribed at an arbitrarily large time and evolved backward to initial data with the same bounds.

Consequently, no function
$$
f(\ep,t)\longrightarrow0
\qquad\text{as }t\to\infty
$$
can bound $\|u(t)\|_{L^\infty(\Omega)}$ uniformly for all nonnegative smooth initial data satisfying
$$
\|\om_{in}\|_{L^1}
+\|\om_{in}\|_{L^\infty}
+|\Supp\om_{in}|
\leq \ep.
$$
\end{remark}

\subsection{Proof of Lyapunov monotonicity}\label{sec proof of Morawetz in channel}

\begin{proof}[Proof of Proposition \ref{prop:one-segment}]

Since the given trajectory $X(t)$ is linear on the interval $I$, by the standard Yudovich theory/weak formulation, $t \mapsto \mathcal M_X(t)$ is Lipschitz on $t\in I$. We focus on obtaining that, for all initial data with sufficiently small $\|\om_{in}\|_{L^\infty}$,
$$
    \frac{d}{dt} \mathcal M_{X}(t)\ge\frac12  \mathcal D_{X}(t),
    \qquad a.e. \,\,
    t\in I.
$$

We differentiate $\mathcal M_X$, using the definitions of $ \mathcal M_X$ and $ \mathcal D_X$, i.e. \eqref{eq def of M_I} and \eqref{eq def of D_I} to obtain that
$$
    \frac{d}{dt} \mathcal M_{X}(t)
    =
     \mathcal D_{X}(t)+ \mathcal N^x (t)+ \mathcal N^y (t),
$$
where the damping term $\mathcal D_{X}$ is given in \eqref{eq def of D_I} and the  nonlinear terms  are
$$
 \mathcal N^x (t)
    :=
    \int_\Om
    (y-c)H'(x-X(t))u^x\om\,dxdy
$$
and
$$
    \mathcal  N^y (t)
    :=
    \int_\Om
    H(x-X(t))u^y\om\,dxdy.
$$

\vspace{0.5em}
\noindent
\textbf{Step 1:  Estimates of $\mathcal N^y$.}

We first expand $\mathcal N^y$ using the Biot--Savart kernel and symmetrize the double integral ($K^y(z,z') = -K^y(z',z ) $ by \eqref{eq:channel Ky}):
\begin{equation}\label{eq:proof_prop:one-segment_0}
\begin{aligned}
 \mathcal N^y 
    &= 
    \iint  H(x-X)  
    K^y(z,z')
    \omega(z)\omega(z')\,dzdz' \\
    &=
    \frac12
    \iint
    \left(
        H(x-X)-H(x'-X)
    \right)
    K^y(z,z')
    \omega(z)\omega(z')\,dzdz' .
\end{aligned}
\end{equation}

Since the weight $H$ is monotone increasing, the sign of the kernel in \eqref{eq:proof_prop:one-segment_0}
$
\left(
        H(x-X)-H(x'-X)
    \right)
    K^y(z,z')
$ is the same as the sign of $\left(
         x- x' 
    \right)
    K^y(z,z')$. Then by \eqref{eq:anti Kx} from Lemma \ref{lem:kernel_change x x'}, 
\begin{equation}\label{eq:proof_prop:one-segment_1}
\left(
        H(x-X)-H(x'-X)
    \right)
    K^y(z,z') \ge 0 \qquad \text{for all $(z,z') \in \Om \times \Om$.}
\end{equation}

Combining \eqref{eq:proof_prop:one-segment_1} with the fact that $\om(z)\ge 0$, we have
\begin{equation}\label{eq:proof_prop:one-segment_2}
\begin{aligned}
 \mathcal N^y 
    &=
    \frac12
    \iint
    \left(
        H(x-X)-H(x'-X)
    \right)
    K^y(z,z')
    \omega(z)\omega(z')\,dzdz' \ge 0.
\end{aligned}
\end{equation}

\vspace{0.5em}
\noindent
\textbf{Step 2:  Splitting $\mathcal N^x$.}

For $ \mathcal N^x $, using the Biot--Savart kernel again, it suffices to estimate
$$
\begin{aligned}
    |\mathcal N^x |
    &\le
    \iint
    |y-c|H'(x-X(t))|K^x(z,z')|
    \omega(z)\omega(z')\,dzdz' .
\end{aligned}
$$
We decompose the region $\Om \times \Om =\left( \R \times [-1,1] \right)^2$ into 
$$
    \left( \R \times [-1,1] \right)^2
    =
    \{|y'-c|\le |y-c|\}
    \cup
    \{|y-c|<|y'-c|\}.
$$

\vspace{0.5em}
\noindent
\textbf{Step 3:  Estimates of $\mathcal N^x$ for $\{|y'-c|\le |y-c|\}$.}

In the region $\{|y'-c|\le |y-c|\}$, we first estimate the $dz'$ integral. Using  \eqref{eq K1} from Lemma \ref{lem:kernel_change x x'} with $\sigma=0$ yields
$$
\begin{aligned}
    &\int_{\{|y'-c|\le |y-c|\}}
    |K^x(z,z')|\omega(z')\,dz' \\
    &\le
    \|\om(t)\|_{L^\infty} \left|
        \{y'\in[-1,1]: |y'-c|\le |y-c|\}
    \right|
    \sup_{x,y,y'}
    \int_{\mathbb R}
    |K^x(x,y;x',y')|\,dx'  \\
    &\lesssim
    \|\om_{in}\|_{L^\infty} |y-c|.
\end{aligned}
$$
Inserting the above bound in $\mathcal N^x$ then gives
\begin{equation}\label{eq:proof_prop:one-segment_3}
\begin{aligned}
   & \iint_{\{|y'-c|\le |y-c|\}}
    |y-c|H'(x-X(t))|K^x(z,z')|
    \omega(z)\omega(z')\,dzdz' \\
    &  \lesssim
    \|\om_{in}\|_{L^\infty}
    \int_\Omega
    |y-c|^2 H'(x-X(t)) \omega(z)\,dz  \\
    &\lesssim
    \|\om_{in}\|_{L^\infty}\mathcal D_{X}.
\end{aligned}
\end{equation}

\vspace{0.5em}
\noindent
\textbf{Step 4:  Estimates of $\mathcal N^x$ for $\{|y-c|<|y'-c|\}$.} 

Now consider the region $\{|y-c|<|y'-c|\}$. In this region, we first estimate the $dz$ integral.

It follows from Lemma \ref{lem:property of H} that
\begin{equation}\label{eq:proof_prop:one-segment_4}
\begin{aligned}
    &\int_{\{|y-c|<|y'-c|\}}
    |y-c|H'(x-X(t))|K^x(z,z')|\omega(z)\,dz \\
    &\le
    \|\om_{in}\|_{L^\infty} |y'-c|
    \int_{\{|y-c|<|y'-c|\}}
    \left(
        \int_{\mathbb R}
        H'(x-X(t) )|K^x(x,y;x',y')|\,dx
    \right)dy \\
    &\lesssim
    \|\om_{in}\|_{L^\infty} |y'-c| H'(x'-X(t) )
    \left|
        \{y\in[-1,1]: |y-c|< |y'-c|\}
    \right| \\
    &\lesssim
    \|\om_{in}\|_{L^\infty} |y'-c|^2H'(x'-X(t)).
\end{aligned}  
\end{equation}
Multiplying by $\omega(z')$ and integrating in $z'$ gives
\begin{equation}\label{eq:proof_prop:one-segment_5}
\begin{aligned}
  &  \iint_{\{|y-c|<|y'-c|\}}
    |y-c|H'(x-X(t))|K^x(z,z')|
    \omega(z)\omega(z')\,dzdz' \\
    &\lesssim
    \|\om_{in}\|_{L^\infty}
    \int_\Omega
    |y'-c|^2H'(x'-X(t))\omega(z')\,dz' \\
    &\lesssim
    \|\om_{in}\|_{L^\infty}\mathcal D_{X}.
\end{aligned}
\end{equation}
Gathering the estimates \eqref{eq:proof_prop:one-segment_2}, \eqref{eq:proof_prop:one-segment_3}, and \eqref{eq:proof_prop:one-segment_5}, we have  that
$$
 \mathcal  N^y (t)\ge0 \quad \text{and} \quad | \mathcal  N^x (t)| \lesssim
    \|\om_{in}\|_{L^\infty}\mathcal D_{X},
$$
which gives the desired result once we choose $\|\om_{in}\|_{L^{\infty}}$ small enough.
\end{proof}

\section{Non-damping in the infinite channel}\label{sec:nodamping_channel}

The damping mechanism established in the previous section relies crucially on the positivity of the vorticity. We now show that, once this sign condition is relaxed, Hamiltonian conservation can prevent velocity decay in the infinite channel. This yields quantitative lower bounds on the velocity and low-Sobolev examples of non-damping.

For the 2D Euler equations \eqref{eq:euler} in the channel $\Omega = \R \times [-1,1]$,  the Hamiltonian
$$
    \mathcal H[\om]
    :=
    \int_\Omega |u|^2\,dz
    +
    \int_\Omega y^2\om(z)\,dz 
$$
is conserved for initial data in the Yudovich class $\omega_{in } \in L^1\cap L^\infty(\Omega)$.

\subsection{No damping for negative vorticity}

We begin by showing how Hamiltonian conservation prevents damping for non-positive vorticity, yielding a uniform-in-time lower bound on the velocity.

The result below is stated under a more general Hamiltonian condition, which is satisfied in particular by a large class of negative perturbations.

\begin{lemma}[A non-damping criterion]
\label{lem:channel-non-damping-criterion}
Let $\om_{in}\in L^1 \cap L^{\infty}(\Omega)$ such that
$$
    \mathcal H[\om_{in}]
    >
    \int_{\Om} ( \om_{in})_+\,dz,
$$
and let $\om(t)$ be the solution of \eqref{eq:euler} with initial vorticity $\om_{in}$. Then for any $1\le p\le \infty$, there exists $c_p>0$ such that
\begin{equation}
    \|u(t)\|_{L^p} \ge c_p
\end{equation}
for all $t\ge 0$.
 
\end{lemma}

\begin{proof}
Since $0\le y^2\le1$, by conservation of $\mathcal H$, we have
\begin{equation}\label{eq sec 41 L2 of u lower}
\begin{aligned}
\|u(t)\|_{L^2}^2=\mathcal H[\om_{in}]-\int_{\Om} y^2\om(t,z)\,dz \ge \mathcal H[\om_{in}]-\int_{\Om} ( \om_{in})_+\,dz>0.
    \end{aligned}
\end{equation}

Recall from \eqref{eq:velocity-bound} that, for any $1 \le q \le \infty$,  $$\|u(t)\|_{L^q}\lesssim_q 1.$$ 
Hence, using H\"older's inequality, for every $1\le p \le \infty$ we have 
$$
1 \lesssim \|u(t)\|_{L^2}^2 \le \|u(t)\|_{L^p}\|u(t)\|_{L^{p'}} \lesssim \|u(t)\|_{L^p}.
$$
This completes the proof.

\end{proof}

\subsection{Small Sobolev non-damping data}

We now use the preceding criterion to construct a class of non-damping examples that are arbitrarily small in $W^{s,p}$ for every $s<1+ \frac1p$.  For the stationary problem, the Sobolev threshold $s = 1+ \frac1p$ was identified for $p=2$  in the periodic case~\cite{MR2796139}, and more recently in the infinite case~\cite{2605.19971} for general $1\le p\le \infty$.

The specific construction below is based on the anisotropic scaling discovered in ~\cite{2605.19971}.

\begin{proposition}
\label{prop:small-Hs-negative-nondamping}
Let $s<1+\frac1p$ and $1\le p\le \infty$. For every $\ep>0$, there exists a smooth, compactly supported, non-positive vorticity
$$
    \om_{in}\in C_c^\infty(\Omega),
    \qquad
    \om_{in}\le 0,
    \qquad
    \|\om_{in}\|_{W^{s,p}(\Omega)}+\|\om_{in}\|_{L^1(\Omega)}+\|\om_{in}\|_{L^{\infty}(\Omega)}<\ep,
$$
such that for any $1 \le q \le +\infty$, there exists $c_q>0$ such that
\[
    \|u(t)\|_{L^q(\Omega)}\ge c_q
\]
holds for any $t \ge 0$.
\end{proposition}

\begin{proof}
First, fix a nonnegative function
$
    \eta\in C_c^\infty(\R^2)
$
supported near the origin and normalized by $\int_{\R^2} \eta\,dz=1$.

Then choose an exponent $\delta>0$ such that
\begin{equation}\label{eq:small-Hs-negative-nondamping 0}
0<\delta <1/3 \quad \text{and} \quad   s<  1+\frac1p    -  3 \delta.
\end{equation}  
For any parameter $0<\ep\ll1$, define the initial vorticity as
\begin{equation}\label{eq:small-Hs-negative-nondamping 1}
\om_{in}(z):=
-\ep^{1-3\delta } \eta\left(\frac{x}{\ep^{2\delta}},\frac{y}{\ep}\right).
\end{equation}  
It follows that $\om_{in}\le0$ and we can assume that $\ep>0$ is sufficiently small such that $\om_{in}\in C_c^\infty(\Omega)$ is well-defined in the channel.

From the definition \eqref{eq:small-Hs-negative-nondamping 1}, a direct calculation gives
$$
 \|\om_{in}\|_{L^1} \lesssim \ep^{2-\delta}, \qquad   \|\om_{in}\|_{L^\infty } \lesssim \ep^{1-3\delta} 
$$
and
$$
    \|\om_{in}\|_{W^{s,p}}
    \lesssim
    \ep^{1-3\delta }
    \ep^{\frac{2\delta+1}{p}}
    \ep^{-s}
    \le
    \ep^{ 1+ \frac{1}{p}- s -    3 \delta } .
$$
By \eqref{eq:small-Hs-negative-nondamping 0}, the exponents of $\ep$ above are all positive, and hence 
$\|\om_{in}\|_{W^{s,p}(\Omega)}+\|\om_{in}\|_{L^1(\Omega)}+\|\om_{in}\|_{L^{\infty}(\Omega)} \to 0$  as  $\ep \to 0$.

Next, we compute the initial Hamiltonian to apply the non-damping criterion from Lemma \ref{lem:channel-non-damping-criterion}. Let $u_{in}$ be the velocity generated by $\om_{in}$. Recall the Green  function for $\Delta$ on $\Omega$, 
\[
    -G_{\Omega}(z,z')
    =
    \frac{1}{2\pi}\log\frac{1}{|z-z'|}
    +
    H_{\Omega}(z,z'),
\]
where $H_{\Omega}$ is bounded away from the boundary.
Thus, 
\[
    -G_{\Omega}(z,z')
    \gtrsim
    \log\frac{1}{|z-z'|} \quad \text{for every $z,z'\in\Supp\om_{in}$}
\]
provided $\ep$ is sufficiently small. Since
\[
\int_{\Omega}|u_{in}|^2\,dz
=
-\iint_{\Omega\times\Omega}
G_{\Omega}(z,z')\om_{in}(z)\om_{in}(z')\,dzdz',
\]
we have
\begin{align*}
    \int_{\Omega} |u_{in}|^2 dz &\gtrsim \iint_{\Om\times \Om} \log \frac{1}{|z-z'|} \om_{in}(z)\om_{in}(z')\,dz\,dz' \\
    &=  \ep^{2+2\delta} \iint_{\R^2} \log \frac{1}{\ep^{2\delta}|z-z'|} \eta(x,\ep^{2\delta-1}y) \eta(x',\ep^{2\delta-1}y') \,dz\,dz' \\
    &= 2 \delta\ep^{2+2\delta} |\log \ep| \left(\int \eta(x,\ep^{2\delta-1}y)\,dz\right)^2\\
    &+\ep^{2+2\delta} \iint_{\R^2} \log \frac{1}{|z-z'|} \eta(x,\ep^{2\delta-1}y) \eta(x',\ep^{2\delta-1}y') \,dz\,dz' \\
    &:=I_1+I_2.
\end{align*}
For $I_1$, a direct computation using the scaling gives $I_1= 2\delta\ep^{4-2\delta}|\log \ep|$. For $I_2$, note that $|z-z'| \le 100$ holds in the support of $\eta(x,\ep^{2\delta-1}y)\eta(x',\ep^{2\delta-1}y')$. So we have $\log \frac{1}{|z-z'|} \ge -\log 100$, which gives
$$
I_2 \gtrsim -\ep^{4-2\delta}.
$$

Combining the estimates for $I_1$ and $I_2$, we have
$$
\int_{\Om} |u_{in}|^2 \,dz \gtrsim  \ep^{4-2\delta}|\log \ep| \quad \text{for all sufficiently small $\ep>0$}.
$$
Next, we estimate the second vertical moment. By the scaling in \eqref{eq:small-Hs-negative-nondamping 1}, we have
$$
    \int_\Omega y^2 \om_{in}(x,y)\,dxdy
    \approx
    -\ep^{4-\delta}.
$$
Therefore, for all sufficiently small $\ep>0$ the initial Hamiltonian is positive,
$$
    \mathcal H[\om_{in}]
    =
    \int_\Omega |u_{in}|^2\,dz
    +
    \int_\Omega y^2 \om_{in}\,dz
    >0.
$$
Since $\max\{\om_{in},0\}=0$, applying Lemma \ref{lem:channel-non-damping-criterion} we get the desired non-damping conclusion.
\end{proof}
\begin{remark}
Since the non-damping criterion is an open condition, one non-positive non-damping example gives rise to open neighborhoods of sign-changing ones.

\end{remark}

\subsection{Lower bounds for velocity}

Lastly, we prove velocity lower bounds from conservation of the
Hamiltonian. This complements the non-existence of a universal damping rate in Section \ref{sec:channel_damping}.

The nonlinear lower bound in \eqref{eq:couette-lower-bound} matches the linear upper bound established in Appendix~\ref{append:linear rate}, showing that nonlinear effects do not in general yield faster decay.  

\begin{theorem}[$L^p$ lower bounds for inviscid damping]
\label{thm:couette-lower-bound}
Let $\om(t)$ be a solution to \eqref{eq:euler} on
$\Omega=\mathbb R\times[-1,1]$. Assume
$
    \om_{in}\in   L^\infty(\Omega)
$
with compact support and
$$
    \int_\Omega\om_{in}\,dz-\mathcal H[\om_{in}]\neq0.
$$
Then, for every $1\le p\le\infty$,
\begin{equation}\label{eq:couette-lower-bound}
    \|u(t)\|_{L^p(\Omega)}
    \geq C_{ {in} }
    \langle t\rangle^{-1/p'} \quad \text{for all $t\ge0$,}
\end{equation}
 where
$
    \frac1p+\frac1{p'}=1.
$
The   constant $C_{ {in} }$ depends only on $
    \left|
        \int_\Omega\om_{in}\,dz-\mathcal H[\om_{in}]
    \right|$, $p$, $
    \|\om_{in}\|_{L^1(\Omega)\cap L^\infty(\Om)}
$ and
$
    \operatorname{diam}_x(\Supp\om_{in}).
$
\end{theorem}

\begin{proof}

By the translation invariance of the equation in the $x$-direction, we assume without loss of generality that
\begin{equation}\label{eq support of w in}
    \Supp_x\om_{in}
    \subset[-R_0,R_0].
\end{equation}
Using the velocity bound in the channel \eqref{eq:velocity-bound} and the fact that $\om$ is transported by the velocity $(y+u^x,u^y)$, we may fix a constant $V_0>0$ (depending only on the initial data) such that
 \begin{equation}\label{eq support of w}
    \Supp_x\om(t)
    \subset[-R_t,R_t] \quad \text{with} \,\, R_t: = 1+R_0 + V_0 t.
\end{equation}

Now choose a standard cutoff function $\eta \in C^\infty_c(\R)$ such that  $\eta(x)=1$ on $[-1,1]$ and $\eta(x)=0$ provided $|x| \ge 2$.

Setting $\eta_t(x)=\eta\left( \frac{x}{R_t} \right)$,  for any $t\ge 0$ we have
\begin{equation}\label{eq hamiltonian obs}
\begin{aligned}
    \int_\Omega\om_{in}\,dz-\mathcal H[\om_{in}]
    &=
    \int_\Omega(1-y^2)\om(t,z)\,dz
    -
    \|u(t)\|_{L^2(\Omega)}^2
    \\
    &=
    \int_\Omega\eta_t(x)(1-y^2)\om(t,z)\,dz
    -
    \|u(t)\|_{L^2(\Omega)}^2.
\end{aligned}
\end{equation}
Using
$
    \om=\partial_xu^y-\partial_yu^x,
$
 we obtain
$$
    \int_\Omega\eta_t(1-y^2)\om(t,z)\,dz
    =
    -\int_\Omega\eta_t'(1-y^2)u^y(t,z)\,dz
    -
    2\int_\Omega\eta_t y u^x(t,z)\,dz
$$
since $1-y^2$ vanishes on the boundary $y=\pm 1$. For any $1\le p\le\infty$,   H\"older's
inequality yields
$$
    \left|
        \int_\Omega\eta_t(1-y^2)\om(t,z)\,dz
    \right|
    \le
    C_p R_t^{1/p'}\|u(t)\|_{L^p(\Omega)}.
$$

For brevity let us define
$$
H_0 :=
    \left|
        \int_\Omega \om_{in}\,dz-\mathcal H[\om_{in}]
    \right|.
$$
Then \eqref{eq hamiltonian obs} gives
$$
H_0 \le
    \|u(t)\|_{L^2(\Omega)}^2
    +
    C_pR_t^{1/p'}\|u(t)\|_{L^p(\Omega)},
$$
which implies that either $
    \|u(t)\|_{L^2(\Omega)}^2\ge\frac{H_0}{2}
$ or $C_pR_t^{1/p'}\|u(t)\|_{L^p(\Omega)} \ge \frac{H_0}{2}$.

\vspace{0.5em}
\noindent
\textbf{Case 1: $\|u(t)\|_{L^2(\Omega)}^2\ge\frac{H_0}{2}
$.}  

In this case we show all $L^p$ norms of $ u$ have an $O(1)$ lower bound, which is stronger than \eqref{eq:couette-lower-bound}. Indeed, for $1\le p\le2$, interpolation together with \eqref{eq:velocity-bound} yields
$$
\begin{aligned}
    \|u(t)\|_{L^p(\Omega)}
    \ge
    \|u(t)\|_{L^2(\Omega)}^{2/p}
    \|u(t)\|_{L^\infty(\Omega)}^{1-2/p}
    \ge C_{in}
\end{aligned}
$$
and similarly, for $2\le p\le\infty$,  
$$
\begin{aligned}
    \|u(t)\|_{L^p(\Omega)}
    \ge
    \|u(t)\|_{L^2(\Omega)}^{2(p-1)/p}
    \|u(t)\|_{L^1(\Omega)}^{-(p-2)/p}
    \ge C_{in}.
\end{aligned}
$$

\vspace{0.5em}
\noindent
\textbf{Case 2: $ C_p R_t^{1/p'}\|u(t)\|_{L^p(\Omega)}
    \ge
    \frac{H_0}{2} $}. 
    
In this case the definition of $R_t$ implies the desired lower bound. Since
$$
    R_t
    =
    1+R_0+V_0 t
    \le C_{in } \langle t \rangle,
$$
we have
$$
    \|u(t)\|_{L^p(\Omega)}
    \ge
    \frac{H_0}{2C_p}R_t^{-1/p'} \ge  C_{in } \langle t \rangle^{-1/p'}.
$$
 
\end{proof}

\section{Damping for positive vorticity in the whole plane}\label{sec:damping_wholespace}

We now turn to the $\R^2$ case. For small nonnegative vorticity in $\R^2$, we first prove uniform evacuation and velocity decay as $t\to\infty$. Quantitative enhanced dispersion and growth of H\"older norms are established later in this section.

Throughout this section, we consider nonnegative solutions $\om(t)$ of
\eqref{eq:euler} with initial data satisfying
\[
    0\le\om_{in}\in L^1\cap L^\infty(\R^2),
    \qquad
    \int_{\R^2}|z|^2\om_{in}(z)\,dz<\infty.
\]
Transport by the area-preserving flow preserves the $L^1$ and
$L^\infty$ norms. The zero solution requires no proof, so we may assume
$\om_{in}\not\equiv0$ below.

For the dispersion and norm-growth estimates,
we first assume $L^1$ smallness, which is removed by spatial rescaling
in the proofs of Theorems~\ref{thm:intro-R2-positive-damping}
and~\ref{thm:intro-R2-Gevrey-norm-growth}.

\subsection{Hamiltonian conservation and moment bounds}

We first derive a logarithmic upper bound for the vertical second moment.
Recall that, under the assumptions above, the Hamiltonian
\begin{equation}\label{eq def of hami on R2 recall}
\mathcal H[\om(t)]
= - \frac{1}{4\pi}\iint_{\R^2\times\R^2}\log|z-z'|\om(t,z)\om(t,z')\,dz\,dz'
+\frac12\int_{\R^2}y^2\om(t,z)\,dz 
\end{equation}
is finite and conserved along the flow.

\begin{lemma}[Logarithmic vertical moment bound]
\label{lem:R2-logarithmic-y2-upper}
Let $\om(t)$ be the solution of \eqref{eq:euler} with initial vorticity $\om_{in}$ such that
\[
    0\le \om_{in}\in L^1(\R^2)\cap L^\infty(\R^2) \quad \text{and}\quad \int_{\R^2}|z|^2\om_{in}(z)\,dz<\infty.
\]
Then  there exists a universal constant $C>0$ such that for every
$t\ge0$,
\begin{equation}\label{eq sec5 upper bound for y2w}
\begin{aligned}
\int_{\R^2}y^2\om(t,z)\,dz
&\le C\,C_{in}^3 \Big( 1+ 
\log\langle t\rangle \Big),
\end{aligned}
\end{equation}
where
\begin{equation}\label{eq sec5 def of cin}
C_{in}=1+\int_{\R^2} |z|^2 \om_{in}(z)\,dz+\|\om_{in}\|_{L^1}+\|\om_{in}\|_{L^\infty}.
\end{equation}
\end{lemma}

\begin{proof}
By the Hamiltonian conservation law \eqref{eq def of hami on R2 recall},
\begin{equation}\label{eq sec5 upper bound for y2w 2}
\begin{aligned}
\int_{\R^2}y^2\om(t,z)\,dz
&=
\int_{\R^2}y^2\om_{in}(z)\,dz
+
\frac1{2\pi}
\iint_{\R^2\times\R^2}
\log|z-z'|
\om(t,z)\om(t,z')\,dzdz'
\\
&+
\frac1{2\pi}
\iint_{\R^2\times\R^2}
\log\frac{1}{|z-z'|}
\om_{in}(z)\om_{in}(z')\,dzdz'.
\end{aligned}
\end{equation}

Since $\om_{in}\ge0$, the last term in \eqref{eq sec5 upper bound for y2w 2} can be controlled by 
\begin{equation}\label{eq sec5 cin1}
\begin{aligned}
 &\iint_{\R^2\times\R^2}
\log\frac{1}{|z-z'|}
\om_{in}(z)\om_{in}(z')\,dzdz' \\
& \le \iint_{|z-z'| \le 1}
\log\frac{1}{|z-z'|}
\om_{in}(z)\om_{in}(z')\,dzdz' \\
&\le C \iint_{|z-z'|\le 1}
\frac{\om_{in}(z)\om_{in}(z')}{|z-z'|}
\,dzdz' \lesssim
C_{in}^2.
\end{aligned}
\end{equation}
Therefore, it remains to estimate the middle term in \eqref{eq sec5 upper bound for y2w 2}.

Since 
\(
    \log|z-z'|
    \lesssim
    1+\log\langle z\rangle+\log\langle z'\rangle
\)
and $\om(t)\ge0$, we obtain
\begin{equation}\label{eq sec5 cin2}\begin{aligned}
&\iint_{\R^2\times\R^2}
\log|z-z'|
\om(t,z)\om(t,z')\,dzdz'
\\
&\lesssim
C_{in}^2
+
C_{in}
\int_{\R^2}
\log\langle z\rangle\om(t,z)\,dz.
\end{aligned}
\end{equation}

To establish \eqref{eq sec5 upper bound for y2w}, it remains to bound the last
integral logarithmically in time. To this end, let $\Phi_t=(\Phi_t^x,\Phi_t^y)$
be the flow map associated with the velocity $(y,0)+u$, namely
\begin{equation}
\begin{cases}
\displaystyle
\frac{d}{dt}\Phi_t(z)
=
(\Phi_t^y(z),0)+u(t,\Phi_t(z)), &\\
\Phi_t(z)|_{t=0}=z. &
\end{cases}
\end{equation}
Recall that
\(
    \|u(t)\|_{L^\infty}
    \lesssim
    \|\om_{in}\|_{L^1}
    +
    \|\om_{in}\|_{L^\infty}.
\)
Integrating in time yields
\[
    |\Phi_t^y(z)|
    \lesssim
    |y|
    +
    C_{in}t
\]
and
\[
\begin{aligned}
    |\Phi_t^x(z)|
    \lesssim
    |x|+C_{in}t
    +
    t|y|
    +
    C_{in}t^2.
\end{aligned}
\]
Consequently,
\begin{equation}\label{eq:R2-flow-log-upper}
\begin{aligned}
\log\langle\Phi_t(z)\rangle
\lesssim
C_{in}
+
\log\langle z\rangle
+
\log\langle t\rangle
\end{aligned}
\end{equation}

Since the vorticity is transported by the area-preserving flow,
\eqref{eq:R2-flow-log-upper} gives
\begin{equation}\label{eq sec5 cin3}\begin{aligned}
\int_{\R^2}
\log\langle z\rangle\om(t,z)\,dz
&=
\int_{\R^2}
\log\langle\Phi_t(z)\rangle
\om_{in}(z)\,dz
\lesssim C_{in}^2 
\left( 1+
\log\langle t\rangle 
\right),
\end{aligned}
\end{equation}
which, together with \eqref{eq sec5 upper bound for y2w 2}, \eqref{eq sec5 def of cin},
\eqref{eq sec5 cin1} and \eqref{eq sec5 cin2}, completes the proof.
 
\end{proof}
 
We also need an upper bound for the total second moment.
By \eqref{eq:euler} and integration by parts,
\[
\frac{d}{dt}\int |z|^2\om(t,z)\,dz
=2\int xy\om(t,z)\,dz
+2\int z\cdot u(t,z)\om(t,z)\,dz.
\]
The Biot--Savart kernel is antisymmetric, so
\[
\begin{aligned}
2\int z\cdot u(t,z)\om(t,z)\,dz
&=2\iint z\cdot K(z,z')\om(t,z)\om(t,z')\,dzdz'
\\
&=\iint (z-z')\cdot K(z,z')\om(t,z)\om(t,z')\,dzdz'=0.
\end{aligned}
\]
Here the last equality follows from
$(z-z')\cdot K(z,z')=0$. Thus
\begin{equation}\label{eq:R2-total-moment-identity}
    \frac{d}{dt}\int |z|^2\om(t,z)\,dz
    =2\int xy\om(t,z)\,dz.
\end{equation}
By Cauchy--Schwarz and \eqref{eq sec5 upper bound for y2w},
\[
\begin{aligned}
\left(\int |z|^2\om(t,z)\,dz\right)^{1/2}
&\le
\left(\int |z|^2\om_{in}(z)\,dz\right)^{1/2}
+\int_0^t\left(\int y^2\om(s,z)\,dz\right)^{1/2}ds
\\
&\lesssim C_{in}^{3/2}
\left(1+t\sqrt{1+\log\langle t\rangle}\right).
\end{aligned}
\]
Consequently,
\begin{equation}\label{eq:R2-total-second-moment-upper}
    \int |z|^2\om(t,z)\,dz
    \le C C_{in}^3\langle t\rangle^2
    \bigl(1+\log\langle t\rangle\bigr).
\end{equation}

\subsection{A Lyapunov functional}

We next introduce a Lyapunov functional adapted to the Couette flow.
Under the linear Couette evolution, a trajectory satisfies
\[
    x(t)=x_0+ty_0,
    \qquad y(t)=y_0,
    \qquad y(t)-\frac{x(t)}t=-\frac{x_0}t.
\]
This suggests replacing $y$ in the Hamiltonian by $(y-x/t)$, leading to the time-weighted functional
\begin{equation}\label{eq:R2-quadratic-functional}
\begin{aligned}
\mathcal F(t)
:={}&\int_{\R^2}\left(y-\frac xt\right)^2\om(t,z)\,dz
\\
&-\frac1{2\pi}\iint_{\R^2\times\R^2}
\log\frac{|z-z'|}{t}\,
\om(t,z)\om(t,z')\,dzdz',
\qquad t>0.
\end{aligned}
\end{equation}
Equivalently, Hamiltonian conservation gives
\begin{equation}\label{eq:R2-functional-Hamiltonian}
\mathcal F(t)
=2\mathcal H[\om_{in}]
-\frac2t\int xy\om(t,z)\,dz
+\frac1{t^2}\int x^2\om(t,z)\,dz
+\frac{\|\om_{in}\|_{L^1}^2}{2\pi}\log t.
\end{equation}

We will prove an upper bound for $\mathcal F(t)$. The significance of
this bound is that a fixed amount of vorticity in a  ball contributes
a term of order $\log t$ to the positive part of the logarithmic term
in \eqref{eq:R2-quadratic-functional}, while  the
contribution of distant pairs is at most of order $\log\log t$ by the moment bounds above.

We first record a static estimate for the interaction term in the derivative
of \eqref{eq:R2-quadratic-functional}.

\begin{lemma}[A quadratic interaction estimate]
\label{lem:R2-quadratic-interaction}
Let $0\le\om\in L^1\cap L^\infty(\R^2)$ have finite second moment.
Then, for every $b\in\R^2$,
\begin{equation}\label{eq:R2-quadratic-interaction}
\iint\frac{\bigl(b\cdot(z-z')\bigr)^2}{|z-z'|^2}
\om(z)\om(z')\,dzdz'
\le4\pi\|\om\|_{L^\infty}
\int (b\cdot z)^2\om(z)\,dz.
\end{equation}
\end{lemma}

\begin{proof}
By rotation and scaling, it suffices to take $b=(0,1)$.
Split the integral into the regions $|y'|\le |y|$ and $|y'|>|y|$.
By symmetry we see that their contributions are equal, which gives
\[
\begin{aligned}
&\iint\frac{(y-y')^2}{|z-z'|^2}\om(z)\om(z')\,dzdz'
\\
&\quad=2\int\om(x,y)
\int_{-|y|}^{|y|}\int_{\R}
\frac{(y-y')^2\om(x',y')}{(x-x')^2+(y-y')^2}\,dx'\,dy'\,dxdy
\\
&\quad\le2\|\om\|_{L^\infty}
\int\om(x,y)
\int_{-|y|}^{|y|}\int_{\R}
\frac{(y-y')^2}{(x-x')^2+(y-y')^2}\,dx'\,dy'\,dxdy
\\
&\quad=2\pi\|\om\|_{L^\infty}
\int\om(x,y)\int_{-|y|}^{|y|}|y-y'|\,dy'\,dxdy
\\
&\quad=4\pi\|\om\|_{L^\infty}\int y^2\om(z)\,dz.
\end{aligned}
\]

\end{proof}

\begin{lemma}[Lyapunov estimate]
\label{lem:R2-quadratic-Lyapunov}
If $\|\om_{in}\|_{L^\infty}\le1/8$, then for a.e. $t\ge1$,
\begin{equation}\label{eq:R2-quadratic-Lyapunov}
\begin{aligned}
    \frac{d}{dt}\mathcal F(t)
    &\le-\frac1t\int\left(y-\frac xt\right)^2\om(t,z)\,dz
    +\frac{9\|\om_{in}\|_{L^1}^2}{4\pi t^2}.
\end{aligned}
\end{equation}
In particular,
$\mathcal F(t)+9\|\om_{in}\|_{L^1}^2/(4\pi t)$ is nonincreasing on
$[1,\infty)$, and
\begin{equation}\label{eq:R2-quadratic-functional-upper}
    \mathcal F(t)\le C C_{in}^3,
    \qquad t\ge1,
\end{equation}
where $C_{in}$ is defined in \eqref{eq sec5 def of cin} and $C>0$ is universal.
\end{lemma}

\begin{proof}
We first compute the time derivatives of the moments appearing in
\eqref{eq:R2-functional-Hamiltonian}. By \eqref{eq:euler} and
integration by parts,
\[
    \frac{d}{dt}\int x^2\om(t,z)\,dz
    =2\int xy\om(t,z)\,dz
    +2\int xu^x(t,z)\om(t,z)\,dz.
\]
Using the Biot--Savart law and then interchanging $z$ and $z'$, we have
\[
\begin{aligned}
\int xu^x(t,z)\om(t,z)\,dz
&=-\frac1{2\pi}\iint
\frac{x(y-y')}{|z-z'|^2}\om(t,z)\om(t,z')\,dzdz'
\\
&=-\frac1{4\pi}\iint
\frac{(x-x')(y-y')}{|z-z'|^2}
\om(t,z)\om(t,z')\,dzdz'.
\end{aligned}
\]
Consequently,
\begin{equation}\label{eq:R2-quadratic-moment-identities}
\begin{aligned}
&\frac{d}{dt}\int x^2\om(t,z)\,dz
\\
&\quad=2\int xy\om(t,z)\,dz
-\frac1{2\pi}\iint
\frac{(x-x')(y-y')}{|z-z'|^2}
\om(t,z)\om(t,z')\,dzdz'.
\end{aligned}
\end{equation}
The same calculation for the vertical second moment gives
\begin{equation}\label{eq:R2-y2-time-derivative}
\begin{aligned}
\frac{d}{dt}\int y^2\om(t,z)\,dz
&=2\int yu^y(t,z)\om(t,z)\,dz
\\
&=\frac1\pi\iint
\frac{y(x-x')}{|z-z'|^2}
\om(t,z)\om(t,z')\,dzdz'
\\
&=\frac1{2\pi}\iint
\frac{(x-x')(y-y')}{|z-z'|^2}
\om(t,z)\om(t,z')\,dzdz'.
\end{aligned}
\end{equation}

For the mixed moment, integration by parts yields
\[
\begin{aligned}
\frac{d}{dt}\int xy\om(t,z)\,dz
&=\int y^2\om(t,z)\,dz
+\int\bigl(yu^x(t,z)+xu^y(t,z)\bigr)\om(t,z)\,dz.
\end{aligned}
\]
For the last integral, similar calculation yields
\[
\begin{aligned}
&\int\bigl(yu^x(t,z)+xu^y(t,z)\bigr)\om(t,z)\,dz
\\
&\quad=\frac1{2\pi}\iint
\frac{x(x-x')-y(y-y')}{|z-z'|^2}
\om(t,z)\om(t,z')\,dzdz'
\\
&\quad=\frac1{4\pi}\iint
\frac{(x-x')^2-(y-y')^2}{|z-z'|^2}
\om(t,z)\om(t,z')\,dzdz'.
\end{aligned}
\]
Since
$(x-x')^2-(y-y')^2=|z-z'|^2-2(y-y')^2$ and the total mass is
conserved, it follows that
\begin{equation}\label{eq:R2-xy-time-derivative}
\begin{aligned}
\frac{d}{dt}\int xy\om(t,z)\,dz
&=\int y^2\om(t,z)\,dz
+\frac{\|\om_{in}\|_{L^1}^2}{4\pi}
\\
&\quad-\frac1{2\pi}\iint
\frac{(y-y')^2}{|z-z'|^2}
\om(t,z)\om(t,z')\,dzdz'.
\end{aligned}
\end{equation}

We now differentiate \eqref{eq:R2-functional-Hamiltonian}. This gives
\[
\begin{aligned}
\mathcal F'(t)
={}&\frac2{t^2}\int xy\om(t,z)\,dz
-\frac2t\frac{d}{dt}\int xy\om(t,z)\,dz
\\
&-\frac2{t^3}\int x^2\om(t,z)\,dz
+\frac1{t^2}\frac{d}{dt}\int x^2\om(t,z)\,dz
+\frac{\|\om_{in}\|_{L^1}^2}{2\pi t}.
\end{aligned}
\]
Substituting \eqref{eq:R2-quadratic-moment-identities} and
\eqref{eq:R2-xy-time-derivative}, the two terms involving
$\|\om_{in}\|_{L^1}^2/(2\pi t)$ cancel, and we obtain
\[
\begin{aligned}
\mathcal F'(t)
={}&-\frac2t\int y^2\om(t,z)\,dz
+\frac4{t^2}\int xy\om(t,z)\,dz
-\frac2{t^3}\int x^2\om(t,z)\,dz
\\
&+\frac1{\pi t}\iint
\frac{(y-y')^2}{|z-z'|^2}
\om(t,z)\om(t,z')\,dzdz'
\\
&-\frac1{2\pi t^2}\iint
\frac{(x-x')(y-y')}{|z-z'|^2}
\om(t,z)\om(t,z')\,dzdz'.
\end{aligned}
\]
Combining the first three terms into a square gives the identity
\begin{equation}\label{eq:R2-quadratic-functional-derivative}
\begin{aligned}
\mathcal F'(t)
={}&-\frac2t\int\left(y-\frac xt\right)^2\om(t,z)\,dz
\\
&+\frac1{\pi t}\iint
\frac{(y-y')^2}{|z-z'|^2}
\om(t,z)\om(t,z')\,dzdz'
\\
&-\frac1{2\pi t^2}\iint
\frac{(x-x')(y-y')}{|z-z'|^2}
\om(t,z)\om(t,z')\,dzdz'.
\end{aligned}
\end{equation}

We estimate the last two terms separately. First,
$2|(x-x')(y-y')|\le |z-z'|^2$ implies
\begin{equation}\label{eq:R2-quadratic-flux-bound}
\left|\iint
\frac{(x-x')(y-y')}{|z-z'|^2}
\om(t,z)\om(t,z')\,dzdz'\right|
\le\frac12\|\om_{in}\|_{L^1}^2.
\end{equation}
Next, using
\[
\frac{(y-y')^2}{|z-z'|^2}
\le2\frac{\bigl((y-y')-(x-x')/t\bigr)^2}{|z-z'|^2}
+\frac2{t^2}
\]
and applying Lemma~\ref{lem:R2-quadratic-interaction} with
$b=(-1/t,1)$, we find
\begin{equation}\label{eq:R2-angular-interaction-absorption}
\begin{aligned}
&\iint\frac{(y-y')^2}{|z-z'|^2}
\om(t,z)\om(t,z')\,dzdz'
\\
&\quad\le
2\iint\frac{\bigl((y-y')-(x-x')/t\bigr)^2}{|z-z'|^2}
\om(t,z)\om(t,z')\,dzdz'
+\frac{2\|\om_{in}\|_{L^1}^2}{t^2}
\\
&\quad\le
8\pi\|\om_{in}\|_{L^\infty}
\int\left(y-\frac xt\right)^2\om(t,z)\,dz
+\frac{2\|\om_{in}\|_{L^1}^2}{t^2}.
\end{aligned}
\end{equation}
Combining \eqref{eq:R2-quadratic-functional-derivative},
\eqref{eq:R2-quadratic-flux-bound}, and
\eqref{eq:R2-angular-interaction-absorption} yields
\[
\begin{aligned}
\mathcal F'(t)
\le{}&-\frac{2-8\|\om_{in}\|_{L^\infty}}{t}
\int\left(y-\frac xt\right)^2\om(t,z)\,dz
\\
&+\frac{\|\om_{in}\|_{L^1}^2}{4\pi t^2}
+\frac{2\|\om_{in}\|_{L^1}^2}{\pi t^3}.
\end{aligned}
\]
Since $\|\om_{in}\|_{L^\infty}\le1/8$ and $t\ge1$, this proves
\eqref{eq:R2-quadratic-Lyapunov}.

Finally, the negative part of the logarithmic interaction at $t=1$
is bounded by
\[
\begin{aligned}
&\iint_{|z-z'|<1}\log\frac1{|z-z'|}\om(1,z)\om(1,z')\,dzdz'
\\
&\quad\le\|\om_{in}\|_{L^\infty}\int_{\R^2}\om(1,z)
\left(\int_{|z-z'|<1}\log\frac1{|z-z'|}\,dz'\right)dz
\\
&\quad=\frac\pi2\|\om_{in}\|_{L^1}\|\om_{in}\|_{L^\infty}.
\end{aligned}
\]
By \eqref{eq:R2-total-second-moment-upper},
\[
\begin{aligned}
\mathcal F(1)
&\le2\int |z|^2\om(1,z)\,dz
+\frac14\|\om_{in}\|_{L^1}\|\om_{in}\|_{L^\infty}
\le C C_{in}^3.
\end{aligned}
\]
Integrating \eqref{eq:R2-quadratic-Lyapunov} from $1$ to $t$ gives
\[
\mathcal F(t)
\le\mathcal F(1)
+\frac{9\|\om_{in}\|_{L^1}^2}{4\pi}\int_1^t\frac{ds}{s^2}
\le C C_{in}^3,
\]
which proves \eqref{eq:R2-quadratic-functional-upper}.
\end{proof}

\subsection{Uniform evacuation and damping}

We now use the upper bound for $\mathcal F$ to exclude concentration as $t\to\infty$.

\begin{theorem}[Uniform evacuation and damping]
\label{thm:R2-global-damping} 
There exists a universal constant $\ep_0>0$ such that if
\[
    0\le\om_{in}\in L^1\cap L^\infty(\R^2),
    \qquad
    \|\om_{in}\|_{L^\infty}\le\ep_0,
    \qquad
    \int_{\R^2}|z|^2\om_{in}(z)\,dz<\infty,
\]
then the corresponding solution satisfies
\begin{equation}\label{eq:R2-strip-evacuation}
    \lim_{t\to\infty}\sup_{x_0\in\R}
    \int_{|x-x_0|\le1}\om(t,z)\,dz=0.
\end{equation}
Consequently, for every $2<p\le\infty$,
\begin{equation}\label{eq:R2-velocity-damping}
    \lim_{t\to\infty}\|u(t)\|_{L^p(\R^2)}=0.
\end{equation}
More precisely, for every $t\ge4$,
\begin{equation}\label{eq:R2-uniform-ball-bound}
\sup_{z_0\in\R^2}
\left(\int_{|z-z_0|\le1}\om(t,z)\,dz\right)^2
\le C C_{in}^3
\frac{1+\log\langle\log t\rangle}{\log(t/2)},
\end{equation}
where $C_{in}$ is defined in \eqref{eq sec5 def of cin} and $C>0$ is universal.
\end{theorem}

\begin{proof}
Take $\ep_0\le1/8$.

\vspace{0.5em}
\noindent
\textbf{Step 1: A bound for close pairs.}

For $z\neq z'$, the identity
\[
\log\frac{t}{|z-z'|}
=\log\left\langle\frac{t}{|z-z'|}\right\rangle
-\log\left\langle\frac{|z-z'|}{t}\right\rangle
\]
separates the two contributions to the logarithmic interaction.
The quadratic term in \eqref{eq:R2-quadratic-functional} is
nonnegative. Therefore \eqref{eq:R2-quadratic-functional-upper} gives
\begin{equation}\label{eq:R2-logarithmic-parts}
\begin{aligned}
\iint\log\left\langle\frac{t}{|z-z'|}\right\rangle\,
\om(t,z)\om(t,z')\,dzdz'
\le{}&C C_{in}^3
\\
&+\iint\log\left\langle\frac{|z-z'|}{t}\right\rangle\,
\om(t,z)\om(t,z')\,dzdz'.
\end{aligned}
\end{equation}
The case $\|\om_{in}\|_{L^1}=0$ is trivial, so assume $\|\om_{in}\|_{L^1}>0$.
Using $\log\langle s\rangle=\frac12\log(1+s^2)$ and applying
Jensen's inequality to the concave function $s\mapsto\log(1+s)$,
we obtain
\[
\begin{aligned}
&\iint\log\left\langle\frac{|z-z'|}{t}\right\rangle\,
\om(t,z)\om(t,z')\,dzdz'
\\
&\qquad\le\frac{\|\om_{in}\|_{L^1}^2}{2}
\log\left(1+\frac{2}{t^2\|\om_{in}\|_{L^1}}
\int |z|^2\om(t,z)\,dz\right).
\end{aligned}
\]
Since $1+\frac{a}{b}\le(1+a)(1+\frac1b)$ for $a\ge0$ and $b>0$, by \eqref{eq:R2-total-second-moment-upper}, for $t\ge4$ the right-hand side is bounded by
\[
\begin{aligned}
&\frac{\|\om_{in}\|_{L^1}^2}{2}
\log\left(1+\frac{2}{t^2}\int |z|^2\om(t,z)\,dz\right)
+\frac{\|\om_{in}\|_{L^1}^2}{2}
\log\left(1+\frac{1}{\|\om_{in}\|_{L^1}}\right)
\\
&\le\frac{\|\om_{in}\|_{L^1}^2}{2}
\log\left(1+C C_{in}^3\bigl(1+\log\langle t\rangle\bigr)\right)
+\frac{\|\om_{in}\|_{L^1}}{2}
\\
&\le C C_{in}^3\bigl(1+\log\langle\log t\rangle\bigr).
\end{aligned}
\]
Therefore,
\begin{equation}\label{eq:R2-close-pair-budget}
\iint\log\left\langle\frac{t}{|z-z'|}\right\rangle\,
\om(t,z)\om(t,z')\,dzdz'
\le C C_{in}^3\bigl(1+\log\langle\log t\rangle\bigr).
\end{equation}

\vspace{0.5em}
\noindent
\textbf{Step 2: Uniform evacuation.}

For any $z_0 \in \R^2$, since $|z-z'|\le2$ for $z,z'\in B_1(z_0)$, restricting the integral in \eqref{eq:R2-close-pair-budget} to $B_1(z_0)\times B_1(z_0)$ gives
\[
\log(t/2)
\left(\int_{B_1(z_0)}\om(t,z)\,dz\right)^2
\le C C_{in}^3\bigl(1+\log\langle\log t\rangle\bigr),
\]
uniformly in $z_0$. This proves \eqref{eq:R2-uniform-ball-bound} and hence
\begin{equation}\label{eq:R2-ball-evacuation}
    \lim_{t\to\infty}\sup_{z_0\in\R^2}
    \int_{B_1(z_0)}\om(t,z)\,dz=0.
\end{equation}

To obtain the strip estimate, let
\[
    E_{t,x_0}:=\{(x,y):|x-x_0|\le1,\ |y|\le\sqrt t\}.
\]
The diameter of $E_{t,x_0}$ is at most $3\sqrt t$ for $t\ge1$.
Thus, for all sufficiently large $t$,
\[
\frac14\log t
\left(\int_{E_{t,x_0}}\om(t,z)\,dz\right)^2
\le C C_{in}^3\bigl(1+\log\langle\log t\rangle\bigr),
\]
again uniformly in $x_0$. The remaining mass is bounded by
\[
    \int_{|y|>\sqrt t}\om(t,z)\,dz
    \le\frac1t\int y^2\om(t,z)\,dz
    \le C C_{in}^3\frac{1+\log\langle t\rangle}{t}.
\]
Both terms tend to zero, which proves
\eqref{eq:R2-strip-evacuation}.

\vspace{0.5em}
\noindent
\textbf{Step 3: Velocity decay.}

Apply Proposition~\ref{prop:damping-circulation-equivalence} to
$\om(t_n)$ for any sequence $t_n\to\infty$. The conservation of $L^1$ and
$L^\infty$ norms and \eqref{eq:R2-ball-evacuation} give
\eqref{eq:R2-velocity-damping} for every $2<p\le\infty$.
\end{proof}

\subsection{\texorpdfstring{$H$}{H}-center for the moment estimates}

We next introduce the $H$-center $X_0(t)$ for the moment estimates used
in the proof of enhanced dispersion.

Throughout this section, we define\footnote{The exponent $3/2$ is not essential; any exponent
$p\in(1,2)$ would work.}
\begin{equation}\label{eq: R2 def H}
    H(x):= c_H \int_0^{x} \langle s \rangle^{-3/2}\,ds \quad \text{for $x\in \R$},
\end{equation}
where the normalization constant $c_H>0$ is chosen so that
$\|H\|_{L^\infty}=1$.

Then we have
\begin{equation}\label{eq:R2-weight}
    |H(x)| \leq 1,
    \qquad
    H'(x)=c_H \langle x \rangle^{-3/2}>0.
\end{equation}
Moreover, since $H$ is increasing, we obtain
\begin{equation}\label{eq:R2-monotone-Ky}
    \bigl(H(x)-H(x')\bigr)(x-x')\ge0.
\end{equation}

\begin{lemma}\label{lem:R2-Xa-basic}
Let $\om(t)$ be the solution of \eqref{eq:euler} with initial vorticity $\om_{in}$ such that
\[
    0\le \om_{in}\in L^1(\R^2)\cap L^\infty(\R^2),
    \qquad
    \om_{in}\not\equiv0,
    \qquad
    \int_{\R^2}|z|^2\om_{in}(z)\,dz<\infty.
\]
For every $t \ge0$, there exists a unique $X_0(t)\in\R$ satisfying
\begin{equation}\label{eq:R2-definition-Xa}
    \int_{\R^2}
    H(x-X_0(t))\om(t,z)\,dz
    =
    0.
\end{equation}
Moreover, the trajectory $t\mapsto X_0(t)$ is locally Lipschitz and
\begin{equation}\label{eq:R2-Xa-derivative-identity}
    X_0'(t)
    =
    \frac{
        \int_{\R^2}
        H'(x-X_0(t))
        \bigl(y+u^x(t,z)\bigr)
        \om(t,z)\,dz
    }{
        \int_{\R^2}
        H'(x-X_0(t))
        \om(t,z)\,dz
    } \qquad \text{for a.e. $t\ge0$.}
\end{equation}
\end{lemma}
\begin{proof}
For each fixed $t\ge0$, the map
\begin{equation}\label{eq:lem:R2-Xa-basic 1}
    X
    \longmapsto
    \int_{\R^2}H(x-X)\om(t,z)\,dz
\end{equation}
is continuous and strictly decreasing: indeed, $ \om \ge0$ and $H'>0$ imply  
\begin{equation}\label{eq sec5 implicit X}
    \frac{d}{dX}
    \int_{\R^2}H(x-X)\om(t,z)\,dz
    =
    -
    \int_{\R^2}H'(x-X)\om(t,z)\,dz
    <0.
\end{equation}
Moreover, using $\|\om(t)\|_{L^1}=\|\om_{in}\|_{L^1}$,
the dominated convergence theorem  gives
\[
\lim_{X\to-\infty}
    \int_{\R^2}H(x-X)\om(t,z)\,dz
    =
     \|\om_{in}\|_{L^1}
\]
and
\[
\lim_{X\to+\infty}
    \int_{\R^2}H(x-X)\om(t,z)\,dz
    =
    - \|\om_{in}\|_{L^1},
\]
where we have used $H( - \infty) =-1$ and $H(+\infty) = 1$ by  definition.

Hence, by the monotonicity of \eqref{eq:lem:R2-Xa-basic 1}, there exists
a unique $X_0(t)$ such that $\int_{\R^2}
    H(x-X_0(t))\om(t,z)\,dz
    =
    0.$

Since the map $(t,X) \longmapsto \int_{\R^2}H(x-X)\om(t,z)\,dz$ is locally Lipschitz continuous on $(t,X) \in [0,\infty)\times \R $,  the local Lipschitz regularity of $X_0(t)$ follows directly
from \eqref{eq sec5 implicit X} and the implicit function theorem.

Next, we prove the time derivative formula for $X'_0(t)$. Differentiating \eqref{eq:R2-definition-Xa} in time, we obtain
\[
\begin{aligned}
0
&=
\frac{d}{dt}
\int_{\R^2}
H(x-X_0(t))\om(t,z)\,dz
\\
&=
\int_{\R^2}
H'(x-X_0(t))
\bigl(
    y+u^x(t,z)-X_0'(t)
\bigr)
\om(t,z)\,dz,
\end{aligned}
\]
which gives \eqref{eq:R2-Xa-derivative-identity}.

\end{proof}

For this moving center $X_0(t)$, we define the functionals
\begin{equation}\label{eq:R2-definition-Ma}
    \mathcal M_0(t)
    :=
    \int_{\R^2}
    yH(x-X_0(t))\om(t,z)\,dz,
\end{equation}
and
\begin{equation}\label{eq:R2-definition-Da}
    \mathcal D_0(t)
    :=
    \int_{\R^2}
    \bigl(y-X_0'(t)\bigr)^2
    H'(x-X_0(t))
    \om(t,z)\,dz.
\end{equation}

\begin{lemma}[Static absorption estimate]
\label{lem:static-absorption-global}
There exists a universal constant $C>0$ such that for any nonnegative function $\om\in L^1(\R^2)\cap L^\infty(\R^2)$ satisfying
$$
\int |z|^2 \om(z)\,dz <\infty
$$
and every $X,v\in\R$, the following holds
\begin{equation}
\begin{aligned}
&\left|
\int_{\R^2}
(y-v)H'(x-X)u^x(z)\om(z)\,dz
\right|
\\
&\le
C\bigl(
\|\om\|_{L^1}+\|\om\|_{L^\infty}
\bigr)
\int_{\R^2}
(y-v)^2H'(x-X)\om(z)\,dz.
\end{aligned}
\end{equation}
\end{lemma}

\begin{proof}
This is a purely static estimate. By a change of variables and the translation invariance of the planar Biot--Savart kernel, it suffices to consider $X=v=0$. So it remains to estimate
$$
    \iint_{\mathbb R^2\times\mathbb R^2}
    |y|H'(x)
    \frac{|y-y'|}{(x-x')^2+(y-y')^2}
    \om(z)\om(z')\,dzdz' .
$$
We split the domain into three regions
$
    \mathbb R^2\times\mathbb R^2
    =
    \mathcal R_1\cup\mathcal R_2\cup\mathcal R_3,
$
where
$$
    \mathcal R_1
    :=
    \{(z,z'):\ |y'|\le |y|\},
$$
$$
    \mathcal R_2
    :=
    \{(z,z'):\ |y|<|y'|\}
    \cap
    \Big(
        \{|x'|\le2\}\cup\{|x|\ge |x'|/2\}
    \Big),
$$
and
$$
    \mathcal R_3
    :=
    \{(z,z'):\ |y|<|y'|\}
    \cap
    \{|x'|>2\}
    \cap
    \{|x|<|x'|/2\}.
$$
\vspace{0.5em}
\noindent
\textbf{On $\mathcal R_1$:} 

Since $\om \ge0$, we integrate in $dz'$ first to obtain
\begin{equation}\label{eq N1 R2}
\begin{aligned}
&\iint_{|y'|\le |y|}
    |y|H'(x)
    \frac{|y-y'|}{(x-x')^2+(y-y')^2}
    \om(z)\om(z')\,dzdz'  \\
&\qquad\lesssim
    \|\om\|_{L^\infty}
    \int_{\mathbb R^2}
    |y|H'(x)\om(z)
    \int_{|y'|\le |y|}
    \int_{\mathbb R}
    \frac{|y-y'|}{(x-x')^2+(y-y')^2}
    \,dx'dy'\,dz  \\
&\qquad\lesssim
    \|\om\|_{L^\infty}
    \int_{\mathbb R^2}y^2H'(x)\om(z)\,dz.
\end{aligned}
\end{equation}

\vspace{0.5em}
\noindent
\textbf{On $\mathcal R_2$:} Since $H'(\xi) \approx \langle \xi \rangle^{-3/2}$, in this region we have
$
    H'(x)\lesssim H'(x').
$
It follows that
\begin{equation}\label{eq N2 R2}
\begin{aligned}
&\iint_{\mathcal R_2}
    |y|H'(x)
    \frac{|y-y'|}{(x-x')^2+(y-y')^2}
    \om(z)\om(z')\,dzdz'  \\
&\qquad\lesssim
    \|\om\|_{L^\infty}
    \int_{\mathbb R^2}H'(x')\om(z')
    \int_{|y|<|y'|}
    \int_{\mathbb R}
    |y|
    \frac{|y-y'|}{(x-x')^2+(y-y')^2}
    \,dxdy\,dz'  \\
&\qquad\lesssim
    \|\om\|_{L^\infty}
    \int_{\mathbb R^2}y'^2H'(x')\om(z')\,dz'.
\end{aligned}
\end{equation}

\vspace{0.5em}
\noindent
\textbf{On $\mathcal R_3$:}
Here
$|x-x'|\ge \frac{|x'|}{2}$ and $|y-y'|\le 2|y'|$, which gives
$$
    |y|
    \frac{|y-y'|}{(x-x')^2+(y-y')^2}
    \lesssim
    \frac{y'^2}{|x'|^2}.
$$
Therefore,
\begin{equation}\label{eq N3 R2}
\begin{aligned}
&\iint_{\mathcal R_3}
    |y|H'(x)
    \frac{|y-y'|}{(x-x')^2+(y-y')^2}
    \om(z)\om(z')\,dzdz'  \\
&\qquad\lesssim
    \left(
        \int_{\mathbb R^2}H'(x)\om(z)\,dz
    \right)
    \left(
        \int_{|x'|>2}
        \frac{y'^2}{|x'|^2}\om(z')\,dz'
    \right).
\end{aligned}
\end{equation}
Since $H'$ is bounded, the first factor in \eqref{eq N3 R2} can be bounded by
\begin{equation}\label{eq N4 R2}
    \int_{\mathbb R^2}H'(x)\om(z)\,dz
    \lesssim
    \|\om\|_{L^1}.
\end{equation}
Moreover, for $|x'|>2$ we have $|x'|^{-2}\lesssim H'(x'),$ which gives the bound for the second factor in \eqref{eq N3 R2}
\begin{equation}\label{eq N5 R2}
   \int_{|x'|>2}
        \frac{y'^2}{|x'|^2}\om(z')\,dz'
    \lesssim
 \int_{\R^2} y'^2H'(x')\om(z')\,dz'.
\end{equation}
Combining \eqref{eq N1 R2}, \eqref{eq N2 R2}, \eqref{eq N3 R2}, \eqref{eq N4 R2} and \eqref{eq N5 R2} yields the desired estimate.
\end{proof}

For any $X\in\R$, we also define
\begin{equation}\label{eq:R2-def-NyX}
\mathcal N_X^y[\om]
:=
\frac1{4\pi}
\iint_{\R^2\times\R^2}
\frac{
    \bigl(H(x-X)-H(x'-X)\bigr)(x-x')
}{
    |z-z'|^2
}
\om(z)\om(z')\,dzdz'.
\end{equation}
Since $H$ is increasing, we see that $\mathcal N_X^y[\om]\ge0$  provided $\om\ge0$.

\begin{lemma}[Lyapunov estimate for the moving center]
\label{lem:R2-Xa-Morawetz}
There exists a universal constant $\ep_0>0$ such that if 
\begin{equation}\label{eq:R2-smallness-for-Xa}
    \|\om_{in}\|_{L^1(\R^2)}
    +
    \|\om_{in}\|_{L^\infty(\R^2)}
    \le \ep_0 \quad \text{and}\quad \int_{\R^2}|z|^2 \om_{in}\,dz<\infty,
\end{equation}
then for a.e. $t\ge0$,   
\begin{equation}\label{eq:R2-Ma-monotonicity}
    \frac{d}{dt}\mathcal M_0(t)
    \ge
    \frac12\mathcal D_0(t)+\mathcal N_{X_0(t)}^y[\om(t)]  .
\end{equation}

\end{lemma}

\begin{proof}
Differentiating \eqref{eq:R2-definition-Ma} gives
\begin{equation}\label{eq:R2-Ma-first-identity}
\begin{aligned}
\frac{d}{dt}\mathcal M_0(t)
&=
\int_{\R^2}
yH'(x-X_0(t))
\left(
    y+u^x(t,z)-X_0'(t)
\right)
\om(t,z)\,dz \\
&\quad
+
\int_{\R^2}
H(x-X_0(t))
u^y(t,z)\om(t,z)\,dz.
\end{aligned}
\end{equation}

Using \eqref{eq:R2-Xa-derivative-identity} from Lemma \ref{lem:R2-Xa-basic}, we see that
\begin{equation}\label{eq:R2-Xa-cancellation}
    \int_{\R^2}
    H'(x-X_0(t))
    \bigl(
        y+u^x(t,z)-X_0'(t)
    \bigr)
    \om(t,z)\,dz
    =
    0.
\end{equation}
Subtracting $X'_0(t)$ times \eqref{eq:R2-Xa-cancellation} from
\eqref{eq:R2-Ma-first-identity} gives
\begin{equation}\label{eq:R2-Ma-second-identity}\begin{aligned}
\frac{d}{dt}\mathcal M_0(t)&=
\mathcal D_0(t)
+
\int_{\R^2}
\bigl(y-X_0'(t)\bigr)
H'(x-X_0(t))
u^x(t,z)
\om(t,z)\,dz
\\
&+
\int_{\R^2}
H(x-X_0(t))
u^y(t,z)\om(t,z)\,dz.
\end{aligned}
\end{equation}

Since $H$ is increasing and $\om(t)\ge0$, by symmetrization (exactly as in the proof of Proposition \ref{prop:one-segment}) we have
\begin{align}
\int_{\R^2}
H(x-X_0(t))
u^y(t,z)\om(t,z)\,dz
=
\mathcal N_{X_0(t)}^y[\om(t)] \ge 0.
\end{align}

On the other hand, Lemma~\ref{lem:static-absorption-global} applied
with $X=X_0(t)$ and $v=X_0'(t)$ gives
\begin{equation}\label{eq:R2-Ma-Nx-bound}\begin{aligned}
&\left|
\int_{\R^2}
\bigl(y-X_0'(t)\bigr)
H'(x-X_0(t))
u^x(t,z)\om(t,z)\,dz
\right|\\
&\le
C
\left(
    \|\om(t)\|_{L^1}
    +
    \|\om(t)\|_{L^\infty}
\right)
\mathcal D_0(t)   .
\end{aligned}
\end{equation}
Combining this with \eqref{eq:R2-Ma-second-identity}, we obtain the desired conclusion provided $\ep_0>0$ is sufficiently small.
\end{proof}

\subsection{Nonlinear moment growth}

We next show that the nonlinear evolution forces sharp growth of both
the vertical and horizontal moments. In particular, the vertical second
moment, which is conserved under the linear Couette evolution,
grows logarithmically under the nonlinear dynamics.

The $H$-center used below is inspired by
\cite{jyz2025superlineargradientgrowth2d}, where a related construction was
introduced in the study of perturbations of the Lamb dipole.

\begin{lemma}[Estimates for the $H$-center]
\label{lem:R2-H-center-estimates}
There exists a universal constant $C>0$ such that for any
\[
    0\le\om\in L^1(\R^2)\cap L^\infty(\R^2)
    \quad \text{and} \quad
    \om\not\equiv0,
\]
the following holds. 

If for some $X_0 \in\R$,
\begin{equation}\label{eq:R2-H-center-static}
    \int_{\R^2}H(x-X_0)\om(z)\,dz=0,
\end{equation}
then 
\begin{equation}\label{eq:R2-H-center-location}
    |X_0|+\int_{\R^2}|x-X_0|\om(z)\,dz
    \le
     \frac{C\left( 1+\|\om\|_{L^1} \right)^2}{\|\om\|_{L^1}}\left(1+\int_{\R^2}|x|\om(z)\,dz\right).
\end{equation} 
\end{lemma}
\begin{proof}

Throughout the proof, all implicit constants are universal.

It suffices to show that
\begin{equation}\label{eq sec5 control of X1}
|X_0 |
\lesssim
\frac{\left(1+\|\om\|_{L^1}\right)}{\|\om\|_{L^1}}
\left(
1+\int_{\R^2}|x|\om(z)\,dz
\right).
\end{equation}
If $|X_0|<2$, then \eqref{eq sec5 control of X1} holds  trivially. So we suppose first that $X_0\ge2$. Since $H(x-X_0)\le -H(1)$ when $x-X_0\le -1$,
it follows from \eqref{eq:R2-H-center-static} that
\[
\begin{aligned}
0
&=
\int_{x\le X_0-1}H(x-X_0)\om(z)\,dz
+
\int_{x>X_0-1}H(x-X_0)\om(z)\,dz
\\
&\le
-H(1)\int_{x\le X_0-1}\om(z)\,dz
+
\|H\|_{L^\infty}
\int_{x>X_0-1}\om(z)\,dz \\
&=-H(1)\|\om\|_{L^1}+\left( H(1)+\|H\|_{L^\infty} \right)\int_{x>X_0-1}\om(z)\,dz,
\end{aligned}
\]
which gives 
$$
\int_{x>X_0-1}\om(z)\,dz \gtrsim \|\om\|_{L^1}.
$$ 
Since $X_0\ge2$, we have $X_0-1\ge \frac12 X_0$ and hence
\[
\begin{aligned}
    \int_{\R^2}|x|\om(z)\,dz
    \ge
    \int_{x>X_0-1}|x|\om(z)\,dz
    \ge
    \frac12 X_0
    \int_{x>X_0-1}\om(z)\,dz
    \gtrsim
    X_0\|\om\|_{L^1},
\end{aligned}
\]
which proves \eqref{eq sec5 control of X1}.

Similarly, if $X_0\le-2$, then
\[
\begin{aligned}
0
&=
\int_{x<X_0+1}H(x-X_0)\om(z)\,dz
+
\int_{x\ge X_0+1}H(x-X_0)\om(z)\,dz
\\
&\ge
-\|H\|_{L^\infty}
\int_{x<X_0+1}\om(z)\,dz
+
H(1)
\int_{x\ge X_0+1}\om(z)\,dz
\\
&=
H(1)\|\om\|_{L^1}
-
\left(
    H(1)+\|H\|_{L^\infty}
\right)
\int_{x<X_0+1}\om(z)\,dz.
\end{aligned}
\]
Therefore,
\[
    \int_{x<X_0+1}\om(z)\,dz
    \gtrsim
    \|\om\|_{L^1},
\]
and by the same argument we obtain \eqref{eq sec5 control of X1}.
\end{proof}

The next lemma gives the coercivity needed for the sharp moment growth
estimates. Intuitively, a positive amount of vorticity cannot simultaneously
avoid vertical spread near the $H$-center and interaction across the center.

\begin{lemma} 
\label{lem:R2-H-center-coercivity}
There exists a universal constant $C>0$ such that for any
\[
    0\le\om\in L^1(\R^2)\cap L^\infty(\R^2)
    \quad \text{and} \quad
    \om\not\equiv0,
\]
the following holds. 

If for some $X \in\R$,
\begin{equation} 
    \int_{\R^2}H(x-X)\om(z)\,dz=0,
\end{equation}
then for any $v\in\R$,
\begin{equation}\label{eq:R2-H-center-coercivity}
\begin{aligned}
&\mathcal N_X^y[\om]
+
\int_{\R^2}
(y-v)^2H'(x-X)\om(z)\,dz\\
&\ge
\frac{C\|\om\|_{L^1}^3}{\left( 1+\|\om\|_{L^1} \right)\left(1+\|\om\|_{L^\infty}^2  \right)}\frac{1}{
1+\displaystyle\int_{\R^2}|x-X|\om(z)\,dz
},
\end{aligned}
\end{equation}
where $
\mathcal N_X^y[\om]
$ is the functional given by \eqref{eq:R2-def-NyX}. 
\end{lemma}

\begin{proof}
Throughout the proof, all implicit constants are universal.

By translation invariance of the planar Biot--Savart kernel, after the
change of variables
\[
    \widetilde\om(x,y)=\om(x+X,y+v),
\]
it is enough to consider the case $X=v=0$. Thus, we assume
\begin{equation}\label{eq:R2-H-center-origin}
    \int_{\R^2}H(x)\om(z)\,dz=0.
\end{equation}

 Let $\theta=\frac{H(1)}{2\left( H(1)+\|H\|_{L^\infty}\right)} >0 $. We divide the proof into two cases according to whether the vorticity mass in the strip $|x|< 1$ is greater or smaller than $\theta\|\om\|_{L^1}$.

\vspace{0.5em}
\noindent
\textbf{Case 1: $\int_{|x|<1}\om(z)\,dz
    \ge
    \theta\|\om\|_{L^1}.$}

Set $\rho=\frac{\theta\|\om\|_{L^1}}{8\|\om\|_{L^\infty}}$. Then
\[
\begin{aligned}
\int_{\substack{|x|<1\\|y|<\rho}}
\om(z)\,dz
&\le
4\rho\|\om\|_{L^\infty}
=
\frac12\theta\|\om\|_{L^1},
\end{aligned}
\]
and therefore
\[
    \int_{\substack{|x|<1\\|y|\ge\rho}}
    \om(z)\,dz
    \ge 
    \frac12 \theta\|\om\|_{L^1}.
\]
Since $H'(x)\gtrsim1$ on $[-1,1]$, we obtain
\begin{equation}\label{eq:R2-coercivity-case1-lower}
\begin{aligned}
    \int_{\R^2}y^2H'(x)\om(z)\,dz
    \gtrsim
    \int_{\substack{|x|<1\\|y|\ge\rho}}y^2H'(x)\om(z)\,dz
    \gtrsim
    \rho^2\|\om\|_{L^1}
    \gtrsim
    \frac{\|\om\|_{L^1}^3}{\|\om\|_{L^\infty}^2},
\end{aligned}
\end{equation}
which yields \eqref{eq:R2-H-center-coercivity} in this case.

\vspace{0.5em}
\noindent
\textbf{Case 2: $\int_{|x|<1}\om(z)\,dz<\theta\|\om\|_{L^1}$.}

In this case, we have
\begin{equation}\label{eq sec5 lower bound for |x| ge 1}
    \int_{x\le-1}\om(z)\,dz
    +
    \int_{x\ge1}\om(z)\,dz
    \ge
    (1-\theta)\|\om\|_{L^1}.
\end{equation}
We claim that 
\begin{equation}\label{eq:R2-left-right-mass}
    \int_{x\le-1}\om(z)\,dz
    \gtrsim
    \|\om\|_{L^1}
    \qquad \text{and} \qquad
    \int_{x\ge1}\om(z)\,dz
    \gtrsim
    \|\om\|_{L^1},
\end{equation}
where the implicit constants depend only on $H$ and are therefore universal.

Indeed, by \eqref{eq:R2-H-center-origin},
\begin{equation}\label{eq sec5 H(x)w decomposition}
    \int_{x\ge1}H(x)\om(z)\,dz
    =
    -
    \int_{x\le-1}H(x)\om(z)\,dz
    -
    \int_{|x|<1}H(x)\om(z)\,dz.
\end{equation}
Since $H(x)\ge H(1)$ when $x\ge1$, we obtain
\begin{equation}\label{eq sec5 H(1)}
    H(1)\int_{x\ge1}\om(z)\,dz
    \le
    \|H\|_{L^\infty}
    \left(
        \int_{x\le-1}\om(z)\,dz
        +
        \int_{|x|<1}\om(z)\,dz
    \right).
\end{equation}
On the other hand, by \eqref{eq sec5 lower bound for |x| ge 1},
\[
    \int_{x\ge1}\om(z)\,dz
    \ge
    (1-\theta)\|\om\|_{L^1}
    -
    \int_{x\le-1}\om(z)\,dz.
\]
Substituting this into \eqref{eq sec5 H(1)}, we get
\[
\begin{aligned}
H(1)
\left(
    (1-\theta)\|\om\|_{L^1}
    -
    \int_{x\le-1}\om(z)\,dz
\right)\le
\|H\|_{L^\infty}
\left(
    \int_{x\le-1}\om(z)\,dz
    +
    \theta\|\om\|_{L^1}
\right).
\end{aligned}
\]
Therefore, recalling that $\theta=\frac{H(1)}{2\left( H(1)+\|H\|_{L^\infty}\right)}$, we obtain
\[
\begin{aligned}
\left(
    H(1)+\|H\|_{L^\infty}
\right)
\int_{x\le-1}\om(z)\,dz
&\ge
\left(
    H(1)
    -
    \theta
    \left(
        H(1)+\|H\|_{L^\infty}
    \right)
\right)
\|\om\|_{L^1}
\\
&=
\frac12 H(1)\|\om\|_{L^1},
\end{aligned}
\]
which gives
\begin{equation}\label{eq sec5 claim1}
    \int_{x\le-1}\om(z)\,dz
    \gtrsim
    \|\om\|_{L^1}.
\end{equation}

Similarly, since $-H(x)\ge H(1)$ when $x \le -1$, it follows from \eqref{eq sec5 H(x)w decomposition} that
\[
    H(1)\int_{x\le-1}\om(z)\,dz
    \le
    \|H\|_{L^\infty}
    \left(
        \int_{x\ge1}\om(z)\,dz
        +
        \int_{|x|<1}\om(z)\,dz
    \right).
\]
Combining this with \eqref{eq sec5 lower bound for |x| ge 1}, the same argument yields
\[
    \int_{x\ge1}\om(z)\,dz
    \gtrsim
    \|\om\|_{L^1},
\]
which, together with \eqref{eq sec5 claim1}, completes the proof of
\eqref{eq:R2-left-right-mass}.

Next, for any $x \ge 1$ and $x' \le -1$, we claim that 
\begin{equation}\label{eq:R2-pointwise-coercivity}
\frac1{x-x'}
\lesssim
\frac{
    \bigl(H(x)-H(x')\bigr)(x-x')
}{
    (x-x')^2+(y-y')^2
}
+
y^2H'(x)
+
(y')^2H'(x').
\end{equation}
Assuming this has been proved, write $A:=\{(z,z')\in \R^2 \times \R^2 : x \ge 1 \,\,\, \text{and} \,\,\, x'\le -1\}$ for short. Multiplying \eqref{eq:R2-pointwise-coercivity} by
$\om(z)\om(z')$ and integrating over $A$,
we obtain
\begin{align}
\iint_{A}
\frac{\om(z)\om(z')}{x-x'}
\,dzdz'\lesssim
\mathcal N_0^y[\om]
+
\|\om\|_{L^1}
\int_{\R^2}y^2H'(x)\om(z)\,dz.
\label{eq:R2-coercivity-upper-interaction}
\end{align}
Meanwhile, by \eqref{eq:R2-left-right-mass},
\begin{equation}\label{eq sec5 claim2}\begin{aligned}
\|\om\|_{L^1}^4
&\lesssim
\left(
\iint_A
\om(z)\om(z')
\,dzdz'
\right)^2
\\
&\le
\left(
\iint_A
\frac{\om(z)\om(z')}{x-x'}
\,dzdz'
\right)
\left(
\iint_A
(x-x')
\om(z)\om(z')
\,dzdz'
\right).
\end{aligned}\end{equation}
Note that
\[
\begin{aligned}
\iint_A
(x-x')
\om(z)\om(z')
\,dzdz'
\le
\|\om\|_{L^1}
\left(
1+\int_{\R^2}|x|\om(z)\,dz
\right).
\end{aligned}
\]
Combining
\eqref{eq:R2-coercivity-upper-interaction} and \eqref{eq sec5 claim2}, we obtain
\[
\begin{aligned}
\frac{\|\om\|_{L^1}^3}{
1+\displaystyle\int_{\R^2}|x|\om(z)\,dz
}
&\lesssim
\mathcal N_0^y[\om]
+
\|\om\|_{L^1}
\int_{\R^2}y^2H'(x)\om(z)\,dz
\\
&\le
\left(
1+\|\om\|_{L^1}
\right)
\left(
\mathcal N_0^y[\om]
+
\int_{\R^2}y^2H'(x)\om(z)\,dz
\right),
\end{aligned}
\]
which proves \eqref{eq:R2-H-center-coercivity}.

It remains to prove \eqref{eq:R2-pointwise-coercivity} for $x\ge1$ and
$x'\le-1$. In this region,
\[
    x-x'\ge2 \qquad \text{and} \qquad H(x)-H(x')\gtrsim1.
\]
Assume
\(
    |y-y'|\le x-x',
\)
then 
\[
    \frac{1}{x-x'} \lesssim \frac{
    \bigl(H(x)-H(x')\bigr)(x-x')
}{
    (x-x')^2+(y-y')^2
}.
\]
If instead
\(
    |y-y'|>x-x',
\)
then
\[
    y^2+(y')^2
    \gtrsim
    (x-x')^2.
\]
Since
\[
    H'(x)\gtrsim (x-x')^{-3/2}
    \qquad \text{and} \qquad
    H'(x')\gtrsim (x-x')^{-3/2},
\]
it follows that
\[
    y^2H'(x)+(y')^2H'(x')
    \gtrsim
    (x-x')^{1/2}
    \gtrsim
    \frac1{x-x'},
\]
which proves \eqref{eq:R2-pointwise-coercivity} and hence completes the proof.

\end{proof}

We can now prove the enhanced dispersion bounds for positive vorticity.
Together with
Theorem~\ref{thm:R2-global-damping}, the following result yields
Theorem~\ref{thm:intro-R2-positive-damping} by spatial rescaling.

\begin{theorem}[Enhanced dispersion]
\label{thm:R2-positive-sharp-moment-growth}
There exists a universal constant $\ep_0>0$ such that the following
holds. For any $m_0>0$ and $R_0>0$, there exist constants
$c_0>0$    and $T_0>0$ depending only on $m_0$ and $R_0$,
such that for any initial data
\[
    0\le\om_{in}\in L^1(\R^2)\cap L^\infty(\R^2)
\]
satisfying
\[
    m_0\le\|\om_{in}\|_{L^1},
    \qquad
    \|\om_{in}\|_{L^1}
    +
    \|\om_{in}\|_{L^\infty}
    \le\ep_0, \qquad \int_{\R^2}|z|^2\om_{in}(z)\,dz\le R_0,
\]
for any $t\ge T_0$, the corresponding solution $\om(t)$  satisfies   the estimates
\begin{equation}\label{eq:R2-positive-y-critical-mass}
    \int_{y\ge c_0\sqrt{\log t}}
    \om(t,z)\,dz
    \ge c_0, \qquad 
    \int_{y\le-c_0\sqrt{\log t}}
    \om(t,z)\,dz
    \ge c_0
\end{equation}
and
\begin{equation}\label{eq:R2-positive-x-critical-mass}
    \int_{x\ge c_0t\sqrt{\log t}}
    \om(t,z)\,dz
    \ge c_0, \qquad
    \int_{x\le-c_0t\sqrt{\log t}}
    \om(t,z)\,dz
    \ge c_0.
\end{equation}
\end{theorem}

\begin{proof}
We first fix  $\ep_0 \leq 1$   so that   Lemma \ref{lem:R2-Xa-Morawetz} holds. Throughout the proof, all implicit constants in
$\lesssim_{m_0,R_0}$ and $\gtrsim_{m_0,R_0}$ depend only on $m_0$ and $R_0$.

\vspace{0.5em}
\noindent
\textbf{Step 1: Upper bounds for the second moments.}

 By Lemma~\ref{lem:R2-logarithmic-y2-upper} and the assumptions on
$\om_{in}$, we have
\begin{equation}\label{eq:R2-moment-growth-y2-upper}
    \int_{\R^2}y^2\om(t,z)\,dz
    \lesssim_{R_0}
    1+\log\langle t\rangle.
\end{equation}

For the horizontal second moment, a direct calculation gives
\[
\frac{d}{dt}
\int_{\R^2}x^2\om(t,z)\,dz
=
2\int_{\R^2}
x\bigl(y+u^x(t,z)\bigr)\om(t,z)\,dz.
\]
Therefore, it follows from \eqref{eq:R2-moment-growth-y2-upper} that
\[
\begin{aligned}
\frac{d}{dt}
\left(
    \int_{\R^2}x^2\om(t,z)\,dz
\right)^{\frac12}
&\le
\left(
    \int_{\R^2}
    \bigl(y+u^x(t,z)\bigr)^2
    \om(t,z)\,dz
\right)^{\frac12}
\\
&\lesssim_{R_0}
\sqrt{1+\log\langle t\rangle}.
\end{aligned}
\]
Thus,
\begin{equation}\label{eq:R2-x2-upper}
    \int_{\R^2}x^2\om(t,z)\,dz
    \lesssim_{R_0}
    1+t^2\left(1+\log\langle t\rangle\right),
\end{equation}
and by Cauchy--Schwarz,
\[
    \int_{\R^2}|x|\om(t,z)\,dz
    \lesssim_{R_0}
    1+t\sqrt{1+\log\langle t\rangle}.
\]
Let $X_0(t)$ be the point defined in Lemma \ref{lem:R2-Xa-basic}. By Lemma \ref{lem:R2-H-center-estimates}, we see that
\begin{equation}\label{eq:R2-centered-x-first-upper}
    \int_{\R^2}|x-X_0(t)|\om(t,z)\,dz
    \lesssim_{m_0,R_0}
    1+t\sqrt{1+\log\langle t\rangle}.
\end{equation}

\vspace{0.5em}
\noindent
\textbf{Step 2: Two vertical mass bounds.}

Since the initial data $\om_{in}$ satisfy the assumptions in Lemma~\ref{lem:R2-Xa-Morawetz}, we have
\begin{equation}\label{eq:thm:R2-growth M0}
\frac{d}{dt}\mathcal M_0(t)
\gtrsim
    \mathcal D_0(t)
    +
    \mathcal N_{X_0(t)}^y[\om(t)].
\end{equation}

Applying Lemma~\ref{lem:R2-H-center-coercivity} and
\eqref{eq:R2-centered-x-first-upper}, and using the conservation of the
$L^1$ and $L^\infty$ norms, we obtain
\begin{equation}\label{eq:thm:R2-growth M0 Ny}
\begin{aligned}
\mathcal D_0(t)+ \mathcal N_{X_0(t)}^y[\om(t)]
&\gtrsim_{m_0}
\frac{1}{
1+
\int_{\R^2}|x-X_0(t)|\om(t,z)\,dz
}\\
&\gtrsim_{m_0,R_0}
    \frac1{
    1+t\sqrt{1+\log\langle t\rangle}
    }. 
\end{aligned}
\end{equation}
 
By \eqref{eq:thm:R2-growth M0} and \eqref{eq:thm:R2-growth M0 Ny}, we can find $T_1>0$ depending only on $m_0$ and $R_0$ such that
\begin{equation}\label{eq:R2-M0-lower}
    \mathcal M_0(t)
    \gtrsim_{m_0,R_0}
    \sqrt{\log t} \qquad \text{for all $t \ge T_1$}.
\end{equation}
Since
\begin{equation}\label{eq sec5 upper bound for M0}
    \mathcal M_0(t)
    \le
    \|H\|_{L^\infty}
    \int_{\R^2}|y|\om(t,z)\,dz,
\end{equation}
we obtain
\begin{equation}\label{eq sec5 yw intergral decom}
    \int_{y>0}y\om(t,z)\,dz+\int_{y<0} |y|\om(t,z)\,dz=\int_{\R^2}|y|\om(t,z)\,dz
    \gtrsim_{m_0,R_0}
    \sqrt{\log t}.
\end{equation}
On the other hand, recall that
\begin{equation}\label{eq:R2-y-momentum}
    \int_{\R^2}y\om(t,z)\,dz
    =
    \int_{\R^2}y\om_{in}(z)\,dz,
\end{equation}
we have 
$$
\int_{y>0}y\om(t,z)\,dz-\int_{y<0} |y|\om(t,z)\,dz=\int_{\R^2}y\om_{in}(z)\,dz,
$$
which, together with \eqref{eq sec5 yw intergral decom}, implies that
\begin{equation}
\begin{aligned}
    \int_{y>0}y\om(t,z)\,dz
    \ge
    C_{m_0,R_0}\sqrt{\log t}+\frac12\int_{\R^2}y\om_{in}(z)\,dz
\end{aligned}
\end{equation}
and
\begin{equation}
\begin{aligned}
    \int_{y<0}|y|\om(t,z)\,dz
    \ge
    C_{m_0,R_0}\sqrt{\log t}-\frac12\int_{\R^2}y\om_{in}(z)\,dz.
\end{aligned}
\end{equation}
Thus, we can choose $T_2$ large enough depending on $m_0$ and $R_0$ such that 
\begin{equation}\label{eq:R2-positive-negative-y-first}
    \int_{y>0}y\om(t,z)\,dz \gtrsim_{m_0,R_0} \sqrt{\log t} \quad \text{and} \quad \int_{y<0}|y|\om(t,z)\,dz\gtrsim_{m_0,R_0} \sqrt{\log t} 
\end{equation}
for all $t \ge T_2$.

Next, choose $c_1>0$ small enough depending only on $m_0$ and $R_0$ such that for all $c \le c_1$,
\begin{equation}\label{eq:R2-positive-negative-y-first 2}
\begin{aligned}
\int_{y\ge c\sqrt{\log t}}
y\om(t,z)\,dz
\ge
\int_{y>0}y\om(t,z)\,dz
-
c\sqrt{\log t}\,
\|\om_{in}\|_{L^1}
\gtrsim_{m_0,R_0}
\sqrt{\log t}.
\end{aligned}
\end{equation}
 Note that
$$
\int_{y \ge c\sqrt{\log t}} y\om(t,z)\,dz \le \left(\int_{y \ge c\sqrt{\log t}} y^2\om(t,z)\,dz\right)^{\frac12}\left(\int_{y \ge c\sqrt{\log t}} \om(t,z)\,dz\right)^{\frac12};
$$
it follows from \eqref{eq:R2-positive-negative-y-first 2} and \eqref{eq:R2-moment-growth-y2-upper} that
\begin{equation}\label{eq sec5 t ge logt1}
    \int_{y\ge c\sqrt{\log t}}
    \om(t,z)\,dz
    \gtrsim_{m_0,R_0}1 \qquad \text{for all $t \ge T_2$ and $c\le c_1$}.
\end{equation}
The same argument gives
\begin{equation}\label{eq sec5 t ge logt2}
    \int_{y\le-c\sqrt{\log t}}
    \om(t,z)\,dz
    \gtrsim_{m_0,R_0}1 \qquad \text{for all $t \ge T_2$ and $c\le c_1$}.
\end{equation}

\vspace{0.5em}
\noindent
\textbf{Step 3: The two horizontal mass bounds.}
Define
\begin{equation}\label{eq:R2-def-F}
    F(x)
    :=
    \int_0^x H(s)\,ds
\end{equation}
and
\begin{equation}\label{eq:R2-def-Q}
    \mathcal Q(t)
    :=
    \int_{\R^2}
    F(x-X_0(t))
    \om(t,z)\,dz.
\end{equation}
Since $H$ is odd and increasing, $F$ is even and
\begin{equation}\label{eq:R2-F-upper}
    0\le F(x)\lesssim |x|.
\end{equation}
Using the defining property of $X_0(t)$, a direct calculation gives
\begin{equation}\label{eq:R2-Q-prime}
\frac{d}{dt}\mathcal Q(t)
=
\mathcal M_0(t)
+
\int_{\R^2}
H(x-X_0(t))
u^x(t,z)\om(t,z)\,dz.
\end{equation}
Moreover,
\[
\begin{aligned}
\left|
\int_{\R^2}
H(x-X_0(t))
u^x(t,z)\om(t,z)\,dz
\right|
&\le
\|H\|_{L^\infty}
\|u(t)\|_{L^\infty}
\|\om_{in}\|_{L^1}
\lesssim1.
\end{aligned}
\]
Therefore, by \eqref{eq:R2-M0-lower},    there exists $T_3 \ge 0$ depending only on $m_0$ and $R_0$ such that
\begin{equation}\label{eq:R2-Q-lower}
    \mathcal Q(t)
    \gtrsim_{m_0,R_0}
    t\sqrt{\log t} \qquad \text{for all $t \ge T_3$}.
\end{equation}
Using \eqref{eq:R2-F-upper} and \eqref{eq:R2-H-center-location} from Lemma \ref{lem:R2-H-center-estimates}, we obtain
\[
\begin{aligned}
\mathcal Q(t)
\lesssim
\int_{\R^2}|x-X_0(t)|\om(t,z)\,dz
\lesssim_{m_0}
1+\int_{\R^2}|x|\om(t,z)\,dz,
\end{aligned}
\]
which, together with \eqref{eq:R2-Q-lower}, yields
\begin{equation}\label{eq:R2-x-first-lower}
    \int_{x>0}x\om(t,z)\,dz+\int_{x<0}|x|\om(t,z)\,dz=\int_{\R^2}|x|\om(t,z)\,dz
    \gtrsim_{m_0,R_0}
    t\sqrt{\log t}
\end{equation}
for all $t \ge T_3$. Recall the identity
\begin{equation}\label{eq:R2-x-momentum}
\begin{aligned}
    \int_{x>0}x\om(t,z)\,dz-\int_{x<0}|x|\om(t,z)\,dz&=\int_{\R^2}x\om(t,z)\,dz\\
   & =
    \int_{\R^2}x\om_{in}(z)\,dz
    +
    t\int_{\R^2}y\om_{in}(z)\,dz.
    \end{aligned}
\end{equation}
Since
\[
    \left|
    \int_{\R^2}x\om_{in}(z)\,dz
    \right|
    +
    \left|
    \int_{\R^2}y\om_{in}(z)\,dz
    \right|
    \lesssim_{R_0}1,
\]
combining \eqref{eq:R2-x-first-lower} and
\eqref{eq:R2-x-momentum}, we can choose $T_4$ large enough depending on $m_0$ and $R_0$ such that 
\begin{equation}\label{eq:R2-positive-negative-x-first}
\begin{aligned}
    \int_{x>0}x\om(t,z)\,dz
    \gtrsim_{m_0,R_0}
    t\sqrt{\log t}
\end{aligned}
\end{equation}
and
\begin{equation}
\begin{aligned}
    \int_{x<0}|x|\om(t,z)\,dz
    \gtrsim_{m_0,R_0}
    t\sqrt{\log t}
\end{aligned}
\end{equation}
for all $t \ge T_4$.

Using \eqref{eq:R2-x2-upper} and arguing as in Step 2, we can find
$c_2>0$ depending only on $m_0$ and $R_0$ such that
\begin{equation}\label{eq sec5 t ge logt3}
    \int_{x\ge ct\sqrt{\log t}}
    \om(t,z)\,dz
    \gtrsim_{m_0,R_0}1 \qquad \text{for all $t \ge T_4$ and $c \le c_2$}
\end{equation}
and
\begin{equation}\label{eq sec5 t ge logt4}
    \int_{x\le-ct\sqrt{\log t}}
    \om(t,z)\,dz
    \gtrsim_{m_0,R_0}1 \qquad \text{for all $t \ge T_4$ and $c \le c_2$}.
\end{equation}

Finally, by \eqref{eq sec5 t ge logt1}, \eqref{eq sec5 t ge logt2},
\eqref{eq sec5 t ge logt3} and \eqref{eq sec5 t ge logt4}, there exist
constants $c_0>0$ and $T_0>0$, depending only on $m_0$ and $R_0$, such
that \eqref{eq:R2-positive-y-critical-mass} and
\eqref{eq:R2-positive-x-critical-mass} hold.

\end{proof}

\begin{proof}[Proof of Theorem~\ref{thm:intro-R2-positive-damping}]
Choose $\ep_0>0$ sufficiently small so that $\ep_0$ satisfies the smallness condition in Theorem~\ref{thm:R2-global-damping} and $2\ep_0$ satisfies that in Theorem~\ref{thm:R2-positive-sharp-moment-growth}.
For $a\ge1$, the rescaled solution
\[
    \om_a(t,z):=\om(t,az)
\]
has unchanged $L^\infty$ norm, with $L^1$ norm and second moments multiplied by
$a^{-2}$ and $a^{-4}$, respectively. Taking $a$ large enough,
Theorem~\ref{thm:R2-positive-sharp-moment-growth} applies; scaling back
gives \eqref{eq:thm:intro-R2-positive-damping_4} for large $t$ and hence $t\ge0$.

Evacuation and damping follow from Theorem~\ref{thm:R2-global-damping}.
\end{proof}

\subsection{Growth of H\"older norms}

We now combine the enhanced dispersion in
Theorem~\ref{thm:R2-positive-sharp-moment-growth} with the
filamentation argument in~\cite{MR4350517} to obtain
norm growth for a large class of initial data.

Theorem~\ref{thm:intro-R2-Gevrey-norm-growth}  follows from Theorem~\ref{thm:R2-Gevrey-norm-growth} below by the same rescaling argument as in the proof of Theorem~\ref{thm:intro-R2-positive-damping} and the Gevrey bound 
\[
    \|f(\cdot/a)\|_{\mathcal G^\lambda_\sigma}
    \le a\|f\|_{\mathcal G^\lambda_\sigma},\qquad a\ge1.
\]

\begin{theorem}
\label{thm:R2-Gevrey-norm-growth}
Let $\ep_0>0$ be the universal constant in
Theorem~\ref{thm:R2-positive-sharp-moment-growth}. For any $0<\sigma<1$, $\lambda>0$, $\delta>0$ and any nonzero initial vorticity $\om_{b}$ satisfying
\[
    0\le \om_{b}\in C_c^\infty(\R^2),
    \qquad     
    \|\om_{b}\|_{L^1(\R^2)}
    +
    \|\om_{b}\|_{L^\infty(\R^2)}
    \le \frac{\ep_0}{2},
\]
there exists a nonnegative perturbation
\(
 \om_{p}\in C_c^\infty(\R^2)
\)
with
\begin{equation}\label{eq:R2-Gevrey-small-perturbation}
\begin{aligned}
\| \om_{p} \|_{\mathcal G^\lambda_\sigma} 
\le \delta,
\end{aligned}
\end{equation}
such that the solution $ \om(t)$     with initial vorticity $ \om_{in}: = \om_{b} +\om_{p} $ satisfies for any $s>0$, 
\begin{equation}\label{eq:R2-Cxs-growth}
    \| \om(t)\|_{C_x^s}
    \gtrsim_{s, \om_{in}}
    (\log \langle t\rangle)^{\frac{s}{2}}
\end{equation}
and
\begin{equation}\label{eq:R2-Cys-growth}
    \|  \om(t)\|_{C_y^s}
    \gtrsim_{s, \om_{in}}
    \left(
        t\sqrt{\log \langle t\rangle}
    \right)^s.
\end{equation} 
\end{theorem}
\begin{proof}

\vspace{0.5em}
\noindent
\textbf{Step 1: Construction of perturbation.}

For any $0<\sigma<1$, fix $\sigma<\sigma'<1$ and a nonnegative radial profile $G \in C^\infty (\R^2) \cap \mathcal G_{\sigma'}^{\lambda'}$ for some $\lambda'>0$ such that $G(  1/2 ) > G(   3/4 )>0$ and $G(z) =0$ for $|z| \ge 1 $. The existence of such a bump function is well-known.

Next, choose $R>0$ such that
\[
\Supp\omega_b\subset B_R,
\]
and set 
\[
G_R(z):=G\left(\frac{z}{4R}\right).
\]
Since $\sigma'>\sigma$, we have $G \in  \mathcal G_\sigma^\lambda$ for any given $0<\lambda<\infty$, and thus
\[
G_R\in C_c^\infty(\mathbb R^2)\cap \mathcal G_\sigma^\lambda,
\qquad
\Supp G_R\subset B_{4R}.
\]

For $\ep>0$, define a family of initial data
\begin{equation}\label{eq:R2-Gevrey-small-perturbation data}
\omega_{ {in},\ep}
:=
\omega_{ b }+ \omega_{ p } =\omega_{ b }+ \ep G_R.
\end{equation}
 
For all sufficiently small $ \ep>0$, we have
\[
\|\omega_{ {in},\ep}\|_{L^1}
+
\|\omega_{ {in},\ep}\|_{L^\infty}
\le \ep_0 .
\]
Since $\omega_b\not\equiv0$, for all sufficiently small $ \ep>0$, the family
$\omega_{ {in},\ep}$ has a uniform positive lower bound on
its $L^1$ norm and a uniform upper bound on its second moment. Let $\om_{\ep}(t)$ be the solution of \eqref{eq:euler} with initial vorticity $\om_{{in},\ep}$.
By Theorem~\ref{thm:R2-positive-sharp-moment-growth}, there exist $c_0,T>0$, independent of $\ep$,
such that for every $t\ge T$,
\begin{equation}\label{eq: proof R2-Cs-growth 1}
\int_{y\ge c_0\sqrt{\log t}}\omega_\ep(t,z)\,dz\ge c_0,
\qquad
\int_{y\le -c_0\sqrt{\log t}}\omega_\ep(t,z)\,dz\ge c_0,
\end{equation}
and
\begin{equation}\label{eq: proof R2-Cs-growth 2}
\int_{x\ge c_0t\sqrt{\log t}}\omega_\ep(t,z)\,dz\ge c_0,
\qquad
\int_{x\le -c_0t\sqrt{\log t}}\omega_\ep(t,z)\,dz\ge c_0.
\end{equation}
Shrinking $ \ep$ if necessary, we may also assume
\[
 \|\omega_{ p } \|_{L^1}\le \frac12 c_0
\quad \text{and} \quad \| 
\omega_{ {p}} \|_{\mathcal G^\lambda_\sigma} \le \delta .
\]

\vspace{0.5em}
\noindent
\textbf{Step 2: Filamentation of the vorticity.}

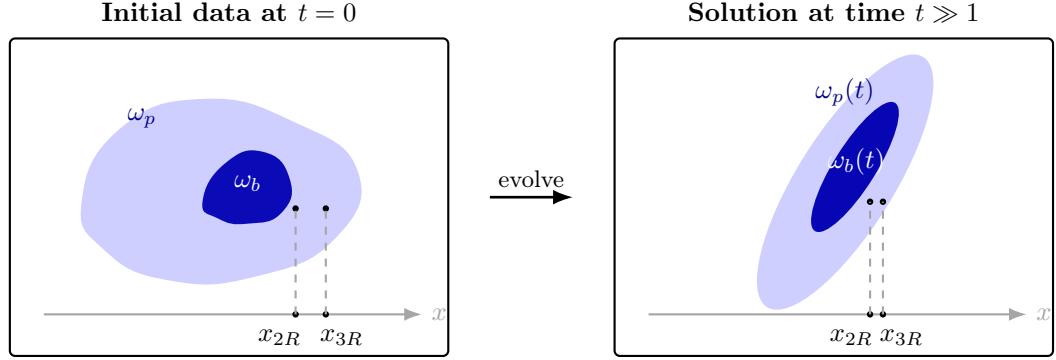
\begin{figure}[ht]
\centering
\begin{tikzpicture}[x=1cm,y=1cm,>=Latex,thick]

\def\W{5.8}
\def\H{4.2}

\begin{scope}
  \draw[rounded corners=2pt] (0,0) rectangle (\W,\H);
  \node[font=\bfseries] at (0.5*\W,\H+0.35)
  {Initial data at $t=0$};

  \draw[->,gray!70]
  (0.45,0.55) -- (\W-0.35,0.55) node[right] {$x$};

  \fill[blue!20,opacity=.95]
    plot[smooth cycle,tension=0.85] coordinates {
      (0.95,2.05) (1.45,3.00) (2.55,3.40) (3.75,3.10)
      (4.60,2.45) (4.30,1.55) (3.25,1.05) (2.10,1.00)
      (1.25,1.40)
    };

  \fill[blue!70!black,opacity=.95]
    plot[smooth cycle,tension=0.9] coordinates {
      (2.55,2.05) (2.85,2.55) (3.35,2.70)
      (3.70,2.35) (3.60,1.90) (3.15,1.75)
      (2.75,1.82)
    };

  \node[blue!50!black] at (1.75,3.15) {$\omega_p$};
  \node[white] at (3.15,2.28) {$\omega_b$};

  \coordinate (xTwoL)   at (3.78,0.55);
  \coordinate (xThreeL) at (4.18,0.55);

  \coordinate (pTwoL)   at (3.78,1.95);
  \coordinate (pThreeL) at (4.18,1.95);

  \fill (xTwoL) circle (1.2pt);
  \fill (xThreeL) circle (1.2pt);

  \fill (pTwoL) circle (1.2pt);
  \fill (pThreeL) circle (1.2pt);

  \node[below=2pt,xshift=-6pt] at (xTwoL) {$x_{2R}$};
  \node[below=2pt,xshift= 6pt] at (xThreeL) {$x_{3R}$};

  \draw[dashed,gray!70] (xTwoL) -- (pTwoL);
  \draw[dashed,gray!70] (xThreeL) -- (pThreeL);
\end{scope}

\draw[->,line width=1pt] (6.35,2.10) -- (7.45,2.10);
\node[above] at (6.90,2.10) {\small evolve};

\begin{scope}[xshift=8cm]
  \draw[rounded corners=2pt] (0,0) rectangle (\W,\H);
  \node[font=\bfseries] at (0.5*\W,\H+0.35)
  {Solution at time $t\gg1$};

  \draw[->,gray!70]
  (0.45,0.55) -- (\W-0.35,0.55) node[right] {$x$};

  \begin{scope}[rotate around={58:(3.05,2.30)}]
    \fill[blue!20,opacity=.95]
    (3.05,2.30) ellipse (1.95 and 0.64);
  \end{scope}

  \begin{scope}[rotate around={58:(3.18,2.50)}]
    \fill[blue!70!black,opacity=.97]
    (3.18,2.50) ellipse (1.00 and 0.28);
  \end{scope}

  \node[blue!50!black] at (3.05,3.50) {$\omega_p(t)$};
  \node[white] at (3.20,2.58) {$\omega_b(t)$};

  \coordinate (xTwoR)   at (3.38,0.55);
  \coordinate (xThreeR) at (3.55,0.55);

  \coordinate (pTwoR)   at (3.38,2.04);
  \coordinate (pThreeR) at (3.55,2.04);

  \fill (xTwoR) circle (1.2pt);
  \fill (xThreeR) circle (1.2pt);

  \fill (pTwoR) circle (1.2pt);
  \fill (pThreeR) circle (1.2pt);

  \node[below=2pt,xshift=-7pt] at (xTwoR) {$x_{2R}$};
  \node[below=2pt,xshift= 7pt] at (xThreeR) {$x_{3R}$};

  \draw[dashed,gray!70] (xTwoR) -- (pTwoR);
  \draw[dashed,gray!70] (xThreeR) -- (pThreeR);
\end{scope}

\end{tikzpicture}

\caption{Schematic illustration for Theorem \ref{thm:R2-Gevrey-norm-growth}.
Left: the initial configuration consists of a small concentrated blob
$\omega_b$ and a larger, lighter perturbation $\omega_p$.
Right: at time $t\gg1$ under the transport and filamentation,
the distinguished points $x_{2R}$ and $x_{3R}$ are closer.}
\end{figure}

For the family of initial data $\om_{in,\ep}$ given by \eqref{eq:R2-Gevrey-small-perturbation data}, let $\Phi_t$ be the flow map associated with $\omega_\ep(t)$.
Since vorticity is transported,
\[
\omega_\ep(t,\Phi_t(z))
=
\omega_b (z)+\ep G_R(z).
\]
Moreover, the mass of $\ep G_R$ in any transported set is at most
$\ep\|G_R\|_{L^1}\le c_0/2$. Therefore, \eqref{eq: proof R2-Cs-growth 1} implies for any $t \ge T$
\begin{equation}\label{eq: proof R2-Cs-growth 33}
\Phi_t(\Supp\omega_b)
\cap
\{y\ge c_0\sqrt{\log t}\}
\neq\varnothing,
\end{equation}
and
\begin{equation}\label{eq: proof R2-Cs-growth 3}
\Phi_t(\Supp\omega_b)
\cap
\{y\le -c_0\sqrt{\log t}\}
\neq\varnothing.
\end{equation}

For $r>0$, write
\[
\Omega_r(t):=\Phi_t(B_r).
\]
Since $\Supp \omega_b\subset B_R$, it follows
from \eqref{eq: proof R2-Cs-growth 33}, \eqref{eq: proof R2-Cs-growth 3} and the connectedness of $\Omega_R(t)$ that its projection onto
the $y$-axis contains $[-c_0\sqrt{\log t},\,c_0\sqrt{\log t}]$, namely for any $t \ge T$,
\begin{equation}\label{eq: proof R2-Cs-growth 4}
[-c_0\sqrt{\log t},\,c_0\sqrt{\log t}] \subset \{y\in\mathbb R:\exists\,x\in\mathbb R
\text{ such that } (x,y )\in \Omega_{R}(t)\}. 
\end{equation}
On the other hand, Fubini and incompressibility give
\begin{equation}\label{eq: proof R2-Cs-growth 4b}
\int_{\R} \left|
\{x\in\R:(x,y )\in\Omega_{3R}(t)\}
\right| \, dy = |\Omega_{3R}(t)|=|B_{3R}| .
\end{equation}

By  \eqref{eq: proof R2-Cs-growth 4b}, for any $t \ge T$ there exists
\[
y_t\in[-c_0\sqrt{\log t},\,c_0\sqrt{\log t}]
\]
such that 
\begin{equation}\label{eq: proof R2-Cs-growth 5}
\left|
\{x\in\mathbb R:(x,y_t)\in\Omega_{3R}(t)\}
\right|
\lesssim
\frac{1}{\sqrt{\log t}}.
\end{equation}

Using \eqref{eq: proof R2-Cs-growth 4}, there exists $x_t\in\R$ such that $(x_t,y_t) \in \Omega_R(t) \subset \Omega_{2R}(t) \subset \Omega_{3R}(t)$. For $r=2R,3R$, let $I_r(t)$ be the connected component of
\begin{equation}\label{eq sec5 Ir}
\{x\in\mathbb R:(x,y_t)\in\Omega_r(t)\}
\end{equation}
containing $x_t$ and let $x_r(t)$ be its right endpoint. Then
\begin{equation}\label{eq: proof R2-Cs-growth 5b}
(x_{2R}(t),y_t)\in\Phi_t(\partial B_{2R}),
\qquad
(x_{3R}(t),y_t)\in\Phi_t(\partial B_{3R}),
\end{equation}
and hence, by \eqref{eq: proof R2-Cs-growth 5} and \eqref{eq sec5 Ir}, for any $t \ge T$
\begin{equation}\label{eq: proof R2-Cs-growth 6}
|x_{3R}(t)-x_{2R}(t)| \le |I_{3R}(t)|
\lesssim
\frac{1}{\sqrt{\log t}}.
\end{equation}

\vspace{0.5em}
\noindent
\textbf{Step 3: All-time norm growth.}

Since by the choice of $R>0$
\[
\omega_b=0
\qquad
\text{on } \R^2 \setminus B_{ R}  ,
\]
we have
\[
\omega_{ {in},\ep}
=
\ep G(1/2)
\quad
\text{on }\partial B_{2R}, \quad \text{and} \quad
\omega_{ {in},\ep}
=
\ep G(3/4)
\quad
\text{on }\partial B_{3R}.
\]
Moreover,
\[
G(1/2)>G(3/4)>0.
\]
Since the vorticity is transported along the flow, by \eqref{eq: proof R2-Cs-growth 5b} we have
\begin{equation}
\omega_\ep(t,x_{2R}(t),y_t) =  \ep G(1/2) \quad \text{and}\quad \omega_\ep(t,x_{3R}(t),y_t) =  \ep G(3/4 ). 
\end{equation} 

Therefore, given any
$0<s\le 1$, we have for any $t \ge T$
\begin{equation}\label{eq: proof R2-Cs-growth 7}
\begin{aligned}
\|\omega_\ep(t)\|_{C_x^s}
&\gtrsim
\frac{
\left|
\omega_\ep(t,x_{2R}(t),y_t)
-
\omega_\ep(t,x_{3R}(t),y_t)
\right|
}{
|x_{2R}(t)-x_{3R}(t)|^s
}\\
&=
\frac{
\ep|G(1/2)-G(3/4)|
}{
|x_{2R}(t)-x_{3R}(t)|^s
}
\\
&\gtrsim_{\omega_{ {in},\ep},s}
(\log t)^{s/2}.
\end{aligned}
\end{equation}

For $s>1$, the 1D interpolation inequality
\[
\|f\|_{C_x^1}
\lesssim_s
\|f\|_{L^\infty}^{1-\frac1s}
\|f\|_{C_x^s}^{\frac1s},
\]
together with \eqref{eq: proof R2-Cs-growth 7} for $s=1$ and conservation of the $L^\infty$ norm,
gives
\[
\|\omega_\ep(t)\|_{C_x^s}
\gtrsim_{\omega_{ {in},\ep},s}
(\log t)^{s/2}. 
\]

The vertical estimate follows in the same way from \eqref{eq: proof R2-Cs-growth 2}, by exchanging the roles of $x$ and $y$. 
 
\end{proof}

\begin{remark}\label{rmk L2 based norm growth}
By an argument similar to that in \cite{MR4350517}, for any $1\le p<\infty$, one may also obtain
$$
\|\p_x\om(t)\|_{L^p} \gtrsim_{p, \om_{in}} \sqrt{\log t} \quad \text{and} \quad \|\p_y\om(t)\|_{L^p} \gtrsim_{p, \om_{in}} t \sqrt{\log t}. 
$$
\end{remark}


\section{Non-damping and tightness for negative vorticity in the whole plane}
\label{sec:plane-negative}
In Section \ref{sec:damping_wholespace}, we established that small nonnegative vorticity in the whole plane undergoes uniform spatial evacuation, leading to velocity decay as $t\to\infty$. We now turn to the complementary regime of non-positive vorticity, where the long-time dynamics exhibit the opposite behavior.

As in the infinite channel, Hamiltonian conservation prevents velocity decay outside the positive-vorticity regime. In the whole plane, however, the unbounded vertical geometry leads to a stronger conclusion: the planar conservation laws force non-positive perturbations to remain spatially tight.

This stark sign dichotomy again relies on  the conserved Hamiltonian on $\mathbb{R}^{2}$:
$$
\mathcal{H}[\omega]=-\frac{1}{4\pi}\iint_{\mathbb{R}^{2}\times\mathbb{R}^{2}}\log\vert{}z-z^{\prime}\vert{}\omega(t,z)\omega(t,z^{\prime})dzdz^{\prime}+\frac{1}{2}\int_{\mathbb{R}^{2}}y^{2}\omega(t,z)dz.
$$

\subsection{A uniform velocity lower bound}
Observe that if non-positive vorticity dispersed and evacuated compact sets, the large-scale spatial separation would force the logarithmic interaction energy to diverge.

 We use this to establish the vorticity trapping in Theorem \ref{thm:intro-R2-trapping}.

\begin{proposition}\label{prop:R2 negative trapping}
Consider non-positive initial data $\om_{in} \in L^1\cap L^\infty(\R^2)$ satisfying the assumptions in Theorem \ref{thm:intro-R2-trapping} and let  $\om(t)$ be the corresponding Yudovich solution.
 
Then 
\begin{equation}\label{eq:R2-negative-uniform-y-second-moment}
    \sup_{t\ge0}
    \int_{\R^2}y^2|\om(t,z)|\,dz
    <\infty,
\end{equation}
and there exists a center trajectory $t\mapsto z_{t}=(x_{t},y_{t})$  such that for every $\ep>0$, there is a radius $R_{\ep}>0$ for which
\begin{equation}\label{eq:prop:R2 negative trapping a}
    \int_{|z-z_t| \le R_{\ep}}
    |\om(t,z)|\,dz
    \ge
    \|\om_{in}\|_{L^1(\R^2)}-\ep \quad \text{for all } t\ge0.
\end{equation}
Moreover, for any $2< p \le \infty$, there exists $c_p>0$ depending only on $\om_{in}$ and $p$ such that
\begin{equation}\label{eq:prop:R2 negative trapping b}
  \| u(t)\|_{L^p(\R^2)} \ge c_p \quad \text{for all } \,\, t\ge0.
\end{equation}
    
\end{proposition}
\begin{proof}

We first prove the vorticity trapping. The uniform lower bounds for the velocity then follow as a consequence.

\vspace{0.5em}
\noindent
\textbf{Bounded moment.}

Since $\om(t)\le0$, the conservation of the Hamiltonian gives
\begin{equation}
\begin{aligned}
 \int_{\R^2}y^2|\om(t,z)|\,dz & = -2\mathcal{H}(\om (t))   - \frac{1}{2\pi} \iint
\log|z-z'|
\om(t ,z)\om(t ,z')\,dzdz' \\
&  = -2 \mathcal{H}(\om_{in}  )   - \frac{1}{2\pi} \iint
\log|z-z'|
\om(t ,z)\om(t ,z')\,dzdz'.  \\
\end{aligned}
\end{equation}
By Lemma \ref{lem:R2-log-energy-divergence} (taking, for instance, $R=3$), we see that
$$
\iint
\log|z-z'|
\om(t ,z)\om(t ,z')\,dzdz' \ge -C_{in}
$$
for some $C_{in}>0$ depending only on $\om_{in}$. Thus, $\sup_{t\ge0}
    \int_{\R^2}y^2|\om(t,z)|\,dz <\infty$.

\vspace{0.5em}
\noindent
\textbf{Vorticity trapping.}
We first claim that for every $\ep>0$, there exists $R_{\ep}>0$ such that for every $t\ge0$, there exists $z_{t,\ep}\in\R^2$ satisfying
\begin{equation}\label{eq:proof of thm:intro-R2-nondamping 2}
    \int_{B_{R_{\ep}}(z_{t,\ep})}
    |\om(t,z)|\,dz
    \ge
    \|\om_{in}\|_{L^1}-\ep.
\end{equation}
Suppose this claim does not hold. Then  there exist $\ep>0$, $R_n\to+\infty$ and $t_n\ge0$ such that
\begin{equation}\label{eq:vorticity-trapping-contradiction}
    \sup_{ z_0\in\R^2}
    \int_{B_{R_n}(z_0)}
    |\om(t_n,z)|\,dz
    \le
    \|\om_{in}\|_{L^1}-\ep.
\end{equation}

Since $\om \le 0$ is sign-definite, 
$$
\begin{aligned}
\mathcal H[\om_{in}]
& =
-\frac{1}{4\pi}
\iint
\log|z-z'|
\om(t_n,z)\om(t_n,z')\,dzdz'
+
\frac12
\int
y^2\om(t_n,z)\,dz \\
& \le  
-\frac{1}{4\pi}
\iint
\log|z-z'|
\om(t_n,z)\om(t_n,z')\,dzdz' ,
\end{aligned}
$$
and then we apply Lemma \ref{lem:R2-log-energy-divergence}  and \eqref{eq:vorticity-trapping-contradiction}  to obtain
$$
\begin{aligned}
\mathcal H[\om_{in}] 
& \le   
\frac{1}{4\pi} \log R_n \, \|\om(t_n)\|_{L^1}
\left(
\sup_{ z_0\in\R^2}
\int_{B_{R_n}( z_0)}|\om(t_n,z)|\,dz -\|\om(t_n)\|_{L^1}
\right)\\
&\qquad \qquad  + \frac{\pi}{2}\|\om(t_n) \|_{L^1 }
\|\om(t_n) \|_{L^\infty }   \\
& \le -\frac{1}{4\pi }
\ep\|\om_{in}\|_{L^1}
\log R_n
+ C_{in},
\end{aligned}
$$
where the constant $C_{in}>0$ depends on the initial Yudovich norms.

This implies that $\mathcal H[\om_{in}] \to -\infty$ as $n \to  \infty$, a contradiction. Therefore, the claim \eqref{eq:proof of thm:intro-R2-nondamping 2} holds.

It remains to show that the point $z_{t,\ep}$ obtained above can be chosen independent of $\ep>0$. For every $t\ge0$, define
$$
    z_t=z_{t,\ep_0 } \quad \text{with} \quad\ep_0  = \frac14\|\om_{in} \|_{L^1}
$$
such that  \eqref{eq:proof of thm:intro-R2-nondamping 2} implies
$$
    \int_{B_{R_{\ep_0}}(z_t)}
    |\om(t,z)|\,dz
    \ge
    \frac34\|\om_{in}\|_{L^1}.
$$
For any $\ep\ge\ep_0$, the same ball already satisfies the required estimate. If $0<\ep< \ep_0 $, let $z_{t,\ep}$ be the point and $R_{\ep}$ be the radius in \eqref{eq:proof of thm:intro-R2-nondamping 2}. Since
$$
    \frac34\|\om_{in}\|_{L^1}
    +
    \|\om_{in}\|_{L^1}-\ep
    >
    \|\om_{in}\|_{L^1},
$$
the balls $B_{R_{ \ep_0 }}(z_t)$ and $B_{R_{\ep}}(z_{t,\ep})$ must intersect. Thus, by taking $\widetilde{R}_{\ep}=2 R_{\ep_0}+2R_{\ep}$ we have
$$
    \int_{B_{\widetilde{R}_{\ep}}(z_t)}
    |\om(t,z)|\,dz
    \ge \int_{B_{R_{\ep}}(z_{t,\ep})}
    |\om(t,z)|\,dz \ge
    \|\om_{in}\|_{L^1}-\ep,
$$
which completes the proof.

\noindent
\textbf{Non-damping velocity.}

Suppose for contradiction that there exist $2<p \le \infty$ and a sequence $t_n \to +\infty$ such that 
$$
 \|u(t_n)\|_{L^p} \to 0 \quad \text{as} \quad n \to \infty.
$$

Since $\om \le 0$ is sign-definite, for $\ep=\frac12\|\om_{in}\|_{L^1}$, it follows from Proposition \ref{prop:damping-circulation-equivalence} that  
\begin{equation}\label{eq 611}\begin{aligned}
\lim_{n\to\infty}
        \sup_{z_0 }
        \int_{B_{R_{\ep}}( z_0)}|\om(t_n,z')|\,dz'\le  C (R_{\ep}+1)^2\lim_{n\to\infty}
        \sup_{z_0\in\R^2}
        \int_{B_1( z_0)}|\om(t_n,z')|\,dz'
        =0,
\end{aligned}
\end{equation}
which contradicts \eqref{eq:proof of thm:intro-R2-nondamping 2} and hence completes the proof.

\end{proof}

\subsection{Dynamics of the concentration center \texorpdfstring{$z_t$}{zt}}

Next, we derive  asymptotics of the concentration center $z_t$ given by Proposition  \ref{prop:R2 negative trapping}. The result below shows that its vertical component remains bounded and that its asymptotic horizontal speed is determined by the  vertical moment.

\begin{proposition}\label{prop:R2 z_t dynamics}
Let $z_t=(x_t,y_t)\in\R^2$ be the point obtained in Proposition \ref{prop:R2 negative trapping}.
Then  
\begin{equation}\label{eq sec6 control xt yt}
    \lim_{t \to \infty} \left|\frac{x_t}{t}-v_{in} \right|=0
    \qquad
    \text{and}
    \qquad
    \limsup_{t \to \infty} |y_t| <\infty,
\end{equation}
where $v_{in} = \frac{ \int_{\R^2} y\om_{in}(z)\,dz}{\int_{\R^2} \om_{in}(z)\,dz}$ denotes the vertical barycenter of $\om_{in}$.
  
\end{proposition}

\begin{proof}
By the Galilean invariance of the equation, after replacing $\om$ by
$$
    \widetilde\om(t,x,y)
    :=
    \om(t,x+\bar x_0+t\bar y_0,y+\bar y_0),
$$
we may assume without loss of generality that
$$
    \int_{\R^2}x\om_{in}(z)\,dz
    =
    \int_{\R^2}y\om_{in}(z)\,dz
    =
    0.
$$
Consequently,
\begin{equation}\label{eq sec6 barcenter}
    \int_{\R^2}x\om(t,z)\,dz
    =
    \int_{\R^2}y\om(t,z)\,dz
    =
    0 \quad \text{for all} \quad t \ge 0.
\end{equation}

\vspace{0.5em}
\noindent
\textbf{Control of $y_t$.} A direct calculation gives
\begin{equation}\label{eq sec6 1}
\begin{aligned}
\frac{d}{dt}
\int_{\R^2}x^2|\om(t,z)|\,dz
&=
2\int_{\R^2}xy|\om(t,z)|\,dz
+
2\int_{\R^2}xu^x(t,z)|\om(t,z)|\,dz.
\end{aligned}
\end{equation}
By symmetrization,
\begin{equation}\label{eq sec6 2}
    \left|
    \int_{\R^2}xu^x(t,z)|\om(t,z)|\,dz
    \right|
    \lesssim \|\om(t)\|_{L^1}^2 \lesssim 1.
\end{equation}
Using \eqref{eq:R2-negative-uniform-y-second-moment} and Cauchy's inequality, we also obtain 
$$
\left|\int xy |\om(t,z)|\,dz\right| \le \left(\int x^2 |\om(t,z)|\,dz  \int y^2|\om(t,z)|\,dz \right)^{1/2}\lesssim \left(\int x^2 |\om(t,z)|\,dz \right)^{1/2}.
$$
Together with \eqref{eq sec6 1} and \eqref{eq sec6 2}, Gr\"onwall's inequality gives
\begin{equation}\label{eq sec6 control of second momentum x}
    \int_{\R^2}x^2|\om(t,z)|\,dz
    \lesssim (1+t)^2.
\end{equation}
For any
$
0<\ep< \frac12\|\om_{in}\|_{L^1},
$
let $R_{\ep}>0$ and $z_t=(x_t,y_t)$ be the constant and point in Proposition \ref{prop:R2 negative trapping}. By \eqref{eq:R2-negative-uniform-y-second-moment},
$$
    \frac{\|\om_{in}\|_{L^1}}{2}(|y_t|-R_{\ep})_+^2
    \le
    \int_{\R^2}y^2|\om(t,z)|\,dz \lesssim 1,
$$
which implies that
$
    y_t=O(1).
$

\vspace{0.5em}
\noindent
\textbf{Control of $x_t$.}

Since $\om(t) \le 0$, it follows from \eqref{eq sec6 barcenter} that
$$
    x_t\int_{|z-z_t| \le R_{\ep}} |\om(t,z)|\,dz=\int_{|z-z_t| \le R_{\ep}} (x_t-x) |\om(t,z)|\,dz -\int_{|z-z_t| > R_{\ep}} x |\om(t,z)|\,dz.
$$
For any $0<\ep<\|\om_{in}\|_{L^1}$,  \eqref{eq sec6 control of second momentum x}, Proposition \ref{prop:R2 negative trapping}, and H\"older's inequality yield 
\begin{equation}\label{eq sec6 control xt1}
    \begin{aligned}
        \left(\|\om_{in}\|_{L^1}-\ep\right)|x_t|
&\le
\|\om_{in}\|_{L^1}R_{\ep}
+
\int_{|z-z_t| > R_{\ep}}
|x||\om(t,z)|\,dz \\
&\le\|\om_{in}\|_{L^1}R_{\ep}+ \left(\int_{|z-z_t|> R_{\ep}}|\om(t,z)|\,dz \int_{\R^2} x^2 |\om(t,z)| \,dz\right)^{1/2}\\
&\lesssim \|\om_{in}\|_{L^1}R_{\ep}+ \ep^{1/2} (1+t),
    \end{aligned}
\end{equation}
which implies
$$
    \limsup_{t\to\infty}
    \frac{|x_t|}{t}
    \lesssim
    \ep^{1/2}.
$$
Since $\ep>0$ can be taken arbitrarily small, we conclude that
$\frac{x_t}{t}\to0$ as $ t\to\infty$.

\end{proof}

\subsection{Improved non-damping and proof of Theorem \ref{thm:intro-R2-trapping}}\label{subsec:R2 nondamping main proof}

The first point in Theorem \ref{thm:intro-R2-trapping} has been proved by Proposition \ref{prop:R2 negative trapping} and Proposition \ref{prop:R2 z_t dynamics}. To conclude its proof, we only need to demonstrate its non-damping statement.

Below we prove something substantially stronger: each component of the velocity, as well as each component of its first derivative, remains bounded away from zero.

\begin{proposition}
\label{cor:R2-horizontal-velocity-derivative-nondamping}
Let $\om(t)$ be the Yudovich solution to \eqref{eq:euler} with nonzero initial vorticity $\om_{in}$ satisfying all the assumptions of Theorem \ref{thm:intro-R2-trapping}.

Let $c >|v_{in}|$, then for any $2 < p \le \infty$ 
\begin{equation}\label{eq:thm:intro-R2-nondamping}
\liminf_{t\to\infty} \min_{\alpha\in\{x,y\}} \|u^\alpha(t)\|_{L^p(\{|z|\leq ct\})} >0,
\end{equation}
and  for $1<p<\infty$
\begin{equation} 
\liminf_{t\to\infty} \min_{\alpha,\beta\in\{x,y\}} \|\partial_\alpha u^\beta(t)\|_{L^p(\{|z|\leq ct\})} >0. 
\end{equation} 
\end{proposition}
\begin{proof}

\vspace{0.5em}
\noindent
\textbf{Zero-order velocity non-damping.} 

We only demonstrate $u^x$. Let $\eta \ge 0$ be a standard cut-off function such that $\eta(s)=1$ when $|s| \le 1$ and $\eta(s)=0$ when $|s| \ge 2$. 

Choose $\ep=\frac{1}{2}\|\om_{in}\|_{L^1}$ and set $\eta_{t}(z)=\eta\left(   \frac{ x-x_t }{N R_\ep}   , \frac{ y-y_t }{R_\ep }   \right)$ for some large $N \ge 1$ to be determined. Using the vorticity trapping from Proposition \ref{prop:R2 negative trapping} and $\om=\p_x u^y-\p_y u^x$, we obtain for any $1 \le p \le \infty$
\begin{equation}\label{eq sec6 lower bound for lp u}
\begin{aligned}
\frac{1}{2}\|\om_{in}\|_{L^1} & \le -\int_{\R^2} \om(t,z)\eta_t(z) \,dz \\
& \leq \|u^x(t)\|_{L^p(\Supp(\eta_t))} \|\p_y \eta_t\|_{L^{p'}(\R^2)} + \|u^y(t)\|_{L^3  } \|\p_x \eta_t\|_{L^{3/2}}.
\end{aligned}
\end{equation}
Since $ \sup_t \| u(t)\|_{L^p}<\infty$ for any $2< p\le \infty $, $\|\p_x \eta_t\|_{L^{p'}} \lesssim 1 $, and
$$
\|\p_x \eta_t\|_{L^{3/2}} \lesssim N^{-1 + \frac{2}{3}}   \quad \text{and }\quad \|\p_y \eta_t\|_{L^{p'}} \lesssim N^{  \frac{1}{p'}},
$$
we can fix a sufficiently large $N \ge 1 $ (independent of $t$) so that  $\|u^y(t)\|_{L^3 } \|\p_x \eta_t\|_{L^{3/2}} \le \frac{1}{4}\|\om_{in}\|_{L^1}$ for all times. Hence, \eqref{eq sec6 lower bound for lp u} yields 
\begin{equation}\label{eq sec6 lower bound for lp u 2}
\begin{aligned}
 \|\om_{in}\|_{L^1}  N^{-\frac{1}{p'}}      \lesssim \|u^x(t)\|_{L^p(\Supp(\eta_t))}.  
\end{aligned}
\end{equation}
Since $c>|v_{in}|$, by Proposition \ref{prop:R2 z_t dynamics},   $ \Supp(\eta_t) \subset \{|z| \le ct\}$ for all sufficiently large $t$. It follows  that
$$
\liminf_{t \to \infty}\|u^x(t)\|_{L^p(|z|\le ct)} \ge \liminf_{t \to \infty}\|u^x(t)\|_{L^p(\Supp(\eta_t))} \gtrsim 1.
$$
The lower bound for $u^y$ follows by a nearly identical argument, by suitably modifying the anisotropic cutoff $\eta_t$.

\vspace{0.5em}
\noindent
\textbf{First-order non-damping.} 

Let $\psi=\Delta^{-1}\om$. It suffices to prove each component of the Hessian of the stream function $ D^2 \psi $ has a lower bound in $L^p$, $1<p<\infty$. We only demonstrate the lower bound for $\p_x^2 \psi$  and $\p_{xy}\psi$.

For $\p_x^2\psi$, fix $\ep=\frac{1}{4}\|\om_{in}\|_{L^1}$, choose an even function $g\in C_c^\infty(\R)$ satisfying
\begin{equation}\label{eq sec6 def of g}
    g=1\quad\text{on }[-R_{\ep},R_{\ep}],
    \qquad
    |g|\le1,
    \qquad
    \int_{\R}g(s)\,ds=0.
\end{equation}
Define
\[
    f(x):=\int_{-\infty}^x(x-s)g(s)\,ds.
\]
It follows that $f\in C_c^\infty(\R)$ and $f''=g$. Let
$\chi\in C_c^\infty(\R)$ satisfy
\[
    0\le\chi\le1,
    \qquad
    \chi=1\quad\text{on }[-R_{\ep},R_{\ep}]
\]
and define
\(
    \phi_t(x,y):=f(x-x_t)\chi(y-y_t).
\)
Then we have
\begin{equation}\label{eq sec6 def of test phi}
\p_x^2 \phi_t =1 \quad \text{on $B_{R_\ep }(z_t)$}, \quad \|\p_x^2 \phi_t\|_{L^\infty} \leq 1 \quad \text{and} \quad \|\Delta \phi_t\|_{L^\infty} \lesssim 1.
\end{equation}

Thus, Proposition \ref{prop:R2 negative trapping} gives
\begin{equation}
\begin{aligned}
- \int_{\R^2}
\om(t,z)\p_{x}^2\phi_t (z)\,dz & \ge  \int_{B_{R_\ep}(z_t)} |\om| \,dz- \int_{\R^2\setminus B_{R_\ep}(z_t)} |\om|\,dz  \\
& \ge \frac{1}{2}\|\om_{in}\|_{L^1}.  
\end{aligned}
\end{equation}
On the other hand, integration by parts  yields
\begin{equation}
 \left| \int_{\R^2}
\om(t,z)\p_{x}^2\phi_t (z)\,dz \right|   = \left| \int_{\R^2}
\p_{x}^2 \psi(t,z) \Delta \phi_t (z)\,dz\right| \leq \| \p_x^2 \psi \|_{L^p(\Supp(\phi_t))} \| \Delta\phi_t \|_{L^{p'}}.
\end{equation}
Thus the lower bound for $\p_x^2 \psi = \p_x u^y$ follows.

For $\p_{xy}\psi$, we reuse the same set of variable names $f,g,\phi$ for better readability. In this case, we instead define
$$
  f(x)
:=
\int_{-\infty}^x g(s)\,ds,
$$
where $g$ is the same function defined in \eqref{eq sec6 def of g}. Thus, it follows that
$ f\in C_c^\infty(\R)$ and $  f'=g$. Define
$
 \phi_t(x,y)
:=
  f(x-x_t)  f(y-y_t).
$
Then
$$
\p_{xy} \phi_t=1
\quad\text{on }B_{R_\ep}(z_t),
\qquad
\|\p_{xy}  \phi_t\|_{L^\infty}\le1,
$$
which implies that
$$
-\int_{\R^2}
\om(t,z)\p_{xy} \phi_t(z)\,dz
\ge
\frac12\|\om_{in}\|_{L^1}.
$$
On the other hand, integration by parts gives
$$
\begin{aligned}
\left|
\int_{\R^2}
\om(t,z)\p_{xy} \phi_t(z)\,dz
\right|
&=
\left|
\int_{\R^2}
\p_{xy}\psi(t,z)\Delta \phi_t(z)\,dz
\right|\\
&\le
\|\p_{xy}\psi(t)\|_{L^p(\Supp( \phi_t))}
\|\Delta \phi_t\|_{L^{p'}}.
\end{aligned}
$$
Thus the desired lower bound for $\p_{xy}\psi$ follows.
 
\end{proof}

\appendix

\section{Proof of Proposition \ref{prop:damping-circulation-equivalence}}\label{append:proof}

\begin{proof}[Proof of Proposition \ref{prop:damping-circulation-equivalence}]
By replacing $\om_n$ with $-\om_n$ if necessary, we may assume without loss of
generality that $\om_n\ge0$ for all $n\ge1$.

Recall that the assumption implies
\begin{equation}\label{eq uniform bound for u}
    \sup_{n\ge1}\|u_n\|_{L^q(\Omega)}<\infty
\end{equation}
for every $1\le q\le\infty$ when $\Omega=\mathbb R\times[-1,1]$ and for
every $2<q\le\infty$ when $\Omega=\mathbb R^2$.

\textbf{$(1)\Rightarrow  (2)$:}

Suppose $ \| u_n\|_{L^{\overline{p}}} \to 0$ for some admissible $\overline{p} \leq \infty$. If $ \overline{p} = \infty$,  then for any $ p <\infty$ admissible, we can choose $q<p$ also admissible  such that
$$
    \|u_n\|_{L^p(\Omega)}
    \le
    \|u_n\|_{L^\infty(\Omega)}^{1-q/p}
    \|u_n\|_{L^q(\Omega)}^{q/p} \to 0.
$$

Now we consider $ \overline{p} < \infty$. It suffices to prove $ L^\infty$ damping from this.  

Suppose that, after passing to a subsequence, there exist $C>0$ and points $z_n\in\Omega$ such that $|u_n(z_n)|\ge C$. By the uniform log-Lipschitz continuity, there exists $r>0$ independent of $n$ such that
$$
    |u_n(z)-u_n(z_n)|\le \frac{C}{2}
    \quad \text{for all} \quad  z\in B_r(z_n).
$$
Therefore,
$$
    \|u_n\|_{L^{\overline p}(\Omega)}^{\overline p}
    \ge
    \int_{B_r(z_n)}|u_n(z)|^{\overline p}\,dz
    \ge
    \left(\frac{C}{2}\right)^{\overline p}
    |B_r(z_n)|
    \gtrsim_{r,C}1,
$$
contradicting $\|u_n\|_{L^{\overline p}(\Omega)}\to0$.

\textbf{$(2)\Rightarrow(3)$:}
Recall that $B_1(z_0)=\{z \in\Om : |z-z_0|\le 1\}$. Since $\om_n\ge0$, Stokes' theorem gives
$$
    \sup_{ z_0\in\Omega}
    \int_{B_1( z_0)}\om_n(z')\,dz'
    \le
    \sup_{ z_0\in\Omega}
    \left|
        \int_{\partial B_1( z_0)}u_n\cdot\tau\,ds
    \right|
    \lesssim
    \|u_n\|_{L^\infty(\Omega)}
    \to0.
$$

\textbf{$(3)\Rightarrow (1)$:} It suffices to show $L^\infty$ damping. 
We split the Biot--Savart integral into near, intermediate, and far regions. Recall that $|K(z,z')| \lesssim \frac{1}{|z-z'|}$ for either choice of \(\Omega\). Choose $R>1 $ sufficiently large such that 
$$ \sup_{z\in \Om} \int_{|z-z'|\le R^{-1} } \frac{|\om_n(z')|}{|z-z'|}\,dz' \le \frac{C}{R} \sup_{n \ge 1}\|\om_n\|_{L^\infty}  \le \frac{\ep}{3},
$$ 
and 
$$ \sup_{z\in \Om} \int_{|z-z'|\ge R } \frac{|\om_n(z')|}{|z-z'|}\,dz' \le \frac{1}{R}\sup_{n \ge 1}\|\om_n\|_{L^1} \le \frac{\ep}{3}. 
$$ 
Finally, in the intermediate region we have 
\begin{equation} \begin{aligned} \sup_{z\in\Om} \int_{ R^{-1} \le |z-z'|\le R } \frac{|\om_n(z')|}{|z-z'|}\,dz' & \le R \sup_{z\in\Om}\int_{|z-z'| \le R } |\om_n(z')| \,dz' \\ & \le C_R(R+1)^2 \sup_{z_0 \in\Om} \left| \int_{B_{1}(z_0)} \om_n (z)\,dz \right| \to 0, \end{aligned} \end{equation} 
which completes the proof.
\end{proof}

\section{Linear damping rates}\label{append:linear rate}

In this section, we prove linear inviscid damping rates using elementary arguments. These seem not to be available in the literature. 

A notable feature is that the rates depend  on both the $L^p$ scale and the spatial domain.

\subsection{Linear damping in the infinite channel \texorpdfstring{$\R\times [-1,1]$}{R x [-1,1]} }

We first prove the linear inviscid damping rates in the infinite channel $\R\times [-1,1]$. 

\begin{theorem}\label{thm:append_linear_damping}
    Let $1<p<\infty$ and $\Om=\R\times [-1,1]$.   Consider $\omega = \om_{in}(x - y t,y)$ for some $\om_{in} \in C_c^\infty(\Om )$ and let $u(t)$ be the corresponding velocity field. Then
\[
\|u(t)\|_{L^p(\Omega)}
\lesssim_{p, \om_{in}} \langle t\rangle^{-1/p'},
\]
where $1<p'<\infty$ is the H\"older dual of $p$.

\end{theorem}
\begin{proof}
We use a duality argument. 
Since
$$
\|u \|_{L^p(\Om)} = \sup_{\| F \|_{ L^{p'}} \le 1 } \int_\Omega u\cdot F,
$$
we take any vector-valued $F \in L^{p'}$ and let \(\Phi\) solve
\[
-\Delta \Phi=\nabla^\perp \cdot F,\qquad \Phi|_{y=\pm1}=0.
\]
Assume that $\Supp  \om_{in} \subset [-R_0,R_0]\times [-1,1]$. Then 
\[
\begin{aligned}
\left|\int_\Omega u\cdot F \right|
&=\left|\int_{\mathbb R}\int_{-1}^1
\omega_{in}(X,y)\Phi(X+ty,y)\,dy\,dX\right| \\
&\le \|\om_{in}\|_{L^\infty}\int_{-R_0}^{R_0}\int_{-1}^1
\left|\Phi(X+ty,y)\right|\,dy\,dX.
\end{aligned}
\]

Setting \(s=X+ty\) and using the 1D embedding
\(W^{1,p'}(-1,1)\hookrightarrow L^\infty(-1,1)\),
\[
\begin{aligned}
\int_{-1}^1 |\Phi(X+ty,y)|\,dy
&=\frac1{t}\int_{X-t}^{X+t}
\left|\Phi\!\left(s,\frac{s-X}{t}\right)\right|\,ds \\
&\lesssim_{p} \frac1{t}
\int_{X-t}^{X+t}
\|\Phi(s,\cdot)\|_{W^{1,p'}_y}\,ds \\
&\lesssim_{p}  t^{-1/p'}\|\Phi\|_{W^{1,p'}(\Omega)},
\end{aligned}
\]
where we have used the slicing property of the Sobolev space $W^{1,p'}$: for $1<p'<\infty$
\begin{equation}
  \|f \|_{L^{p'}_x W^{1,p'}_y} + \|f \|_{L^{p'}_y W^{1,p'}_x} \lesssim \|f\|_{W^{1,p'}(\Om) } .
\end{equation}
Hence, standard elliptic regularity yields
\[
\left|\int_\Omega u(t)\cdot F\right|
\lesssim |t|^{-1/p'}\|\Phi\|_{W^{1,p'}}
\lesssim |t|^{-1/p'}\|F\|_{L^{p'}}.
\]
By duality, 
\[
\|u(t)\|_{L^p(\Omega)}
\lesssim \langle t\rangle^{-1/p'}.
\]
\end{proof}

\subsection{Linear damping in the whole plane \texorpdfstring{$\R^2$}{R2}}

Next, we prove the linear inviscid damping rates in the whole plane. Note that the whole plane rate below $ \langle t\rangle^{-1+\frac{2}{p}}$ is slower than the channel case $ \langle t\rangle^{-1+\frac{1}{p}}$.

\begin{theorem}\label{thm:append_linear_damping R^2}
    Let $2<p<\infty$.   Consider $\omega = \om_{in}(x - y t,y)$ for some $\om_{in} \in C_c^\infty(\R^2 )$ and let $u(t)$ be the corresponding velocity field. Then
\[
\|u(t)\|_{L^p(\R^2)}
\lesssim_{p, \om_{in}} \langle t\rangle^{-1+\frac{2}{p}}.
\]

\end{theorem}
\begin{proof}
We also use the  duality 
$$
\|u \|_{L^p(\R^2)} = \sup_{\| F \|_{ L^{p'}} \le 1 } \int_{\R^2} u\cdot F.
$$
Take any vector-valued $F \in L^{p'}(\R^2)$ and let \(\Phi\) solve
\begin{equation}\label{eq sec app F}
-\Delta \Phi= \nabla^\perp \cdot F.
\end{equation}
  Assume that $\Supp \om_{in} \subset [-R_0,R_0]^2$. Then 
\[
\begin{aligned}
\left|\int_{\R^2} u\cdot F \right|
&=\left|\iint_{\mathbb R^2}
\omega_{in}(X,y)\Phi(X+ty,y)\,dy\,dX\right| \\
&\le \|\om_{in}\|_{L^\infty}\int_{-R_0}^{R_0}\int_{-R_0}^{R_0}
\left|\Phi(X+ty,y)\right|\,dy\,dX.
\end{aligned}
\] 
We perform the change of variables (so that $(x,y) \mapsto (\xi,\eta)$ is a rotation)
\[
    \xi
    =
    \frac{tX}{\langle t\rangle}
    +
    \langle t\rangle y,
    \qquad
    \eta
    =
    -\frac{X}{\langle t\rangle}.
\]
Define
$
    \overline{\Phi}(\xi,\eta)
    := \Phi(x,y)= 
    \Phi\left(
        \frac{t\xi-\eta}{\langle t\rangle},
        \frac{\xi+t\eta}{\langle t\rangle}
    \right)$.
Note that $|X|\le R_0, |y|\le R_0$ implies
\[
    |\eta|
    \le
    \frac{R_0}{\langle t\rangle} \quad \text{and} \quad |\xi|
    \le
    2R_0\langle t\rangle;
\]
we obtain
\[
\begin{aligned}
\int_{-R_0}^{R_0}\int_{-R_0}^{R_0}
    |\Phi(X+ty,y)|\,dy\,dX &\le
    \int_{|\xi|\le 2R_0\langle t\rangle}
    \int_{|\eta|\le R_0/\langle t\rangle}
    |\overline{\Phi}(\xi,\eta)|\,d\eta\,d\xi \\
&\le
    \frac{2R_0}{\langle t\rangle}
    \int_{|\xi|\le 2R_0\langle t\rangle}
    \|\overline{\Phi}(\xi,\cdot)\|_{L^\infty_\eta}
    \,d\xi.
\end{aligned}
\]

For $q=\frac{2p}{p-2}$, we use the 1D Sobolev interpolation 
$$ 
\|\overline{\Phi}(\xi,\cdot)\|_{L^\infty_\eta} \lesssim_p \|\overline{\Phi}(\xi,\cdot)\|_{L^{q}_\eta}^{\frac2p} \|\partial_\eta\overline{\Phi}(\xi,\cdot)\|_{L^{p'}_\eta}^{1-\frac2p}, 
$$ 
to obtain 
\begin{align*} 
\left|\int_{\R^2}u(t)\cdot F\,dz\right| &\lesssim_{p,\omega_{ in}} \frac{1}{\langle t\rangle} \int_{|\xi|\leq 2R_0 \langle t\rangle } \|\overline{\Phi}(\xi,\cdot)\|_{L^{q}_\eta}^{\frac2p} \|\p_\eta\overline{\Phi}(\xi,\cdot)\|_{L^{p'}_\eta}^{1-\frac2p} \,d\xi \\ 
&\lesssim_{p,\omega_{ in}} \langle t\rangle^{-1+\frac2p} \|\overline{\Phi}\|_{L^q(\R^2)}^{\frac2p} \|\p_{\eta}\overline{\Phi}\|_{L^{p'}(\R^2)}^{1-\frac2p} \\
&\lesssim_{p,\omega_{ in}} \langle t\rangle^{-1+\frac2p} \|\Phi\|_{L^q(\R^2)}^{\frac2p} \|\nabla \Phi\|_{L^{p'}(\R^2)}^{1-\frac2p},
\end{align*} 
where in the last step we have used that $(x,y) \mapsto (\xi,\eta)$ is a rotation that preserves all the Sobolev norms.

Finally, since $1<p'<2$, the homogeneous Sobolev inequality and the elliptic estimates for \eqref{eq sec app F} give
$$ 
\|\Phi\|_{L^q(\R^2)} \lesssim_p \|\nabla\Phi\|_{L^{p'}(\R^2)} \lesssim_p \|F\|_{L^{p'}(\R^2)},
$$ 
which implies
$$ \left|\int_{\R^2}u(t)\cdot F\,dz\right| \lesssim_{p,\omega_{ in}} \langle t\rangle^{-1+\frac2p}, 
$$ 
where the constant  depends only on the initial support size and Yudovich norms. 
\end{proof}

\bibliographystyle{plain}
\bibliography{damping}

@article{Case1960Couette,
  author  = {Case, K. M.},
  title   = {Stability of Inviscid Plane {Couette} Flow},
  journal = {Physics of Fluids},
  volume  = {3},
  number  = {2},
  pages   = {143--148},
  year    = {1960},
  doi     = {10.1063/1.1706010}
}

@article{BouchetMorita2010Euler,
  author  = {Bouchet, Freddy and Morita, Hiroki},
  title   = {Large Time Behavior and Asymptotic Stability of the
             {2D Euler} and Linearized {Euler} Equations},
  journal = {Physica D: Nonlinear Phenomena},
  volume  = {239},
  number  = {12},
  pages   = {948--966},
  year    = {2010},
  doi     = {10.1016/j.physd.2010.01.020}
}

@article{Zillinger2016FiniteChannel,
  author  = {Zillinger, Christian},
  title   = {Linear Inviscid Damping for Monotone Shear Flows in a
             Finite Periodic Channel, Boundary Effects, Blow-Up and
             Critical {Sobolev} Regularity},
  journal = {Archive for Rational Mechanics and Analysis},
  volume  = {221},
  number  = {3},
  pages   = {1449--1509},
  year    = {2016},
  doi     = {10.1007/s00205-016-0991-1}
}

@article{Zillinger2017MonotoneShear,
  author  = {Zillinger, Christian},
  title   = {Linear Inviscid Damping for Monotone Shear Flows},
  journal = {Transactions of the American Mathematical Society},
  volume  = {369},
  number  = {12},
  pages   = {8799--8855},
  year    = {2017},
  doi     = {10.1090/tran/6942}
}

@article{WeiZhangZhao2018Sobolev,
  author  = {Wei, Dongyi and Zhang, Zhifei and Zhao, Weiren},
  title   = {Linear Inviscid Damping for a Class of Monotone Shear
             Flow in {Sobolev} Spaces},
  journal = {Communications on Pure and Applied Mathematics},
  volume  = {71},
  number  = {4},
  pages   = {617--687},
  year    = {2018},
  doi     = {10.1002/cpa.21672}
}

@article{WeiZhangZhao2019VorticityDepletion,
  author  = {Wei, Dongyi and Zhang, Zhifei and Zhao, Weiren},
  title   = {Linear Inviscid Damping and Vorticity Depletion for
             Shear Flows},
  journal = {Annals of PDE},
  volume  = {5},
  number  = {1},
  pages   = {Paper No. 3},
  year    = {2019},
  doi     = {10.1007/s40818-019-0060-9}
}

@article {MR3974608,
    AUTHOR = {Bedrossian, Jacob and Germain, Pierre and Masmoudi, Nader},
     TITLE = {Stability of the {C}ouette flow at high {R}eynolds numbers in
              two dimensions and three dimensions},
   JOURNAL = {Bull. Amer. Math. Soc. (N.S.)},
  FJOURNAL = {American Mathematical Society. Bulletin. New Series},
    VOLUME = {56},
      YEAR = {2019},
    NUMBER = {3},
     PAGES = {373--414},
      ISSN = {0273-0979,1088-9485},
   MRCLASS = {76E05 (35B25 35B34 35B35 35Q30 76D05 76E30)},
  MRNUMBER = {3974608},
MRREVIEWER = {Georgios\ C.\ Georgiou},
       DOI = {10.1090/bull/1649},
       URL = {https://doi.org/10.1090/bull/1649},
}

@article{Jia2020Gevrey,
  author  = {Jia, Hao},
  title   = {Linear Inviscid Damping in {Gevrey} Spaces},
  journal = {Archive for Rational Mechanics and Analysis},
  volume  = {235},
  number  = {2},
  pages   = {1327--1355},
  year    = {2020},
  doi     = {10.1007/s00205-019-01445-x}
}

@article{ChenLiMiao2026InfiniteChannel,
  author  = {Chen, Q. and Li, Z. and Miao, C.},
  title   = {Quantitative Stability for the {2D} {Couette} Flow on
             the Infinite Channel with Non-Slip Boundary Condition},
  journal = {Journal of the London Mathematical Society},
  volume  = {113},
  number  = {1},
  pages   = {e70440},
  year    = {2026},
  doi     = {10.1112/jlms.70440}
}

@article{IftimieSiderisGamblin1999PlanarVorticity,
  author  = {Iftimie, Drago{\c{s}} and Sideris, Thomas C. and
             Gamblin, Paul},
  title   = {On the Evolution of Compactly Supported Planar
             Vorticity},
  journal = {Communications in Partial Differential Equations},
  volume  = {24},
  number  = {9--10},
  pages   = {1709--1730},
  year    = {1999},
  doi     = {10.1080/03605309908821480}
}

@article{IftimieLopesFilhoNussenzveigLopes2007Confinement,
  author  = {Iftimie, Drago{\c{s}} and Lopes Filho, Milton C. and
             Nussenzveig Lopes, Helena J.},
  title   = {Confinement of Vorticity in Two-Dimensional Ideal
             Incompressible Exterior Flow},
  journal = {Quarterly of Applied Mathematics},
  volume  = {65},
  number  = {3},
  pages   = {499--521},
  year    = {2007},
  doi     = {10.1090/S0033-569X-07-01059-4}
}

@article{ChoiDenisov2019InfiniteCylinder,
  author  = {Choi, Kyudong and Denisov, Sergey},
  title   = {On the Growth of the Support of Positive Vorticity for
             {2D Euler} Equation in an Infinite Cylinder},
  journal = {Communications in Mathematical Physics},
  volume  = {367},
  number  = {3},
  pages   = {1077--1093},
  year    = {2019},
  doi     = {10.1007/s00220-019-03295-w}
}

@article{Marchioro1994VortexSupport,
  author  = {Marchioro, Carlo},
  title   = {Bounds on the Growth of the Support of a Vortex Patch},
  journal = {Communications in Mathematical Physics},
  volume  = {164},
  number  = {3},
  pages   = {507--524},
  year    = {1994},
  doi     = {10.1007/BF02101489}
}

@article {MR4555167,
    AUTHOR = {Choi, Kyudong and Jeong, In-Jee},
     TITLE = {Filamentation near {H}ill's vortex},
   JOURNAL = {Comm. Partial Differential Equations},
  FJOURNAL = {Communications in Partial Differential Equations},
    VOLUME = {48},
      YEAR = {2023},
    NUMBER = {1},
     PAGES = {54--85},
      ISSN = {0360-5302,1532-4133},
   MRCLASS = {76B47 (35Q31)},
  MRNUMBER = {4555167},
MRREVIEWER = {Da-Wen\ Deng},
       DOI = {10.1080/03605302.2022.2139721},
       URL = {https://doi.org/10.1080/03605302.2022.2139721},
}

@article {MR4168275,
    AUTHOR = {Choi, Kyudong and Jeong, In-Jee},
     TITLE = {Growth of perimeter for vortex patches in a bulk},
   JOURNAL = {Appl. Math. Lett.},
  FJOURNAL = {Applied Mathematics Letters. An International Journal of Rapid
              Publication},
    VOLUME = {113},
      YEAR = {2021},
     PAGES = {Paper No. 106857, 9},
      ISSN = {0893-9659,1873-5452},
   MRCLASS = {76B47},
  MRNUMBER = {4168275},
       DOI = {10.1016/j.aml.2020.106857},
       URL = {https://doi.org/10.1016/j.aml.2020.106857},
}

@article {MR4350517,
    AUTHOR = {Choi, Kyudong and Jeong, In-Jee},
     TITLE = {Infinite growth in vorticity gradient of compactly supported
              planar vorticity near {L}amb dipole},
   JOURNAL = {Nonlinear Anal. Real World Appl.},
  FJOURNAL = {Nonlinear Analysis. Real World Applications. An International
              Multidisciplinary Journal},
    VOLUME = {65},
      YEAR = {2022},
     PAGES = {Paper No. 103470, 20},
      ISSN = {1468-1218,1878-5719},
   MRCLASS = {35Q31 (76B47)},
  MRNUMBER = {4350517},
       DOI = {10.1016/j.nonrwa.2021.103470},
       URL = {https://doi.org/10.1016/j.nonrwa.2021.103470},
}

@article{HamelNadirashvili2017ShearFlows,
  author  = {Hamel, Fran{\c{c}}ois and Nadirashvili, Nikolai},
  title   = {Shear Flows of an Ideal Fluid and Elliptic Equations in
             Unbounded Domains},
  journal = {Communications on Pure and Applied Mathematics},
  volume  = {70},
  number  = {3},
  pages   = {590--608},
  year    = {2017},
  doi     = {10.1002/cpa.21670}
}

@article{HamelNadirashvili2019LiouvillePlane,
  author  = {Hamel, Fran{\c{c}}ois and Nadirashvili, Nikolai},
  title   = {A {Liouville} Theorem for the {Euler} Equations in the Plane},
  journal = {Archive for Rational Mechanics and Analysis},
  volume  = {233},
  number  = {2},
  pages   = {599--642},
  year    = {2019},
  doi     = {10.1007/s00205-019-01364-x}
}

@article{LiuMasmoudiZhao2026HalfPlane,
  author        = {Liu, Ning and Masmoudi, Nader and Zhao, Weiren},
  title         = {Long-Time Behaviour of Two-Dimensional
                   {Navier--Stokes} Equations in the Presence of
                   {Couette} Flow on the Half Plane},
  journal       = {preprint},
  year          = {2026},
  eprint        = {2605.21663},
  archivePrefix = {arXiv},
  primaryClass  = {math.AP}
}

@article{bedrossian2024uniforminvisciddampinginviscid,
      title={Uniform Inviscid Damping and Inviscid Limit of the 2{D} {N}avier-{S}tokes equation with {N}avier Boundary Conditions}, 
      author={Jacob Bedrossian and Siming He and Sameer Iyer and Fei Wang},
      journal       = {preprint},
      year={2024},
      eprint={2405.19249},
      archivePrefix={arXiv},
      primaryClass={math.AP},
      url={https://arxiv.org/abs/2405.19249}, 
}

@article{DrivasNualart2026LaminarFlow,
  author  = {Drivas, Theodore D. and Nualart, Marc},
  title   = {A Geometric Characterization of Steady Laminar Flow},
  journal = {Communications on Pure and Applied Mathematics},
  year    = {2026},
  note    = {Early View},
  doi     = {10.1002/cpa.70055}
}

@article {MR3415068,
    AUTHOR = {Bedrossian, Jacob and Masmoudi, Nader},
     TITLE = {Inviscid damping and the asymptotic stability of planar shear
              flows in the 2{D} {E}uler equations},
   JOURNAL = {Publ. Math. Inst. Hautes \'Etudes Sci.},
  FJOURNAL = {Publications Math\'ematiques. Institut de Hautes \'Etudes
              Scientifiques},
    VOLUME = {122},
      YEAR = {2015},
     PAGES = {195--300},
      ISSN = {0073-8301,1618-1913},
   MRCLASS = {35Q31 (35B35 35Q35)},
  MRNUMBER = {3415068},
MRREVIEWER = {Matthew\ Paddick},
       DOI = {10.1007/s10240-015-0070-4},
       URL = {https://doi.org/10.1007/s10240-015-0070-4},
}

@article{MarchioroPulvirenti1983SingularVorticity,
  author  = {Marchioro, Carlo and Pulvirenti, Mario},
  title   = {Euler Evolution for Singular Initial Data and Vortex Theory},
  journal = {Communications in Mathematical Physics},
  volume  = {91},
  number  = {4},
  pages   = {563--572},
  year    = {1983},
  doi     = {10.1007/BF01206023}
}

@article{Marchioro1988GlobalVortices,
  author  = {Marchioro, Carlo},
  title   = {Euler Evolution for Singular Initial Data and Vortex Theory: A Global Solution},
  journal = {Communications in Mathematical Physics},
  volume  = {116},
  number  = {1},
  pages   = {45--55},
  year    = {1988},
  doi     = {10.1007/BF01239024}
}

@article{MarchioroPulvirenti1993Localization,
  author  = {Marchioro, Carlo and Pulvirenti, Mario},
  title   = {Vortices and Localization in {Euler} Flows},
  journal = {Communications in Mathematical Physics},
  volume  = {154},
  number  = {1},
  pages   = {49--61},
  year    = {1993},
  doi     = {10.1007/BF02096831}
}

@article {MR4628607,
    AUTHOR = {Ionescu, Alexandru D. and Jia, Hao},
     TITLE = {Non-linear inviscid damping near monotonic shear flows},
   JOURNAL = {Acta Math.},
  FJOURNAL = {Acta Mathematica},
    VOLUME = {230},
      YEAR = {2023},
    NUMBER = {2},
     PAGES = {321--399},
      ISSN = {0001-5962,1871-2509},
   MRCLASS = {76E30 (35Q35 76B03)},
  MRNUMBER = {4628607},
MRREVIEWER = {Yanguang\ (Charles)\ Li},
       DOI = {10.4310/acta.2023.v230.n2.a2},
       URL = {https://doi.org/10.4310/acta.2023.v230.n2.a2},
}

@article{Taylor54,
    author = {Taylor, Geoffrey Ingram},
    title = {The dispersion of matter in turbulent flow through a pipe},
    journal = {Proceedings of the Royal Society of London. A. Mathematical and Physical Sciences},
    volume = {223},
    number = {1155},
    pages = {446-468},
    year = {1954},
    month = {05}, 
    issn = {0080-4630},
    doi = {10.1098/rspa.1954.0130},
    url = {https://doi.org/10.1098/rspa.1954.0130},
    eprint = {https://royalsocietypublishing.org/rspa/article-pdf/223/1155/446/48253/rspa.1954.0130.pdf},
}

@article {LLZ26,
    AUTHOR = {Li, Hui and Liu, Ning and Zhao, Weiren},
     TITLE = {Stability threshold of the two-dimensional {C}ouette flow in
              the whole plane},
   JOURNAL = {J. Funct. Anal.},
  FJOURNAL = {Journal of Functional Analysis},
    VOLUME = {290},
      YEAR = {2026},
    NUMBER = {4},
     PAGES = {Paper No. 111271, 40},
      ISSN = {0022-1236,1096-0783},
   MRCLASS = {35Q30 (35B35 76D03)},
  MRNUMBER = {4988188},
MRREVIEWER = {Hyunseok\ Kim},
       DOI = {10.1016/j.jfa.2025.111271},
       URL = {https://doi.org/10.1016/j.jfa.2025.111271},
}

@article {AB25,
    AUTHOR = {Arbon, Ryan and Bedrossian, Jacob},
     TITLE = {Quantitative hydrodynamic stability for {C}ouette flow on
              unbounded domains with {N}avier boundary conditions},
   JOURNAL = {Comm. Math. Phys.},
  FJOURNAL = {Communications in Mathematical Physics},
    VOLUME = {406},
      YEAR = {2025},
    NUMBER = {6},
     PAGES = {Paper No. 129, 57},
      ISSN = {0010-3616,1432-0916},
   MRCLASS = {35Q30 (35B35 76E07)},
  MRNUMBER = {4902850},
MRREVIEWER = {Tatsu-Hiko\ Miura},
       DOI = {10.1007/s00220-025-05306-5},
       URL = {https://doi.org/10.1007/s00220-025-05306-5},
}

@article {MR4400903,
    AUTHOR = {Ionescu, Alexandru D. and Jia, Hao},
     TITLE = {Axi-symmetrization near point vortex solutions for the 2{D}
              {E}uler equation},
   JOURNAL = {Comm. Pure Appl. Math.},
  FJOURNAL = {Communications on Pure and Applied Mathematics},
    VOLUME = {75},
      YEAR = {2022},
    NUMBER = {4},
     PAGES = {818--891},
      ISSN = {0010-3640,1097-0312},
   MRCLASS = {76D17 (76E30)},
  MRNUMBER = {4400903},
MRREVIEWER = {Takashi\ Sakajo},
       DOI = {10.1002/cpa.21974},
       URL = {https://doi.org/10.1002/cpa.21974},
}

@article {MR4740211,
    AUTHOR = {Masmoudi, Nader and Zhao, Weiren},
     TITLE = {Nonlinear inviscid damping for a class of monotone shear flows
              in a finite channel},
   JOURNAL = {Ann. of Math. (2)},
  FJOURNAL = {Annals of Mathematics. Second Series},
    VOLUME = {199},
      YEAR = {2024},
    NUMBER = {3},
     PAGES = {1093--1175},
      ISSN = {0003-486X,1939-8980},
   MRCLASS = {35Q31 (76E05)},
  MRNUMBER = {4740211},
       DOI = {10.4007/annals.2024.199.3.3},
       URL = {https://doi.org/10.4007/annals.2024.199.3.3},
}

@article {MR4630602,
    AUTHOR = {Deng, Yu and Masmoudi, Nader},
     TITLE = {Long-time instability of the {C}ouette flow in low {G}evrey
              spaces},
   JOURNAL = {Comm. Pure Appl. Math.},
  FJOURNAL = {Communications on Pure and Applied Mathematics},
    VOLUME = {76},
      YEAR = {2023},
    NUMBER = {10},
     PAGES = {2804--2887},
      ISSN = {0010-3640,1097-0312},
   MRCLASS = {76E30},
  MRNUMBER = {4630602},
       DOI = {10.1002/cpa.22092},
       URL = {https://doi.org/10.1002/cpa.22092},
}

@article {CastroLear23Travellingwaves,
    AUTHOR = {Castro, \'Angel and Lear, Daniel},
     TITLE = {Traveling waves near {C}ouette flow for the 2{D} {E}uler
              equation},
   JOURNAL = {Comm. Math. Phys.},
  FJOURNAL = {Communications in Mathematical Physics},
    VOLUME = {400},
      YEAR = {2023},
    NUMBER = {3},
     PAGES = {2005--2079},
      ISSN = {0010-3616,1432-0916},
   MRCLASS = {76B03 (35C07 35Q31 76E05)},
  MRNUMBER = {4595614},
MRREVIEWER = {Delyan\ Zhelyazov},
       DOI = {10.1007/s00220-023-04636-6},
       URL = {https://doi.org/10.1007/s00220-023-04636-6},
}

@article {MR4076093,
    AUTHOR = {Ionescu, Alexandru D. and Jia, Hao},
     TITLE = {Inviscid damping near the {C}ouette flow in a channel},
   JOURNAL = {Comm. Math. Phys.},
  FJOURNAL = {Communications in Mathematical Physics},
    VOLUME = {374},
      YEAR = {2020},
    NUMBER = {3},
     PAGES = {2015--2096},
      ISSN = {0010-3616,1432-0916},
   MRCLASS = {35Q31 (35B35 76B99)},
  MRNUMBER = {4076093},
MRREVIEWER = {Ionu\c t\ Munteanu},
       DOI = {10.1007/s00220-019-03550-0},
       URL = {https://doi.org/10.1007/s00220-019-03550-0},
}

@article {MR2796139,
    AUTHOR = {Lin, Zhiwu and Zeng, Chongchun},
     TITLE = {Inviscid dynamical structures near {C}ouette flow},
   JOURNAL = {Arch. Ration. Mech. Anal.},
  FJOURNAL = {Archive for Rational Mechanics and Analysis},
    VOLUME = {200},
      YEAR = {2011},
    NUMBER = {3},
     PAGES = {1075--1097},
      ISSN = {0003-9527,1432-0673},
   MRCLASS = {76B03 (35L60 35Q35)},
  MRNUMBER = {2796139},
MRREVIEWER = {David\ M.\ Ambrose},
       DOI = {10.1007/s00205-010-0384-9},
       URL = {https://doi.org/10.1007/s00205-010-0384-9},
}

@article{Rayleigh1879,
author = {Rayleigh, Lord},
title = {On the Stability, or Instability, of certain Fluid Motions},
journal = {Proceedings of the London Mathematical Society},
volume = {s1-11},
number = {1},
pages = {57-72},
doi = {https://doi.org/10.1112/plms/s1-11.1.57},
url = {https://londmathsoc.onlinelibrary.wiley.com/doi/abs/10.1112/plms/s1-11.1.57},
eprint = {https://londmathsoc.onlinelibrary.wiley.com/doi/pdf/10.1112/plms/s1-11.1.57},
year = {1879}
}

@article{Kelvin1887,
author = {Kelvin, Lord},
 journal = {Philosophical Magazine},
 pages = {188--196},
 publisher = {Royal Irish Academy},
 title = {Stability of fluid motion: rectilinear motion of viscous fluid between two parallel plates},
 urldate = {2025-11-11},
 volume = {24},
NUMBER = {5},
 year = {1887}
}

@article{Orr1907,
 ISSN = {00358975},
 URL = {http://www.jstor.org/stable/20490591},
 author = {William M'F. Orr},
 journal = {Proceedings of the Royal Irish Academy. Section A: Mathematical and Physical Sciences},
 pages = {69--138},
 publisher = {Royal Irish Academy},
 title = {The Stability or Instability of the Steady Motions of a Perfect Liquid and of a Viscous Liquid. {P}art {II}: A Viscous Liquid},
 urldate = {2025-11-11},
 volume = {27},
 year = {1907}
}

@article {MR4302767,
    AUTHOR = {Deng, Yu and Zillinger, Christian},
     TITLE = {Echo chains as a linear mechanism: norm inflation, modified
              exponents and asymptotics},
   JOURNAL = {Arch. Ration. Mech. Anal.},
  FJOURNAL = {Archive for Rational Mechanics and Analysis},
    VOLUME = {242},
      YEAR = {2021},
    NUMBER = {1},
     PAGES = {643--700},
      ISSN = {0003-9527,1432-0673},
   MRCLASS = {76B03 (35Q31)},
  MRNUMBER = {4302767},
MRREVIEWER = {Francesco\ Fanelli},
       DOI = {10.1007/s00205-021-01697-6},
       URL = {https://doi.org/10.1007/s00205-021-01697-6},
}

@article {MR988302,
    AUTHOR = {Dritschel, David G.},
     TITLE = {The repeated filamentation of two-dimensional vorticity
              interfaces},
   JOURNAL = {J. Fluid Mech.},
  FJOURNAL = {Journal of Fluid Mechanics},
    VOLUME = {194},
      YEAR = {1988},
     PAGES = {511--547},
      ISSN = {0022-1120},
   MRCLASS = {76E30 (76C05)},
  MRNUMBER = {988302},
       DOI = {10.1017/S0022112088003088},
       URL = {https://doi.org/10.1017/S0022112088003088},
}

@article{2605.19971,
      title={Flexibility and rigidity for the {C}ouette flow in the infinite channel}, 
      author={Dengjun Guo and Xiaoyutao Luo and Guolin Qin},
      year={2026},
      journal={preprint},
      eprint={2605.19971},
      archivePrefix={arXiv},
      primaryClass={math.AP},
      url={https://arxiv.org/abs/2605.19971}, 
}

@incollection {MR4680382,
    AUTHOR = {Ionescu, Alexandru D. and Jia, Hao},
     TITLE = {On the nonlinear stability of shear flows and vortices},
 BOOKTITLE = {I{CM}---{I}nternational {C}ongress of {M}athematicians. {V}ol.
              5. {S}ections 9--11},
     PAGES = {3776--3799},
 PUBLISHER = {EMS Press, Berlin},
      YEAR = {[2023] \copyright 2023},
      ISBN = {978-3-98547-063-1; 978-3-98547-563-6; 978-3-98547-058-7},
   MRCLASS = {35Q31 (35B20 35B40 76E30)},
  MRNUMBER = {4680382},
       DOI = {10.4171/ICM2022/1},
       URL = {https://doi.org/10.4171/ICM2022/1},
}

@article{jyz2025superlineargradientgrowth2d,
      title={Superlinear gradient growth for 2{D} {E}uler equation without boundary}, 
      author={In-Jee Jeong and Yao Yao and Tao Zhou},
      year={2025},
      journal={preprint},
      eprint={2507.15739},
      archivePrefix={arXiv},
      primaryClass={math.AP},
      url={https://arxiv.org/abs/2507.15739}, 
}

@article{2511.21583,
      title={Long time inviscid damping near {C}ouette in {S}obolev spaces}, 
      author={Dengjun Guo and Xiaoyutao Luo},
      year={2025},
      journal={preprint},
      eprint={2511.21583},
      archivePrefix={arXiv},
      primaryClass={math.AP},
      url={https://arxiv.org/abs/2511.21583}, 
}

@article{2603.20065,
      title={Asymptotic stability of shear flows for 2{D} {E}uler equations at {Y}udovich regularity}, 
      author={Dengjun Guo and Xiaoyutao Luo},
      year={2026},
      journal={preprint},
      eprint={2603.20065},
      archivePrefix={arXiv},
      primaryClass={math.AP},
      url={https://arxiv.org/abs/2603.20065}, 
}

@article {Turkington1987,
    AUTHOR = {Turkington, Bruce},
     TITLE = {On the evolution of a concentrated vortex in an ideal fluid},
   JOURNAL = {Arch. Rational Mech. Anal.},
  FJOURNAL = {Archive for Rational Mechanics and Analysis},
    VOLUME = {97},
      YEAR = {1987},
    NUMBER = {1},
     PAGES = {75--87},
      ISSN = {0003-9527},
   MRCLASS = {76C05 (35Q15 58F05)},
  MRNUMBER = {856310},
MRREVIEWER = {Jacob\ Burbea},
       DOI = {10.1007/BF00279847},
       URL = {https://doi.org/10.1007/BF00279847},
}

@article{ChenWeiZhangZhangInhomogeneous,
      title={Nonlinear inviscid damping for 2{D} inhomogeneous incompressible {Euler} equations}, 
      author={Qi Chen and Dongyi Wei and Ping Zhang and Zhifei Zhang},
      year={2026},
      eprint={2303.14858},
      archivePrefix={arXiv},
      primaryClass={math.AP},
      url={https://arxiv.org/abs/2303.14858}, 
      JOURNAL = {JEMS},
      VOLUME    = {to appear},
}

@article {ZhaoInhomogeneousDamping,
    AUTHOR = {Zhao, Weiren},
     TITLE = {Inviscid damping of monotone shear flows for 2{D}
              inhomogeneous {E}uler equation with non-constant density in a
              finite channel},
   JOURNAL = {Ann. PDE},
  FJOURNAL = {Annals of PDE. Journal Dedicated to the Analysis of Problems
              from Physical Sciences},
    VOLUME = {11},
      YEAR = {2025},
    NUMBER = {1},
     PAGES = {Paper No. 8, 152},
      ISSN = {2524-5317,2199-2576},
   MRCLASS = {35Q31 (35B40 76B03)},
  MRNUMBER = {4857962},
MRREVIEWER = {Qin\ Zhao},
       DOI = {10.1007/s40818-025-00197-0},
       URL = {https://doi.org/10.1007/s40818-025-00197-0},
}

\end{document}